\documentclass[11pt,a4paper]{amsart}
 
\usepackage{latexsym,cite, amssymb, amsmath,enumitem,ulem,soul,esint,pifont} 
\usepackage{amsfonts, graphicx,geometry}
\usepackage{graphicx,color}
\makeatletter

\usepackage [ colorlinks, citecolor=blue, anchorcolor=blue, linkcolor=blue ]{ hyperref }
\newcommand{\be}{\begin{equation}}
	\newcommand{\ee}{\end{equation}}
\newcommand{\beq}{\begin{eqnarray}}
	\newcommand{\eeq}{\end{eqnarray}}

\newtheorem{thm}{Theorem}[section]

\newtheorem{lem}{Lemma}[section]
\newtheorem{prop}{Proposition}[section]
\newtheorem{cor}{Corollary}[section]

\theoremstyle{remark}
\newtheorem{rmk}{Remark}[section]
\numberwithin{equation}{section}

 \def \Rn {\mathbb{R}^n}
 
  \newcommand{\R}{\mathbb{R}}

\newcommand{\Sp}{\mathbb{S}}

\def\ss{\mathbb{S}}

\def\C{\mathcal{C}}

\def\<{\langle}
\def\>{\rangle}
\def\div{{\rm div}}

\def\R{{\mathbb R}}

\def\be{\begin{equation}}
	\def\ee{\end{equation}}
\def\bee{\begin{equation*}}
	\def\eee{\end{equation*}}

\allowdisplaybreaks

\begin{document} 
	 \title{Blow-up phenomena for the constant Q/R-curvature equation}  
	\author{Caiyan Li }
	\address[C.L]{School of Mathematical Sciences, Xiamen University, 361005, Xiamen, P.R. China}
	\email{caiyanli@xmu.edu.cn}
	\author{Guofang Wang}
	\address[G.W]{Mathematisches Institut, Universit\"at Freiburg, Ernst-Zermelo-Str.1, 79104, Freiburg, Germany }
	\email{guofang.wang@math.uni-freiburg.de}
	  \author{Wei Wei}
	 \address[W.W]{School of Mathematics, Nanjing University, 210093, Nanjing, P.R. China}
	 \email{wei\_wei@nju.edu.cn} 
	\date{\today } 
\begin{abstract}	
Let $n\ge 25$ be an integer.
In this paper, we construct a Riemannian metric $g_{0}$ on $\mathbb{S}^n$, smooth for $n\geq 26$ and of class $C^9$ but not $C^{10}$ for $n=25$, with the property that the set of metrics in the conformal class of $g_{0}$ having positive scalar curvature and positive constant quotient $Q/R$ is non-compact. Equivalently, we construct families of solutions exhibiting blow-up behavior
 for the following equation  
\begin{align*} 
	P _{g_{0}}u- \frac{ (n+2 )(n-4 )}{4}  u^{  \frac{2}{n-4}}    L_{g_{0}}u^{  \frac{n-2}{n-4}} =0, \quad u>0\quad\text{on} \  \mathbb{S}^{n},
\end{align*}	
where  $P _{g_{0}}$ is the Paneitz operator and $ L_{g_{0}}=-\Delta_{g_{0}}  +\frac{n-2}{4(n-1 )}R_{g_{0}} $ is the conformal Laplacian of $ g_{0}$.
	\end{abstract}

	\maketitle 
	  	\tableofcontents
\section{Introduction}
Let $(M^{n},g_{0})$ be a compact Riemannian manifold of dimension $n\ge 3$. The problem of determining whether the background metric $g_{0}$ can be conformally transformed into a metric $g$ with a prescribed geometric property plays an important role in differential geometry and geometric analysis. The most famous example is the Yamabe problem, which asks whether any conformal class admits a metric of constant scalar curvature; this is equivalent to solving the following Yamabe equation.
 \begin{align*} 
		-\frac{4(n-1 )} {n-2} \Delta_{g_0} u+ R_{g_0} u
	= u^{\frac{n+2}{n-2}}, \quad u>0\quad\text{on } M.
	\end{align*}	
where $R_{g_0}$ is the scalar curvature of $g_0$. The affirmative answer was given by Yamabe\cite{Yamabe1960}, Trudinger\cite{Trudinger1968}, Aubin\cite{Aubin1976}, and Schoen\cite{Schoen1984}.

One of the central developments in the study of the Yamabe equation is the so-called Compactness Conjecture, which asserts that the set of all solutions to the Yamabe equation is compact, unless the underlying manifold $(M,g_0)$ is conformally equivalent to the standard sphere $\mathbb{S}^n$. This conjecture reflects a fundamental dichotomy between compactness and blow-up phenomena, and has played a crucial role in understanding the global structure of the solution space.
The conjecture was first established by Schoen \cite{Schoen1989,Schoen1991} in the locally conformally flat case. In the general setting, a complete compactness theory has been developed in low dimensions. More precisely, compactness was proved in dimension $3$ by Li-Zhu \cite{LiZhu1999}, and in dimensions $4$ and $5$ by Druet \cite{Druet2005}. These results rely on delicate blow-up analysis and refined estimates for the scalar curvature equation.

In higher dimensions, the situation becomes substantially more involved. For $n \geq 6$, compactness can be obtained under the combined input of the Positive Mass Theorem and the resolution of the Weyl Vanishing Conjecture. The latter asserts that, at any blow-up point, the Weyl tensor must vanish to order strictly greater than $\frac{n-6}{2}$. This condition imposes strong geometric restrictions on the underlying manifold and prevents the formation of certain singularities.
The Weyl Vanishing Conjecture has been verified in several important cases. In dimensions $6$ and $7$, it was proved independently by Marques \cite{Marques2005} and Li-Zhang \cite{LiZhang2005}. Subsequently, Khuri-Marques-Schoen \cite{KhuriMarquesSchoen2009} extended the result to all dimensions $n \leq 24$, thereby establishing compactness in this entire range.

On the other hand, a striking contrast emerges in higher dimensions. Marques \cite{Marques2009} constructed smooth counterexamples to the Weyl Vanishing Conjecture in dimensions $n \geq 25$, which in turn led to the failure of compactness. 

The set of constant scalar curvature metrics within a given conformal class may fail to be compact. This non-compactness phenomenon was first observed in the work of Ambrosetti-Malchiodi \cite{AmbrosettiMalchiodi1999} for non-smooth background metrics, and subsequently by Berti-Malchiodi \cite{BertiMalchiodi2001}. 
A major breakthrough was achieved by Brendle \cite{Brendle2008}, who showed that in dimensions $n \geq 52$, the set of solutions to the Yamabe equation can be non-compact even when the background metric $g_0$ is smooth. This result was later extended by Brendle-Marques \cite{BrendleMarques2009} to all dimensions $25 \leq n \leq 51$. Their constructions are based on a delicate gluing procedure, relying on the introduction of suitably chosen local model metrics. In \cite{Brendle2008}, a linear function is used in the construction, while in \cite{BrendleMarques2009} a cubic polynomial is employed to capture the higher-order interactions. The resulting blowing-up sequences exhibit a single bubble with finite volume.
More recently, Gong-Li \cite{GongLi2025} constructed sequences of conformal metrics of constant scalar curvature whose volumes become unbounded in dimensions $n \geq 25$, revealing a new type of non-compactness behavior beyond the single-bubble phenomenon.

Non-compactness in high dimensions is not restricted to the classical Yamabe problem. Similar phenomena have been observed in several related geometric equations, including the $Q$-curvature Yamabe problem \cite{WeiZhao2013}, the fractional Yamabe problem \cite{KimMussoWei}, and the boundary Yamabe problem \cite{Almaraz2011,DisconziKhuri2017,Wang2018,GhimentiMichelettiPistoia2019,ChenWu2020,GhimentiMicheletti2022,HoShin2025}. These results highlight the ubiquity of non-compactness in conformally invariant geometric problems and underscore the intricate role played by the dimension and the structure of the underlying operators.

In this paper, we further study this phenomenon for a new Yamabe-type problem involving the quotient of the $Q$-curvature and the scalar curvature, proposed by Ge-Wang-Wei \cite{GeWangWei2025-2}.
The problem is to find a metric in the conformal class $g=u^{\frac{4}{n-4}}g_0$ ($u>0$) with constant $Q/R$-curvature $\frac{n^2-4}{8(n-1)}$. This is equivalent to solving $Q_g/R_g=\frac{n^2-4}{8(n-1)}$, that is,
\begin{align} \label{eq-intro} 
	P _{g_{0}}u- \frac{ (n+2 )(n-4 )}{4}  u^{  \frac{2}{n-4}}    L_{g_{0}}u^{  \frac{n-2}{n-4}} =0, \quad u>0\quad  \text{on} \ M,
\end{align} 
where the conformal Laplacian $
L_{g _{0}}$ is defined by
 \begin{equation} \label{conf lap} 
	L_{g _{0}}:=-\Delta_{g_{0} }  +\frac{n-2}{4(n-1 )}R_{g_{0} },  
\end{equation} 
and $P _{g_{0}}$ is the fourth-order Paneitz operator introduced in \cite{Paneitz2008} and defined by 
\begin{equation}\label{p operator}  
			P _{g_{0}} : = \Delta_{g_{0}}^{2}  -{\div} _{g_{0}} \big( (a_{n}R_{g_0}g_0+b_{n}\operatorname{Ric}_{g_{0}} )\nabla_{g_0} \, \big)+	\frac{n-4}{2} Q_{g_0 },
\end{equation}   
 with
	\begin{align*}  
		a_{n}=\frac { (n-2)^{2}+4}{2(n-1)(n-2)},\quad b_{n}=-\frac{4}{n-2}.
	\end{align*}  
Here $Q_{g_0 }$ is the $Q$-curvature, defined by
	\begin{align*}  
		Q _{g _{0}}   := -\frac{ 1}{2(n-1)}\Delta_{g_{0}} R_{g_{0}}+\frac{n^{3}-4n^{2}+16n-16}{8(n-1)^{2}(n-2)^{2}}R_{g_{0}}^{2}-\frac{2}{(n-2)^{2}}|\operatorname{Ric}_{g_{0}}|^{2},
	\end{align*}  
and $\operatorname{Ric}_{g_{0}} $
is the Ricci tensor. 

As in the Yamabe problem for $Q$-curvature, the solutions of equation \eqref{eq-intro} correspond  to the critical points of a functional; in our case
the \textit{$Q/R$-curvature Yamabe functional} is 
\begin{align}  \label{energy}
	\mathcal{I}_{4 ,2}(M,g):=   \frac{\int_{M}\frac{n-4}{2} Q_{g} dv_{g}}{(\int_{M}R_{g} dv_{g})^{\frac{n-4}{n-2}}}\overset{g=u^{\frac{4}{n-4}}g_0}{=} \frac{\int_{M}uP_{g_0}u dv_{g_0}}{\big(\int_{M} \frac{4(n-1)}{n-2}u^{  \frac{n-2}{n-4}}    L_{g_{0}}u^{  \frac{n-2}{n-4}} dv_{g_{0}}\big)^{\frac{n-4}{n-2}}}.
\end{align} 
Here $P_{g_0}$ and $L_{g_{0}}$ are defined as in \eqref{p operator}  and  \eqref{conf lap}, respectively. 
Moreover, the \textit{$Q/R$-Yamabe constant} is  
\begin{align*}   
	Y_{4,2}(M,[g_{0}]):=\inf _{R_{g}>0,Q_{g}\geq0}\frac{\int_{M}\frac{n-4}{2}Q_{g} dv_{g}}{(\int_{M}R_{g} dv_{g})^{\frac{n-4}{n-2}}}.
\end{align*}
This Yamabe constant is well-defined. In \cite{GeWangWei2025-2}, an optimal Sobolev inequality involving the $Q/R$-curvature on $\mathbb{S}^{n}$ with positive scalar curvature metrics was established, and the corresponding Yamabe problem was solved under the assumption that the background metric has positive scalar curvature and semi-positive $Q$-curvature.

It is natural to ask about the compactness of solutions to this new Yamabe problem. The main objective of this paper is to
demonstrate the non-compactness of the solution set for the constant $Q/R$-curvature problem in large dimensions by constructing a one-bubble blow-up sequence. More precisely, we prove the following theorem.
\begin{thm}\label{main thm}
	Suppose $n\geq 25$. Then there exists a Riemannian metric $g_0$ on $\mathbb{S}^{n}$, with
    \[
    g_0\in
    \begin{cases}
    C^{\infty}, & n\geq 26,\\
    C^9\setminus C^{10}, & n=25,
    \end{cases}
    \]
    and a sequence of positive functions $u_k\in C^4(\mathbb{S}^n)$ $(k\in\mathbb{N})$, which are smooth when $n\geq 26$, with the following properties:
	\begin{enumerate}[label=\textup{(\roman*)}]
		\item $g_0 $ is not conformally flat,
		\item $u_{k}$ is a positive solution of the $Q/R$-curvature equation \eqref{main Q:R=c  equ} for all $k\in \mathbb{N}$ with positive scalar curvature,
        \begin{align} \label{main Q:R=c  equ} 
	  P _{g_{0}}u- \frac{ (n+2 )(n-4 )}{4}  u^{  \frac{2}{n-4}}    L_{g_{0}}u^{  \frac{n-2}{n-4}} =0, \quad u>0\quad \ \text{on}\   \mathbb{S}^n.
	\end{align}	
		\item  $\mathcal{I}_{4 ,2}(\mathbb{S}^{n},u_{k}^{\frac{4}{n-4}}g_{0})<Y_{4 ,2}(\mathbb{S}^{n})$ for all $k\in \mathbb{N}$, and  $	\mathcal{I}_{4 ,2}(\mathbb{S}^{n},u_{k}^{\frac{4}{n-4}}g_{0})\rightarrow  Y_{4 ,2}(\mathbb{S}^{n})$ as $k\rightarrow\infty$,
		\item  $\sup_{\mathbb{S}^{n}}u_{k}\rightarrow\infty$, as $k\rightarrow\infty$.
	\end{enumerate}
	Here $Y_{4 ,2}(\mathbb{S}^{n})$ denotes the $Q/R$ constant of the round metric on $ \mathbb{S}^{n}$. 
\end{thm}

\begin{rmk}[Optimal regularity in dimension $25$]\label{rmk:optimal-regularity}
The regularity $g_0\in C^9\setminus C^{10}$ in dimension $n=25$ in Theorem~\ref{main thm} seems optimal with respect to the principal Pohozaev compactness obstruction. Indeed, the principal Pohozaev form for the constant $Q/R$ equation is coercive on the classical metric jet through order ten. Thus $C^9$ is the largest integer regularity below the threshold at which this ten-jet mechanism becomes available, and our construction lies precisely below that threshold. For $n\geq 26$, the corresponding form already admits a negative direction in the polynomial jet space, consistently with our smooth construction.
\end{rmk}

J. Wei-Zhao \cite{WeiZhao2013} established the non-compactness of the set of positive solutions to the prescribed $Q$-curvature equation in dimensions $n\ge 25$. Their result can be viewed as a higher-order analogue of the classical non-compactness phenomenon in the Yamabe problem mentioned above. However, in contrast to the scalar curvature case---where a cubic polynomial is sufficient to construct blowing-up sequences---the $Q$-curvature setting exhibits a more delicate structure. In particular, the higher-order nature of the underlying operator requires a more refined choice of test functions: J. Wei-Zhao employed a quartic polynomial to capture the subtle balance between leading-order terms and error terms in the expansion, which is essential for producing non-compact sequences of solutions.

On the other hand, a complementary compactness theory has been developed in low dimensions.
Gong-Kim-J. Wei \cite{GongKimWei2025} proved $C^4$-compactness for the constant $Q$-curvature equation on smooth closed Riemannian manifolds in dimensions $5\le n\le 24$. More precisely, they established compactness results for both the fourth- and sixth-order constant $Q$-curvature equations under the assumption that the underlying manifold is not conformally diffeomorphic to $\mathbb{S}^n$, and that an associated conformally covariant operator is positive. Their work shows that in this range of dimensions, blow-up phenomena do not occur and all solutions remain uniformly bounded in the $C^4$ topology. The compactness relies on the Liouville theorem for positive solutions of $(-\Delta)^2u=u^{\frac{n+4}{n-4}}$ in Euclidean space. However, the classification of positive entire solutions of our problem,
\begin{eqnarray}
    \label{entire-sol}
    (-\Delta)^2 u =\frac{(n+2)(n-4)}{4}\,u^{\frac{2}{n-4}}\,(-\Delta)u^{\frac{n-2}{n-4}} \qquad \text{in }\Rn,
\end{eqnarray}
remains open; we will further investigate these questions.

Taken together, these results reveal a sharp contrast between the compactness behavior in low and high dimensions: while compactness holds in dimensions $5 \leq n \leq 24$, it fails in dimensions $n \geq 25$. This dichotomy highlights the critical role played by the dimension and the order of the equation, and provides important context and motivation for the present work.   

By employing stereographic projection, we transfer our analysis from the sphere $\mathbb{S}^n$ to the Euclidean space $\mathbb{R}^n$, where the computations are more tractable. 
We are thus led to study equation \eqref{main Q:R=c  equ} in $\mathbb{R}^n$. 

One of the main difficulties in the $Q/R$-curvature setting stems from the lack of a precise understanding of the spectral properties of the linearized operator associated with \eqref{main Q:R=c  equ}. In contrast to the classical scalar curvature case, where the structure of the linearized operator and its spectrum are well understood, the corresponding eigenvalue problem in $\mathbb{R}^n$ for the $Q/R$-curvature equation remains largely unknown. This lack of spectral information creates a serious obstacle in carrying out a Lyapunov-Schmidt reduction. In particular, it becomes essential to obtain refined information on the eigenspaces corresponding to the first two eigenvalues of the linearized operator, as these modes govern the solvability of the linearized equation and the construction of approximate solutions.

A further difficulty arises in the choice of appropriate test functions used to detect non-compactness. In previously studied settings, such as the scalar curvature problem and the $Q$-curvature problem, polynomial test functions of low degree (typically cubic or quartic) are sufficient to capture the relevant asymptotic behavior. However, in our mixed $Q/R$-curvature framework, these classical choices are no longer adequate, especially in the borderline dimension $n=25$. More precisely, while a quartic polynomial $f(s)$ (see \eqref{4th-poly}) allows us to establish the desired non-compactness result in dimensions $n \geq 26$, it fails to produce the necessary sign structure in dimension $n=25$. To overcome this difficulty, we introduce a function $f(s)$ of fractional order $3.5$ (see \eqref{3.5th-poly}), which captures a finer balance between competing terms in the expansion. This choice, although less standard, turns out to be crucial for handling the delicate borderline case and requires a more careful analysis of the associated error terms. The resulting metric is smooth for $n\geq 26$, whereas in dimension $n=25$ the final gluing construction yields a metric of class $C^9$ but not $C^{10}$; see Remark~\ref{rmk:optimal-regularity}.

In addition, from a technical perspective, the nonlinear structure of our equation is significantly more involved than the polynomial-type nonlinearities considered in \cite{Brendle2008, BrendleMarques2009, WeiZhao2013}. In particular, the interaction between the $Q$-curvature and scalar curvature terms leads to more complicated error estimates and a stronger coupling between different modes. This makes it more challenging to establish  the required a priori estimates and to control the remainder terms in the perturbative scheme. To deal with these issues, we adapt and refine the functional setting introduced in \cite{WeiZhao2013} by employing a modified norm (see \eqref{norm}), which is specifically tailored to capture the decay and regularity properties of the solutions in our problem. This modification plays a crucial role in closing the estimates and ensuring the convergence of the scheme.

\medskip
\textit{Organization of the paper.}  
In Section~\ref{Sect. Local model metric}, we first construct a suitable local model metric, which serves as the fundamental building block for the final gluing procedure. This construction captures the essential geometric features of the problem and provides a precise framework for the subsequent analysis. In Section~\ref{Sect. Solutions under the Flat Metric}, we then study equation \eqref{main Q:R=c  equ} in the simpler setting of the flat metric, where explicit approximate solutions can be obtained and analyzed in detail.
In Section~\ref{Sect. Lyapunov-Schmidt-type reduction}, we implement a Lyapunov-Schmidt-type reduction to transform the original problem into a finite-dimensional one. More precisely, by imposing suitable orthogonality conditions, we decompose the problem into a finite-dimensional part and an infinite-dimensional remainder, thereby reducing the existence problem to the search for critical points of an associated variational functional. 
A crucial ingredient in this reduction is the solvability of the corresponding linearized equation. In Subsection~\ref{Sect. linear problem}, we carry out a detailed analysis of the linear problem \eqref{linear equ zeta z}, which is essential for controlling the error terms arising in the nonlinear scheme. This analysis relies in a fundamental way on the spectral properties of the linearized operator, which are studied in Section~\ref{Sect. eigenvalue problem}, where we investigate the associated eigenvalue problem and describe the relevant eigenspaces.
The reduction procedure is completed in Subsection~\ref{Sect. nonlinear problem}, where we solve the nonlinear equation \eqref{main Q:R=c  tilde g equ} under the same orthogonality constraints, thereby reducing the problem to a finite-dimensional nonlinear system \eqref{nonlinear c}. In order to extract precise information about this reduced problem, we further introduce a refined formulation that allows us to derive a sharp expansion of the associated energy functional; this is carried out in Section~\ref{Sect. Energy expansion}.
Finally, in Section~\ref{Sect. pf main thm}, we construct a suitable auxiliary function whose critical points correspond to solutions of the reduced problem. We show that the existence of a strict minimum of this function yields a solution to the original equation, thereby completing the proof of the main theorem. The verification of the existence of such a strict minimum is postponed to Appendix A. For dimensions $n \geq 26$, this is achieved in Section~\ref{Sect. n26} by employing a carefully chosen quartic polynomial, while for the borderline case $n=25$, the argument is carried out in Section~\ref{Sect. n25} using a function of fractional order $3.5$, reflecting the additional difficulties in this critical dimension.

\medskip
\noindent{\it Acknowledgments.}
This work was carried out while C. Li and W. Wei were visiting the University of Freiburg, and the authors would like to thank the Institute of Mathematics for its hospitality. W. Wei was supported by the Alexander von Humboldt Research Fellowship and partially supported by NSFC (Grant No. 12571218, 12271244). C. Li was partially supported by NSFC (Grant No. 12501274).

\section{Preliminaries}\label{Sect. Preliminaries}
\subsection{Local model metric}\label{Sect. Local model metric}

For convenience, we work on Euclidean space $\mathbb{R}^n$ instead of the sphere $\mathbb{S}^n$. Let $g_{0}$ be a Riemannian metric that coincides with the Euclidean metric outside the unit ball $B_1$. We assume that $g_0$ is smooth for $n\geq 26$ and of class $C^9$ for $n=25$; the final metric constructed in dimension $25$ belongs to $C^9\setminus C^{10}$. Throughout this work, we take $g_0=e^h$, where $h$ is a trace-free matrix; in particular, $\det g_{0}(x)=1$ for all $x\in\mathbb{R}^{n}$, so the volume form induced by $g_{0}$ agrees with the Euclidean one.

Fix a number $\tau\in\mathbb{R}$ depending only on the dimension $n$, and let 
$W:\mathbb{R}^{n}\times\mathbb{R}^{n}\times\mathbb{R}^{n}\times\mathbb{R}^{n}\to\mathbb{R}$ 
be a multilinear form. Assume that $W_{ijkl}$ satisfies all the algebraic properties of the Weyl tensor and that some components of $W$ are nonzero, so that
\begin{align*}  
	\sum_{ i,j,k,l=1}^{n}(W_{ikjl}+W_{iljk})^{2}>0.
\end{align*} 
We set  
\begin{align*}  
	H_{ij}(x)=\sum_{p,q=1}^nW_{ipjq}x_{p}x_{q},
\end{align*} 
and 
\begin{align*}  
	\bar{H}_{ij}(x)=f(|x|^{2})H_{ij}(x),
\end{align*}  
where $f(s)$ will be determined later. It is straightforward to verify that $H_{ij}(x)$ is trace-free, $\sum_{i=1}^{n} x_{i}H_{ik}(x)=0$, and  $\sum_{i=1}^{n} \partial_{i}H_{ik}(x)=0$ for all $x\in\mathbb{R}^{n}$. 

More precisely, we choose a quartic polynomial
\begin{equation}\label{4th-poly}
f(s)=\tau-8126s+1662s^{2}-98s^{3}+ s^{4}, 
\end{equation}
for $n\geq 26$, and the following function
\begin{equation}\label{3.5th-poly}
f(s)=1028(s+1)^{1/2}-10 (s+1)^{3/2}-\frac{121}{6}(s+1)^{5/2} +(s+1)^{7/2} +\tau(s+1)^{-1/2}, 
\end{equation}
for $n=25$, where $\tau$ will be determined later.

Consider a Riemannian metric $g_{0}$ of the form
\[
g_{0}(x)=\exp(h(x)),
\]
where $h(x)$ is a trace-free symmetric two-tensor on $\mathbb{R}^{n}$ satisfying $h(x)=0$ for $|x|\geq 1$, and
\begin{align*}  
 |h(x)|+|\partial h(x)|+|\partial^{2} h(x)|+|\partial^{3} h(x)|+|\partial^{4} h(x)|\leq \alpha
\end{align*} 
for all $x\in\mathbb{R}^{n}$.

Moreover, when $|x|\leq \rho$, we define 
\begin{align*}  
	h_{ij}(x)=\mu \epsilon^{8}f(\epsilon^{-2}|x|^{2})H_{ij}(x),
\end{align*} 
in the case \eqref{4th-poly}, while
\begin{align*}
 h_{ij}(x)=\mu \epsilon^{7}f(\epsilon^{-2}|x|^{2})H_{ij}(x) 
\end{align*} 
in the case \eqref{3.5th-poly}.  

Throughout what follows, the parameters $\epsilon$, $\mu$, and $\rho$ are chosen so that $\mu\le 1$ and $\epsilon\le \rho\le 1$.

\subsection{Solutions under the Flat Metric}\label{Sect. Solutions under the Flat Metric}
For any $\xi\in\mathbb{R}^{n}$ and $\lambda\in\mathbb{R}_{+}$, we define
\begin{align*}
u_{0}(x):=\alpha_{n}\Big(\frac{\lambda}{\lambda^{2}+|x-\xi|^{2}}\Big)^{\frac{n-4}{2}}, \quad \text{where } \alpha_{n}=2^{\frac{n-4}{2}}.
\end{align*}

One can directly verify that $u_{0}$ satisfies
\begin{align}\label{equ tilde u0}  
  \Delta^{2} u_{0}+ \frac{ (n+2 )(n-4 )}{4}  u_{0}^{  \frac{2}{n-4}}    \Delta u_{0}^{  \frac{n-2}{n-4}}=0\quad \text{in}\  \mathbb{R}^{n}.
\end{align}

We also note that
\begin{align}\label{u_0 n-2 n-4 u_0  8}
	\Delta {u}_0^{\frac{n-2}{ n-4 }} =-\frac{n(n-2)}{4}{u}_0^{\frac{n+2}{ n-4 }} \quad \text{in}\  \mathbb{R}^{n}.
\end{align}    
 
Setting $\tilde{g}_{0}(y)=g_{0}(\epsilon y)=e^{\tilde{h}(y)}$, we have
\begin{align*}
\tilde{h}_{ij}(y)=\mu\epsilon^{10}f(|y|^{2})H_{ij}(y)=\mu\epsilon^{10}\bar{H}_{ij}(y) \qquad \text{for } n\ge 26,
\end{align*}
and
\begin{align*}
\tilde{h}_{ij}(y)=\mu\epsilon^{9}f(|y|^{2})H_{ij}(y)=\mu\epsilon^{9}\bar{H}_{ij}(y) \qquad \text{for } n=25.
\end{align*}

A straightforward computation yields
\begin{align*}
\frac{Q_{g_{\mathbb{S}^{n}}}}{R_{g_{\mathbb{S}^{n}}}}=\frac{n^2-4}{8(n-1)}.
\end{align*}
It is clear that $u(x)$ satisfies \eqref{main Q:R=c  equ} if and only if the $Q/R$-curvature of the conformal metric $(\epsilon^{\tfrac{n-4}{2}}u(\epsilon y))^{\frac{4}{n-4}}\tilde{g}_0(y)$ is equal to $\frac{n^2-4}{8(n-1)}$.
Equivalently, setting
\[
v(y):=\epsilon^{\tfrac{n-4}{2}}u(\epsilon y),
\]
we see that $v$ satisfies
\begin{align} \label{main Q:R=c  tilde g equ} 
  P _{\tilde{g}_{0}}v- \frac{ (n+2 )(n-4 )}{4}  v^{  \frac{2}{n-4}}    L_{\tilde{g}_{0}}v^{  \frac{n-2}{n-4}} =0 \quad \text{in}\  \mathbb{R}^{n}.
\end{align}	

Introducing the rescaled variables $\xi'=\frac{\xi}{\epsilon}$ and $\lambda'=\frac{\lambda}{\epsilon}$, we define
\begin{align*}
\tilde{u}_{0}(y):=\epsilon^{\tfrac{n-4}{2}}u_{0}(\epsilon y)=\alpha_{n}\Big(\frac{\lambda'}{\lambda'^{2}+|y-\xi'|^{2}}\Big)^{\frac{n-4}{2}},
\end{align*}
and the set
\begin{align*}
\Lambda:=\big\{(\xi',\lambda')\in\mathbb{R}^{n}\times\mathbb{R}: |\xi'|\le 1,\ \tfrac{1}{2}<\lambda'<\tfrac{3}{2}\big\}.
\end{align*}

 \section{A Lyapunov--Schmidt-Type Reduction}\label{Sect. Lyapunov-Schmidt-type reduction}  
 
 In this section, we reduce the solution of \eqref{main Q:R=c  tilde g equ}  to finding a critical point of an energy functional, as stated in Proposition \ref{equvalence between solution and critical point}. To this end, we first solve the linearized equation corresponding to \eqref{main Q:R=c  tilde g equ} in an orthogonal subspace (Proposition \ref{unique sol linear equ zeta}, Subsection \ref{Sect. linear problem}), and subsequently treat the nonlinear equation \eqref{main Q:R=c  tilde g equ} in the same orthogonal subspace (Proposition \ref{nonlinear existence}, Subsection \ref{Sect. nonlinear problem}).
 
\subsection{Linearized problem}\label{Sect. linear problem}     
In this subsection, we study the linearized problem associated with \eqref{main Q:R=c  tilde g equ}. The corresponding linearized operator $\mathcal{L}_{\tilde{g}_{0}}$ can be directly verified to be
\begin{align*}
	\mathcal{L}_{\tilde{g}_{0}  } :=	P _{\tilde{g}_{0}}  
	 - \frac{ n+2}{4}    \Big[   (n-2 ) \tilde{u}_0^{\frac{ 2}{n-4}  }  L_{\tilde{g}_{0}} ( \tilde{u}_0^{\frac{  2}{n-4}  } \cdot)  +2 ( \tilde{u}_0^{\frac{ 6-n}{n-4}  } L_{\tilde{g}_{0}} \tilde{u}_0^{\frac{ n-2}{n-4}  })  \cdot  \Big] ,
\end{align*}	
or, equivalently, in the following form:
\begin{align*}  
	\mathcal{L}_{\tilde{g}_{0}  }	 =    & \Delta_{\tilde{g}_{0}}^{2}   - \sum_{i,j}\partial_{i}\big[ (a_{n}R_{\tilde{g}_{0}}\tilde{g}_{0}^{ij}+b_{n}\operatorname{Ric}_{\tilde{g}_{0}} ^{ij})\partial_{j}  \big] +	\frac{n-4}{2} Q_{\tilde{g}_{0} }  
	\\	 &+ \frac{  n^{2} -4}{4}  \big(    \tilde{u}_0^{\frac{ 4}{n-4}  }  \Delta_{\tilde{g}_{0} } +   \nabla _{\tilde{g}_{0}}\tilde{u}_0^{  \frac{4}{n-4}} \cdot \nabla_{\tilde{g}_{0}}   \big)
		\\&  + \frac{ n+2 }{4}   \big[ 2 \tilde{u}_0^{  \frac{6-n}{n-4}} \Delta_{\tilde{g}_{0}}\tilde{u}_0^{\frac{n-2}{n-4} }+ ( n-2 )  \tilde{u}_0^{\frac{ 2}{n-4}  } \Delta_{\tilde{g}_{0}}\tilde{u}_0^{  \frac{2}{n-4}} -\frac{n(n-2)}{4(n-1)}R_{\tilde{g}_{0}}\tilde{u}_0^{  \frac{4}{n-4}}  \big] ,
\end{align*}	 
where $\operatorname{Ric}_{\tilde{g}_{0}}^{ij}=\tilde{g}_{0}^{is}(\operatorname{Ric}_{\tilde{g}_{0}})_{st}\tilde{g}_{0}^{tj}$.

Define $v(y)=\tilde{u}_{0}(y)+\phi(y)$, where $\phi(y)$ is viewed as a perturbation. With this decomposition, $v(y)$ solves \eqref{main Q:R=c  equ} if and only if $\phi(y)$ satisfies the corresponding linearized equation:
\begin{equation} \label{linear equ}
  \begin{aligned} 
 	  & P _{\tilde{g}_{0}} \phi-\frac{ n+2}{4}    \Big[   (n-2 ) \tilde{u}_0^{\frac{ 2}{n-4}  }  L_{\tilde{g}_{0}} ( \tilde{u}_0^{\frac{  2}{n-4}  } \phi )  +2 ( \tilde{u}_0^{\frac{ 6-n}{n-4}  } L_{\tilde{g}_{0}} \tilde{u}_0^{\frac{ n-2}{n-4}  })\phi     \Big]  
 =  \mathcal{N} (\phi)-\mathcal{R} (y) ,
\end{aligned}
\end{equation} 
where  
\begin{equation}  \label{def N R}
 \left\{\begin{aligned} 
 	\mathcal{N}(\phi)=&  \frac{ (n+2 )(n-4 )}{4} \Big[  (  \tilde{u}_{0}+\phi) ^{  \frac{2}{n-4}}    L_{\tilde{g}_{0}} ( \tilde{u}_{0}+\phi) ^{  \frac{n-2}{n-4}}-   \tilde{u}_{0} ^{  \frac{2}{n-4}}    L_{\tilde{g}_{0}} \tilde{u}_{0} ^{  \frac{n-2}{n-4}} 
 	\\& \quad\quad\quad\quad\quad\quad\quad  -\frac{ n-2}{n-4}     \tilde{u}_0^{\frac{ 2}{n-4}  }  L_{\tilde{g}_{0}} ( \tilde{u}_0^{\frac{  2}{n-4}  } \phi ) - \frac{  2}{n-4}   ( \tilde{u}_0^{\frac{ 6-n}{n-4}  } L_{\tilde{g}_{0}} \tilde{u}_0^{\frac{ n-2}{n-4}  })\phi     \Big]  , 
 	\\\mathcal{R} (y)=&  P _{\tilde{g}_{0}}\tilde{u}_{0}-  \frac{ (n+2 )(n-4 )}{4}  \tilde{u}_{0} ^{  \frac{2}{n-4}}    L_{\tilde{g}_{0}} \tilde{u}_{0} ^{  \frac{n-2}{n-4}}. 
\end{aligned}\right.
\end{equation}

For $n\geq 26$, we define the norms
\begin{align}\label{norm}
	\|\phi \|_{\star}=\sup _{y \in \mathbb{R}^n} \sum_{i=0}^2\left[\frac{1}{\frac{1}{(1+|y-\xi^{\prime}|)^i}\left(\frac{\mu \epsilon^{10}}{\left(1+\left|y-\xi^{\prime}\right|\right)^{n-14}}+\alpha (\frac{\epsilon}{\rho})^{n-4}\right) }+\frac{\left(1+\left|y-\xi^{\prime}\right|\right)^{n-4+i}}{\alpha}\right] |\partial^i \phi(y)| ,
\end{align}  
and
\begin{align*}
	\| \zeta \|_{{\star}{\star}}= & \sup _{y \in \mathbb{R}^n}\bigg[\frac{\chi_{\left\{\left|y-\xi^{\prime}\right| \leq \frac{\rho}{\epsilon}\right\}}\left(1+\left|y-\xi^{\prime}\right|\right)^{n-10}}{\mu \epsilon^{10}}+\frac{\chi_{\left\{\frac{\rho}{\epsilon} \leq\left|y-\xi^{\prime}\right| \leq \frac{1}{\epsilon}\right\}}\left(1+\left|y-\xi^{\prime}\right|\right)^{n-1}}{\alpha \epsilon} \\
	&\quad \quad \quad +\frac{\chi_{\{|y-\xi^{\prime}| \geq \frac{1}{\epsilon}\}} \left(1+\left|y-\xi^{\prime}\right|\right)^{n+\sigma}}{\alpha}\bigg]|\zeta(y)|,
\end{align*}
where $\chi_E$ is the characteristic function of $E$, and $\sigma>0$ is a sufficiently small constant (see \cite[Lemma 5.2]{WeiZhao2013}) introduced to ensure integrability.
We point out that this definition of the norm is slightly different from the one in \cite{WeiZhao2013}.

For the case $n=25$, the norms are defined as above, except that in the definition of $\|\phi\|_{\star}$ the term $\frac{\mu \epsilon^{10}}{\left(1+\left|y-\xi^{\prime}\right|\right)^{n-14+i}}$ is replaced by $\frac{\mu \epsilon^{9}}{\left(1+\left|y-\xi^{\prime}\right|\right)^{n-13+i}}$, and in the definition of $\|\zeta\|_{\star\star}$ the term $\frac{\chi_{\left\{\left|y-\xi^{\prime}\right| \leq \frac{\rho}{\epsilon}\right\}}\left(1+\left|y-\xi^{\prime}\right|\right)^{n-10}}{\mu \epsilon^{10}}$ is replaced by $\frac{\chi_{\left\{\left|y-\xi^{\prime}\right| \leq \frac{\rho}{\epsilon}\right\}}\left(1+\left|y-\xi^{\prime}\right|\right)^{n-9}}{\mu \epsilon^{9}}$.
    
Denote  $ Z_{0}=\frac{\partial \tilde{u}_0}{\partial \lambda'}$, and $Z_{j}=\frac{\partial \tilde{u}_0}{\partial \xi'_{j}}$ for $j=1,\ldots,n$.  More precisely,
$$ \frac{\partial \tilde u_0}{\partial \lambda'}
= \tilde u_0 \cdot
\frac{ (n-4) (|y-\xi'|^2-\lambda'^2 )}{ \lambda'( \lambda'^2+|y-\xi'|^2 )}, \ \text{and}\quad \  \frac{\partial \tilde u_0}{\partial \xi'_i}
= \tilde u_0 \cdot
\frac{(n-4)(y_i-\xi'_i)}{\lambda'^2+|y-\xi'|^2}.$$

Let $\chi\in C^{\infty}(\mathbb{R}^{n})$ be a cut-off function such that $\chi(y)=1$ for $|y-\xi'|\le r_{0}$ and $\chi(y)=0$ for $|y-\xi'|\ge r_{0}+1$, where $r_{0}$ is chosen sufficiently large so that the matrix $\int \chi Z_i Z_j$ is invertible.
Assume $\zeta\in C^{\alpha}(\mathbb{R}^{n})$, and let $c_{i}$ be constants for $i=0,1,\ldots,n$.
We now study the following linear problem:
	\begin{align}\label{linear equ zeta z} 
		\left\{\begin{aligned}   
			&	\mathcal{L}_{\tilde{g}_{0}}  \phi =  \zeta(y)+	 \sum_{i=0}^n  c_{i}\chi Z_{i} \quad \text{in}  \ \mathbb{R}^{n},
			\\  &  \int _{\mathbb{R}^{n}}\phi \chi Z_{i}=0, \quad i=0,1,\ldots,n .   
		\end{aligned}\right. 
\end{align}	 

\begin{prop}\label{unique sol linear equ zeta}
	Suppose that $n\ge 25$. Let $(\xi',\lambda')\in\Lambda$ be fixed, and assume that $\alpha$ is sufficiently small. Then, for all sufficiently small $\epsilon>0$, equation \eqref{linear equ zeta z} admits a unique solution. Moreover,
	\begin{align*}
		\|\phi\|_{\star}\le C\|\zeta\|_{\star\star},
	\end{align*}
where the constant $C>0$ is independent of $\alpha$ and $\epsilon$.
\end{prop}

To establish Proposition \ref{unique sol linear equ zeta}, we begin with the following a priori estimate.

\begin{lem}\label{est linear equ zeta}
	Suppose that $\phi$ is a solution to \eqref{linear equ zeta z}. Then, under the assumptions of Proposition \ref{unique sol linear equ zeta}, we have
	\begin{align*}
		\|\phi\|_{\star}\le C\|\zeta\|_{\star\star}.
	\end{align*}
\end{lem}
\begin{proof}
We proceed by contradiction, following the approach in \cite[Lemma 5.2]{WeiZhao2013}. Suppose that there exist sequences $\epsilon_k\to 0$, $\zeta_k$, and $\phi_k$ satisfying $\|\zeta_k\|_{\star\star}\to 0$ while $\|\phi_k\|_{\star}=1$.

We only present the case $n\ge 26$; the case $n=25$ is entirely analogous.
The proof is divided into two steps.

\medskip
\textit{Step 1.}  We prove for any fixed $R>0$,  
\begin{equation}\label{firstboundforphi}
\lim_{k\rightarrow\infty}\lim_{\mu,\epsilon\rightarrow 0}\frac{\|\phi_k\|_{C^2(B_R(\xi' ))}}{\mu\epsilon^{10}}=0.
\end{equation}
 
Let $\bar\chi\in C^{\infty}(\mathbb{R}^{n})$ be a cut-off function such that $\bar\chi(y)=1$ for $|y-\xi'|\le \frac{\rho}{4\epsilon}$, $\bar\chi(y)=0$ for $|y-\xi'|\ge \frac{\rho}{2\epsilon}$, and $|\nabla^{l}\bar\chi|\le C\big(\frac{\rho}{\epsilon}\big)^{-l}$ for $1\le l\le 4$. Multiplying \eqref{linear equ zeta z} by $\bar\chi Z_j$ and integrating by parts, we have
 \begin{align*} 
 \sum_{i=0}^n c_{i}\int_{\mathbb{R}^{n}}\bar{\chi} Z_{i}Z_{j}=\int_{\mathbb{R}^{n}} 	\mathcal{L}_{\tilde{g}_{0}}( \bar{\chi} Z_{j}  )  \phi_k  -\zeta_k \bar{\chi} Z_{j},\ \ j=0,1,\ldots,n . 
 \end{align*}
Noting that $r_0\le \frac{\rho}{4\epsilon}$, we have \begin{align*} 
	 \int_{\mathbb{R}^{n}}\bar{\chi} Z_{0}Z_{j}=\delta_{0j} \int_{\mathbb{R}^{n}}{\chi} \big( \frac{\partial \tilde{u}_{0}}{\partial \lambda'}\big)^{2},   \ \int_{\mathbb{R}^{n}}\bar{\chi} Z_{i}Z_{j}=\int_{\mathbb{R}^{n}}{\chi} Z_{i}Z_{j}=\delta_{ij} \int_{\mathbb{R}^{n}}{\chi} \big( \frac{\partial \tilde{u}_{0}}{\partial y_{j}}\big)^{2},\   i=1,\ldots,n,
\end{align*}
and  
\begin{align*}  
 \Delta ^{2}Z_{j}   
+  \frac{ n+2}{4}    \Big[   (n-2 ) \tilde{u}_0^{\frac{ 2}{n-4}  }  \Delta ( \tilde{u}_0^{\frac{  2}{n-4}  } Z_{j}  ) +2 ( \tilde{u}_0^{\frac{ 6-n}{n-4}  } \Delta  \tilde{u}_0^{\frac{ n-2}{n-4}  })   Z_{j}  \Big]  =0 ,\ \ j=0,1,\ldots,n .
\end{align*}
 Using the curvature estimates in \cite[Section 4]{WeiZhao2013}, we have 
 \begin{align} \label{c1}
 	& \int_{\mathbb{R}^{n}} 	\mathcal{L}_{\tilde{g}_{0}} (\bar{\chi} Z_{j} )  \phi_k   \nonumber
 	\\=& \int_{\mathbb{R}^{n}}  \big\{ \Delta_{\tilde{g}_{0}}^{2}(\bar{\chi} Z_{j} ) -\bar{\chi} \Delta ^{2}Z_{j}   - \sum_{s,t}\partial_{s} \big[(a_{n}R_{\tilde{g}_{0}}\tilde{g}_{0}^{st}+b_{n}\operatorname{Ric}_{\tilde{g}_{0}} ^{st})\partial_{t}(\bar{\chi} Z_{j} )\big]  \big\}\phi_k\nonumber
 	\\&+ \int_{\mathbb{R}^{n}}\big[\frac{n-4}{2} Q_{\tilde{g}_{0} }   -\frac{n(n^{2}-4)}{16(n-1)}R_{\tilde{g}_{0} }   \tilde{u}_0^{\frac{ 4}{n-4}  }\big]  \bar{\chi}Z_{j} \phi_k  \nonumber
 	\\   &+ \frac{n^{2} -4 }{4}  \int_{\mathbb{R}^{n}} \tilde{u}_0^{\frac{ 2}{n-4}  } \big[ \Delta_{\tilde{g}_{0}} ( \tilde{u}_0^{\frac{ n-2}{n-4}  } \bar{\chi}  Z_{j}  )-\bar{\chi}  \Delta ( \tilde{u}_0^{\frac{ n-2}{n-4}  } Z_{j}  )  \big]  \phi_k \nonumber
 	\\	 &+ \frac{n+2 }{2}     \int_{\mathbb{R}^{n}}    \tilde{u}_0^{\frac{ 6-n}{n-4}  } \big(  \Delta_{\tilde{g}_{0}}  \tilde{u}_0^{\frac{ n-2}{n-4}  }    -\Delta  \tilde{u}_0^{\frac{ n-2}{n-4}  } \big) \bar{\chi}  Z_{j}   \phi_k\nonumber
   	\\=&o(\mu \epsilon^{10})\| \phi_k\| _{\star}.
 \end{align}  

It is also straightforward to obtain 
 \begin{align*} 
 \int_{\mathbb{R}^{n}}  \zeta_k \bar{\chi} Z_{j}\leq  C\mu\epsilon^{10}  \| \zeta_k\| _{\star\star}.
\end{align*} 
Therefore we have  
 \begin{align}\label{c_i est 26} 
|c_{i} | \leq o(\mu \epsilon^{10})\| \phi_k\| _{\star}+C\mu\epsilon^{10}  \| \zeta_k\| _{\star\star}.
\end{align} 
 
We now establish that, for any fixed $R>0$, we have $\|\phi_k\|_{C^2(B_R(\xi'))}=o(\mu\epsilon^{10})$ for large $k$. To see this, by elliptic regularity, there exists a function $\hat\phi$ such that $\frac{\phi_k}{\mu\epsilon^{10}}\to\hat\phi$ in $C^4_{\mathrm{loc}}(\mathbb{R}^n)$ as $\mu,\epsilon\to 0$ and $k\to\infty$,
and
\begin{align*}    
  \Delta ^{2}  \hat{\phi} 
	+  \frac{ n+2}{4}    \big[ 2 ( \tilde{u}_0^{\frac{ 6-n}{n-4}  } \Delta \tilde{u}_0^{\frac{ n-2}{n-4}  })   \hat{\phi} +  (n-2 ) \tilde{u}_0^{\frac{ 2}{n-4}  }  \Delta  ( \tilde{u}_0^{\frac{  2}{n-4}  } 
\hat{\phi} )  \big] =0
 	\ \  \text{in}\   \mathbb{R}^n,
\end{align*}	
that is, by \eqref{u_0 n-2 n-4 u_0  8},
 \begin{align*}    
 	\Delta^{2}  \hat{\phi} + \frac{  n^{2}-4  } {4 }  [ \tilde{u}_0^{\frac{2}{ n-4 }} \Delta (  \tilde{u}_0^{\frac{2}{ n-4 }} \hat{\phi} )-\frac{n}{2} \tilde{u}_0^{  \frac{8}{n-4}} \hat{\phi} ]=0,
 	\ \  \text{in}\   \mathbb{R}^n.
 \end{align*}	   	  
Consequently, using Theorem \ref{eigenvalue thm}, $\hat{\phi}$ can be expressed as a linear combination of the functions $Z_j, j=0,1,\dots,n$. On the other hand, the orthogonality conditions imposed on $\hat{\phi}$ imply that  $ \int_{\mathbb{R}^{n}}  \hat{\phi} \bar{\chi} Z_{j}=0$. Thus, $\hat{\phi} \equiv 0 $, and the claim follows.

\medskip
\textit{Step 2.} We conclude that for large $k$ and sufficiently small $\varepsilon$,
\begin{align*}
\|\phi_k\|_{\star}\le C\|\zeta_k\|_{\star\star}+o(1),
\end{align*}
which contradicts $\|\phi_k\|_{\star}=1$. This completes the proof.
 
 \medskip

To simplify notation, we omit the subscript $k$ in what follows.
By the definition of the Green function $G$ for $\Delta_{g_0}^2$, for $j=0,1,2$, we have
\begin{align*}
\partial_{y}^j	\phi (y)=& \int _{\mathbb{R}^{n}}  \partial_{y}^jG(y,z)  \partial_{s}\big[ (  a_{n}R_{\tilde{g}_{0}}\tilde{g}_{0}^{st}+b_{n}\operatorname{Ric}_{\tilde{g}_{0}}^{st})\partial_{t}\phi\big](z)dz -\frac{n-4}{2} \int _{\mathbb{R}^{n}} \partial_{y_{}}^jG(y,z) Q_{\tilde{g}_{0}} (z )\phi (z )dz
\\&  + \frac{  n(n^{2}-4)  }{16(n-1)}   \int _{\mathbb{R}^{n}} \partial_{y_{}}^jG(y,z)   R_{g_{0}} \tilde{u}_0^{\frac{ 4}{n-4}  } \phi  (z)dz  
	\\&-  \frac{n^{2}-4}{4}   \int _{\mathbb{R}^{n}} \partial_{y_{}}^jG(y,z)    \Big[  \tilde{u}_0^{\frac{ 2}{n-4}  }  \Delta_{\tilde{g}_{0}} ( \tilde{u}_0^{\frac{  2}{n-4}  } \phi)  +\frac{2}{n-2} ( \tilde{u}_0^{\frac{ 6-n}{n-4}  } \Delta_{\tilde{g}_{0}} \tilde{u}_0^{\frac{ n-2}{n-4}  })   \phi \Big]  (z)dz   
\\& +  \int _{\mathbb{R}^{n}} \partial_{y_{}}^jG(y,z)  \zeta(z)dz   
 + \sum_{i}c_{i} \int _{\mathbb{R}^{n}} \partial_{y_{}}^jG(y,z)  \chi Z_{i} (z)dz\\
 =:& I +II- \frac{n^{2}-4}{4} III+IV+V.
\end{align*}   
In the proof of \cite[Lemma 5.2]{WeiZhao2013}, the terms $I$, $II$, $IV$, and $V$ were handled, and it remains to estimate the term $III$.
Next, we estimate $III$.
\begin{align*} 
	III=&\int _{\mathbb{R}^{n}} \partial_{y_{}}^jG(y,z)  \tilde{u}_{0}^{  \frac{4}{n-4}}\Delta_{\tilde{g}_{0}} \phi dz+ \int _{\mathbb{R}^{n}}  \partial_{y_{}}^jG(y,z)  \nabla _{\tilde{g}_{0}}\tilde{u}_{0}^{  \frac{4}{n-4}}\cdot\nabla _{\tilde{g}_{0}}\phi dz
	\\& +\int _{\mathbb{R}^{n}} \partial_{y_{}}^j G(y,z) \big( \tilde{u}_{0}^{  \frac{2}{n-4}}  \Delta_{\tilde{g}_{0}} \tilde{u}_{0}^{  \frac{2}{n-4}}+\frac{2}{n-2}\tilde{u}_{0}^{\frac{6-n}{n-4}} \Delta_{\tilde{g}_{0}} \tilde{u}_{0}^{  \frac{n-2}{n-4}  }\big) \phi dz.  
\end{align*} 
Since $|\partial_{y_{i}}G(y,z)|\leq C[1+O(\alpha)]||y-z|^{3-n}$ and $|\partial_{y_{i}y_{j}}^{2}G(y,z)|\leq C[1+O(\alpha)]||y-z|^{2-n}$, we note that
\begin{align*} 
	&  
 \Big|\int _{\mathbb{R}^{n}}  \partial_{y_{}}^jG(y,z)  \tilde{u}_{0}^{  \frac{4}{n-4}}(  \tilde{g}_{0}^{ij}  \partial_{ij}\phi+\partial_{i}\tilde{g}_{0}^{ij}\partial_{j}\phi  )   dz \Big|
	\\\leq  &C\int _{\mathbb{R}^{n}}  \frac{1}{| y-z | ^{n-4+j} }\frac{1}{ (1+ | z-\xi'|)^{4} }|\partial ^{2}\phi (z)|dz +C\epsilon\int _{B_{\frac{1}{\epsilon}}}  \frac{1}{| y-z | ^{n-4+j} }\frac{1}{ (1+ | z-\xi'|)^{4} }|\partial  \phi (z)|dz , 
\end{align*} 
\begin{align*} 
	&\Big|\int _{\mathbb{R}^{n}}| \partial_{y_{}}^jG(y,z)  \tilde{g}_{0} ^{ij} \partial_{i}\tilde{u}_{0}^{  \frac{4}{n-4}}  \partial_{j}\phi\Big|
	\leq  C\int _{\mathbb{R}^{n}}  \frac{1}{| y-z | ^{n-4+j} }\frac{1}{ (1+ | z-\xi'|)^{5} }|\partial \phi (z)|dz,
\end{align*} 
and
\begin{align*} 
	& \Big|\int _{\mathbb{R}^{n}} |\partial_y^jG(y,z) \big[   \big(\tilde{u}_{0}^{  \frac{2}{n-4}} \tilde{g}_{0} ^{ij} \partial_{ij}\tilde{u}_{0}^{  \frac{2}{n-4}}+\frac{2}{n-2}\tilde{u}_{0}^{\frac{6-n}{n-4}}       \tilde{g}_{0} ^{ij} \partial_{ij}\tilde{u}_{0}^{  \frac{n-2}{n-4}}\big) 
 	\\&\quad\quad\quad\quad\quad\quad+  \big(\tilde{u}_{0}^{  \frac{2}{n-4 }} \partial_{i}\tilde{g}_{0} ^{ij} \partial_{j}\tilde{u}_{0}^{  \frac{2}{n-4}} +\frac{2}{n-2}\tilde{u}_{0}^{\frac{6-n}{n-4}}       \partial_{i}\tilde{g}_{0} ^{ij} \partial_{j}\tilde{u}_{0}^{  \frac{n-2}{n-4}}\big) \big] \phi  dz\Big|
 \\\leq  &C\int _{\mathbb{R}^{n}}  \frac{1}{| y-z | ^{n-4+j} }\frac{1}{ (1+ | z-\xi'|)^{6} }| \phi (z)|dz +C\epsilon\int _{B_{\frac{1}{\epsilon}}}  \frac{1}{| y-z | ^{n-4+j} }\frac{1}{ (1+ | z-\xi'|)^{5} }|  \phi (z)|dz .
\end{align*}  

From the above three formulas, it is clear that the terms in $III$ can be classified as follows:
\begin{align*}
III_{i}^{j}(y)&:=\int_{\mathbb{R}^{n}}\frac{1}{|y-z|^{n-4+j}}\frac{1}{(1+|z-\xi'|)^{6-i}}\,|\partial^{i}\phi(z)|\,dz,
\end{align*}
for $i,j\in\{0,1,2\}$, and
\begin{align*}
\widetilde{III}_{l}^{j}(y)&:=C\epsilon\int_{B_{\frac{1}{\epsilon}}}\frac{1}{|y-z|^{n-4+j}}\frac{1}{(1+|z-\xi'|)^{5-i}}\,|\partial^{i}\phi(z)|\,dz,
\end{align*}
for $l\in\{0,1\}$ and $j\in\{0,1,2\}$.
\medskip

We claim that
\begin{align}\label{III}
|III_{i}^{j}(y)|&\le\big(o(1)+o(1)\|\phi\|_{\star}\big)\bigg(\chi_{\{x,\,|x-\xi'|\le \frac{\rho}{\epsilon}\}}(y)\frac{1}{\left(1+\left|y-\xi^{\prime}\right|\right)^j}\Big(\frac{\mu\epsilon^{10}}{\left(1+\left|y-\xi^{\prime}\right|\right)^{n-14}}+\alpha \big(\tfrac{\epsilon}{\rho}\big)^{n-4}\Big)\nonumber\\
&\qquad\qquad\qquad\quad+\chi_{\{x,\,|x-\xi'|> \frac{\rho}{\epsilon}\}}(y)\frac{\alpha}{\left(1+\left|y-\xi^{\prime}\right|\right)^{n-4+j}}\bigg).
\end{align}

\medskip

In the same manner, we obtain that $\widetilde{III}_{l}^{j}(y)$ obeys the same estimate as $|III_{i}^{j}(y)|$.

For the reader's convenience, we recall two inequalities from \cite[Appendix]{WeiZhao2013} and \cite{LiNi}. For any $0<s,t<n$ such that $s+t<n$,
\begin{equation}\label{ineq1}
    \int_{\mathbb{R}^n} \frac{1}{|y-z|^{n-s}} \frac{1}{\left(1+\left|z-\xi^{\prime}\right|\right)^{n-t}} \mathrm{~d} z \leq C\left(1+\left|y-\xi^{\prime}\right|\right)^{t+s-n}
\end{equation}
and
\begin{equation}\label{ineq2}
    \int_{B_r} \frac{1}{|y-z|^{n-s}} \frac{1}{\left(1+\left|z-\xi^{\prime}\right|\right)^{n-t}} \mathrm{~d} z \leq C r^t\left(1+\left|y-\xi^{\prime}\right|\right)^{s-n}.
\end{equation}

Note that by \eqref{firstboundforphi} and \eqref{ineq2}, we have
\begin{align}
 & \int_{|z-\xi^{\prime}|<R}\frac{1}{|y-z|^{n-4+j}}\frac{1}{(1+|z-\xi'|)^{6-i}}|\partial^{i}\phi|(z)dz\nonumber\\
= & {o(\mu\varepsilon^{10})}\int_{|z-\xi^{\prime}|<R}\frac{1}{|y-z|^{n-4+j}}\frac{1}{(1+|z-\xi'|)^{6-i}}dz\nonumber\\
= & o(\mu\varepsilon^{10})\int_{|z-\xi^{\prime}|<R}\frac{1}{|y-z|^{n-4+j}}\frac{1}{(1+|z-\xi'|)^{n-(n-6+i)}}dz\nonumber\\
\le & o(\mu\varepsilon^{10})R^{n-6+i}(1+|y-\xi'|)^{4-n-j}\nonumber\\
\le & \frac{o(\mu\varepsilon^{10})R^{n-6+i}}{\alpha}\frac{\alpha}{\ensuremath{\left(1+\left|y-\xi^{\prime}\right|\right)^{n-4+j}}},\label{caseR}
\end{align}
and  
\begin{align}
 & \int_{|z-\xi^{\prime}|<R}\frac{1}{|y-z|^{n-4+j}}\frac{1}{(1+|z-\xi'|)^{6-i}}|\partial^{i}\phi|(z)dz\nonumber\\
 =&{o(\mu\varepsilon^{10})}\int_{|z-\xi^{\prime}|<R}\frac{1}{|y-z|^{n-4+j}}\frac{1}{(1+|z-\xi'|)^{n-10}}(1+|z-\xi'|)^{n-16+i}dz\nonumber\\
 \le & C\frac{o(\mu\varepsilon^{10})}{(1+|y-\xi'|)^{n-14+j}}R^{n-16+i}.\label{|y|ball}
\end{align} 

From the definition of $\|\phi\|_{\star}$, we know that 
\begin{align*}
III_{i,R}^{j}(y)
:=&\int_{|z-\xi^{\prime}|\ge R}\frac{1}{|y-z|^{n-4+j}}\frac{1}{(1+|z-\xi'|)^{6-i}}|\partial^{i}\phi|(z)dz\\
\le&\int_{|z-\xi^{\prime}|\ge R}\frac{1}{|y-z|^{n-4+j}}\frac{\|\phi\|_{\star}}{(1+|z-\xi'|)^{6-i}}\Bigg(\frac{1}{\ensuremath{\frac{1}{\frac{1}{(1+|z-\xi^{\prime}|)^i}\left(\frac{\mu \epsilon^{10}}{\left(1+\left|z-\xi^{\prime}\right|\right)^{n-14}}+\alpha (\frac{\epsilon}{\rho})^{n-4}\right) }+\frac{\left(1+\left|z-\xi^{\prime}\right|\right)^{n-4+i}}{\alpha}}}\Bigg)dz.
\end{align*}

Case 1: $\left|y-\xi^{\prime}\right|\le \frac{\rho}{\epsilon}$.

\begin{align} 
III_{i,R}^{j}(y) 
\le &\int_{|z-\xi^{\prime}|\ge R}\frac{1}{|y-z|^{n-4+j}}\frac{\|\phi\|_{\star}}{(1+|z-\xi^{\prime}|)^{6-i}}\left(\frac{\mu\epsilon^{10}}{\left(1+\left|z-\xi^{\prime}\right|\right)^{n-14+i}}+\alpha (\frac{\epsilon}{\rho})^{n-4}\frac{1}{(1+|z-\xi^{\prime}|)^i}\right)dz\nonumber\\ 
= & \int_{R<|z-\xi^{\prime}|}\frac{1}{|y-z|^{n-4+j}}\frac{\|\phi\|_{\star}\mu\epsilon^{10}}{\left(1+\left|z-\xi^{\prime}\right|\right)^{n-8}}dz\nonumber\\
 & +\alpha (\frac{\epsilon}{\rho})^{n-4}\int_{R<|z-\xi^{\prime}|}\frac{1}{|y-z|^{n-4+j}}\frac{\|\phi\|_{\star}}{(1+|z-\xi^{\prime}|)^{6}}dz\nonumber\\ 
\le & C\|\phi\|_{\star}\frac{\mu\epsilon^{10}}{R^{2}}\frac{1}{\left(1+\left|y-\xi^{\prime}\right|\right)^{n-14+j}} 
+C\frac{1}{R}\alpha (\frac{\epsilon}{\rho})^{n-4}\int_{R<|z-\xi^{\prime}|}\frac{1}{|y-z|^{n-4+j}}\frac{\|\phi\|_{\star}}{(1+|z-\xi^{\prime}|)^{5}}dz\nonumber\\
\le & C\|\phi\|_{\star}\frac{1}{R}\Big(  \frac{\mu\epsilon^{10} }{\left(1+\left|y-\xi^{\prime}\right|\right)^{n-14+j}} 
+ \alpha (\frac{\epsilon}{\rho})^{n-4}\frac{1}{(1+|y-\xi^{\prime}|)^{j}}\Big)
\label{case11}
\end{align} 
where the last inequality holds due to \eqref{ineq1}. 
 
Case 2: $\left|y-\xi^{\prime}\right|\ge\frac{\rho}{\epsilon}$.  

We have  
\begin{align}\label{case2Re}
III_{i,R}^{j}(y) & \le\int_{|z-\xi^{\prime}|\ge R}\frac{1}{|y-z|^{n-4+j}}\frac{\|\phi\|_{\star}}{(1+|z-\xi^{\prime}|)^{6-i}}\frac{\alpha}{\ensuremath{\left(1+\left|z-\xi^{\prime}\right|\right)^{n-4+i}}}dz\nonumber\\
 & \le\alpha\|\phi\|_{\star}\int_{|z-\xi^{\prime}|\ge R}\frac{1}{|y-z|^{n-4+j}}\frac{1}{\ensuremath{\left(1+\left|z-\xi^{\prime}\right|\right)^{n+2}}}dz\nonumber\\
 & \le\alpha\|\phi\|_{\star}o(\frac{1}{R})(1+|y-\xi^{\prime}|)^{4-n-j},
\end{align}
where we used $\ensuremath{\int_{\mathbb{R}^{n}}\frac{1}{|x-y|^{n-s}}\frac{1}{(1+|y|)^{t}}\mathrm{~d}y}\le C(1+|x|)^{s-n}$
for $t>n$. 

Collecting all terms \eqref{caseR}-\eqref{case2Re}, we obtain the \textit{claim}.
 
Recalling the estimates in \cite{WeiZhao2013}, we know that $I$, $II$, and $V$ satisfy the same estimates as $III$, namely,
\begin{align}\label{I+II+V}
&|I(y)|+|II(y)|+|V(y)|\nonumber\\
\le&(o(1)+o(1)\|\phi\|_{\star})\bigg(\chi_{\{x,|x-\xi'|\le \frac{\rho}{\epsilon}\}}(y)\frac{1}{\left(1+\left|y-\xi^{\prime}\right|\right)^j}\big(\frac{\mu\epsilon^{10}}{\left(1+\left|y-\xi^{\prime}\right|\right)^{n-14}}+\alpha (\frac{\epsilon}{\rho})^{n-4}\big)\nonumber\\
&+\chi_{\{x,|x-\xi'|> \frac{\rho}{\epsilon}\}}(y)\frac{\alpha}{\left(1+\left|y-\xi^{\prime}\right|\right)^{n-4+j}}\bigg),
\end{align}
by virtue of \begin{align*} 
\frac{\epsilon}{\rho}\leq    \frac{C}{  1+ | y-\xi'| }, \quad \text{for} \quad  | y-\xi'| \leq \frac{\rho}{\epsilon}.
\end{align*}

Moreover, the proof of \cite[Lemma 5.2]{WeiZhao2013} yields the estimate for $IV$:
\begin{align}\label{IV}
 |IV(y)| 
\le &C\|\zeta\|_{\star\star}\bigg(\chi_{\{x,|x-\xi'|\le \frac{\rho}{\epsilon}\}}(y)\frac{1}{\left(1+\left|y-\xi^{\prime}\right|\right)^j}\big(\frac{\mu\epsilon^{10}}{\left(1+\left|y-\xi^{\prime}\right|\right)^{n-14}}+\alpha (\frac{\epsilon}{\rho})^{n-4}\big)\nonumber\\
&\quad\quad\quad\quad +\chi_{\{x,|x-\xi'|> \frac{\rho}{\epsilon}\}}(y)\frac{\alpha}{\left(1+\left|y-\xi^{\prime}\right|\right)^{n-4+j}}\bigg).
\end{align}

To see this, for example, when $|y-\xi^{\prime}|\le \frac{\rho}{\epsilon}$, 
\begin{align*}
   |IV(y)| 
 \le C \|\zeta\|_{\star\star}\bigg(& \int_{|z-\xi^{\prime}|\le \frac{\rho}{\epsilon}} \frac{1}{|y-z|^{n-4+j}} \frac{\mu \epsilon^{10}}{(1+|z-\xi^{\prime}|)^{n-10}}\\
 &+\alpha \epsilon\int_{ \frac{\rho}{\epsilon}\le |z-\xi^{\prime}|\le \frac{1}{\epsilon}} \frac{1}{|y-z|^{n-4+j}}\frac{1}{(1+|z-\xi^{\prime}|)^{n-1}}\\
 &+\int_{|z-\xi^{\prime}|\ge \frac{1}{\epsilon}} \frac{1}{|y-z|^{n-4+j}}\frac{\alpha}{(1+|z-\xi^{\prime}|)^{n+\sigma}}\bigg)\\
 \le  C\|\zeta\|_{\star\star}\bigg(&\frac{\mu \epsilon^{10}}{(1+|y-\xi^{\prime}|)|^{n-14+j}}+\alpha  (\frac{\epsilon}{\rho})^{n-4}\frac{1}{1+|y-\xi^{\prime}|^{ j}}\bigg).
\end{align*}

Substituting the estimates \eqref{III}, \eqref{I+II+V}, and \eqref{IV} into the definition of $\|\phi\|_{\star\star}$, and first choosing $R$ large and then taking $\epsilon$ sufficiently small, we obtain
 \begin{align*} 
\|\phi\|_{\star}\leq C\|\zeta\|_{\star\star}+o(1),
 \end{align*}
which yields a contradiction.
\end{proof}

\begin{proof}[Proof of Proposition \ref{unique sol linear equ zeta}]
Let
\begin{align} \label{set H def}
\mathcal{H}=\big\{ \phi \in  H^{2}(\mathbb{R}^{n}) : \int _{\mathbb{R}^{n}} \chi Z_{j} \phi=0,\ \  j=0,1,\ldots,n\big\}
\end{align} 
be a linear space endowed with the inner product $\langle\phi,\psi\rangle=\int _{\mathbb{R}^{n}} \Delta_{\tilde{g}_{0}}\phi\,\Delta_{\tilde{g}_{0}}\psi$. Then we can reduce \eqref{linear equ zeta z} to finding $\phi\in\mathcal{H}$ such that, for any $\psi\in\mathcal{H}$,
\begin{align*} 
	\langle\phi,\psi\rangle=
	& - \int _{\mathbb{R}^{n}}    (  a_{n}R_{\tilde{g}_{0}}\tilde{g}_{0}^{ij}+b_{n}\operatorname{Ric}_{\tilde{g}_{0}}^{ij})\partial_{i}\phi \partial_{j}\psi+\big[ \frac{n-4}{2}   Q_{\tilde{g}_{0}} (z )-\frac{  n+2  }{2}       \tilde{u}_0^{  \frac{6-n}{n-4}} \Delta_{\tilde{g}_{0}}\tilde{u}_0^{\frac{n-2}{n-4} }  \big]\phi  \psi 
	\\&+ \frac{n^{2} -4 }{4}   \int _{\mathbb{R}^{n}}     \nabla_{\tilde{g}_{0} }(    \tilde{u}_0^{\frac{ 2}{n-4}  }  \phi  )\cdot \nabla_{\tilde{g}_{0} }(    \tilde{u}_0^{\frac{ 2}{n-4}  }  \psi  )  +     \int _{\mathbb{R}^{n}}   \zeta \psi.  
\end{align*}      
By Riesz's representation theorem, we can rewrite this equation in the operator form
\begin{align*}
\phi=K(\phi)+\tilde{\zeta},
\end{align*}
where $\tilde{\zeta}\in\mathcal{H}$ depends linearly on $\zeta$ and $K$ is a compact operator on $\mathcal{H}$.

Fredholm's alternative ensures unique solvability of this problem for any $\tilde{\zeta}$, provided that the homogeneous equation $\phi=K(\phi)$ admits only the trivial solution in $\mathcal{H}$, which corresponds to \eqref{linear equ zeta z} with $\zeta=0$. Lemma \ref{est linear equ zeta} ensures that the only solution in this case is zero, completing the proof.
\end{proof}

By Proposition \ref{unique sol linear equ zeta}, equation \eqref{linear equ zeta z} defines a continuous linear map $\phi=T(\zeta)$ from $L_{\star\star}^{\infty}$, equipped with the norm $\|\cdot\|_{\star\star}$, to $L_{\star}^{\infty}$, equipped with the norm $\|\cdot\|_{\star}$. Moreover, the differentiability of the operator $T$ is established as follows.

\begin{prop}\label{linear equ zeta C1}
 Suppose $(\xi' ,\lambda')\in \Lambda$. Then we have
\begin{align*}  
\|\nabla _{\xi'}T(\zeta) \|_{\star}\leq C  (\|\zeta \|_{\star\star}+\|\partial_{\xi'}\zeta \|_{\star\star}),
\end{align*}	
and 
\begin{align*}\|\nabla _{\lambda'}T(\zeta) \|_{\star}\leq C  (\|\zeta \|_{\star\star}+\|\partial_{\lambda'}\zeta \|_{\star\star}).
\end{align*}
\end{prop} 
 \begin{proof}
 Denote $Z=\partial_{\xi' } \phi $. By equation \eqref{linear equ zeta z},  $Z$ satisfies
  \begin{align*}  
  	\mathcal{L}_{\tilde{g}_{0}}  Z  
  =&  - \frac{n^{2}-4}{4} \Big\{ ( \partial_{\xi' } \tilde{u}_{0}^{  \frac{4}{n-4}})\Delta_{\tilde{g}_{0}}  \phi +   \nabla_{\tilde{g}_{0}}(\partial_{\xi'}   \tilde{u}_{0}^{  \frac{4}{n-4}})\cdot\nabla_{\tilde{g}_{0}}  \phi 
 	\\&\quad\quad\quad\quad + \partial_{\xi' } \big(  \tilde{u}_{0}^{  \frac{2}{n-4}}  \Delta_{\tilde{g}_{0}} \tilde{u}_{0}^{  \frac{2}{n-4}}+\frac{2}{n-2}\tilde{u}_{0}^{\frac{6-n}{n-4}} \Delta_{\tilde{g}_{0}} \tilde{u}_{0}^{  \frac{n-2}{n-4}}-\frac{n}{4(n-1 )}   R_{\tilde{g}_{0}} \tilde{u}_{0}^{  \frac{4}{n-4}}   \big) \phi \Big\} \\&+ \sum_{i=0}^{n}  c_{i}\partial_{\xi' }(\chi Z_{i})+  (\partial_{\xi' } c_{i})\chi Z_{i}.
 \end{align*}	  
 
 On the other hand, by differentiating the orthogonality condition $\int _{\mathbb{R}^{n}} \phi \chi Z_{j}=0 $ with respect to $\xi' $, we obtain
\begin{align*} 
\int_{ \mathbb{R}^{n } } \phi  \partial_{\xi' }(\chi Z_{j})+ \chi ZZ_{j} =0,\ j=0,1,\ldots,n.
 \end{align*}	 
 Select $b_{i}$ satisfying 
 \begin{align*} 
 \sum_{i=0}^{n} b_{i}	\int_{ \mathbb{R}^{n } }   \chi Z_{i} Z_{j} = \int_{ \mathbb{R}^{n } }  \phi  \partial_{\xi' }(\chi Z_{j}),\ j=0,1,\ldots,n.
 \end{align*}	 
 Owing to the diagonal dominance of the system and the uniform boundedness of its coefficients, we conclude that it admits a unique solution and that for any $(\xi',\lambda')\in \Lambda$ and for $n\geq 26$,
 \begin{equation}\label{estimateforb}
     |b_{i}|\leq C\mu \epsilon^{10}\|\phi\|_{\star}, 
 \end{equation} 
and the factor $\epsilon^{10}$ is replaced by  $\epsilon^{9}$ when  $n=25$.

 Set $\eta=Z+\sum_{i=0}^{n}b_{i}\chi Z_{i}$. Then we have 
 \begin{align*} 
 	\left\{\begin{aligned}   
 		&	\mathcal{L}_{\tilde{g}_{0}}  \eta=  \tilde{\zeta} +	\sum_{i=0}^{n}  (\partial_{\xi'} c_{i})\chi Z_{i} \quad \text{in}  \ \mathbb{R}^{n},
 		\\  &  \int _{\mathbb{R}^{n}}\eta \chi Z_{i}=0, \quad i=0,1,\ldots,n, 
 	\end{aligned}\right. 
\end{align*}	  
 where 	   	  
\begin{align*} 
 \tilde{\zeta}= &\sum_{i} b_{i}  	\mathcal{L}_{\tilde{g}_{0}}  (\chi Z_{i}) +\partial_{\xi'}\zeta+c_{i}\partial_{\xi'}(\chi Z_{i})
\\&- \frac{n^{2}-4}{4} \Big\{  ( \partial_{\xi' } \tilde{u}_{0}^{  \frac{4}{n-4}})\Delta_{\tilde{g}_{0}}  \phi +   \nabla_{\tilde{g}_{0}}(\partial_{\xi'}   \tilde{u}_{0}^{  \frac{4}{n-4}})\cdot\nabla_{\tilde{g}_{0}}  \phi 
 \\&\quad\quad\quad\quad + \partial_{\xi' } \big(  \tilde{u}_{0}^{  \frac{2}{n-4}}  \Delta_{\tilde{g}_{0}} \tilde{u}_{0}^{  \frac{2}{n-4}}+\frac{2}{n-2}\tilde{u}_{0}^{\frac{6-n}{n-4}} \Delta_{\tilde{g}_{0}} \tilde{u}_{0}^{  \frac{n-2}{n-4}}-\frac{n}{4(n-1 )}   R_{\tilde{g}_{0}} \tilde{u}_{0}^{  \frac{4}{n-4}}   \big) \phi \Big\} .
\end{align*}	   
By Lemma \ref{est linear equ zeta}, we have $\|\eta\|_{\star}\le C\|\tilde{\zeta}\|_{\star\star}$ and $|\partial c_{i}|\le o(\mu\epsilon^{10})\|\phi\|_{\star}+C\mu\epsilon^{10}\|\partial\zeta\|_{\star\star}$ for $n\ge 26$; the factor $\epsilon^{10}$ is replaced by $\epsilon^{9}$ when $n=25$.

In fact, note that
\begin{align*}
&| ( \partial_{\xi' } \tilde{u}_{0}^{  \frac{4}{n-4}})\Delta_{\tilde{g}_{0}}  \phi +   \nabla_{\tilde{g}_{0}}(\partial_{\xi'}   \tilde{u}_{0}^{  \frac{4}{n-4}})\cdot\nabla_{\tilde{g}_{0}}  \phi | 
\\\leq &C  \big[ \frac{1}{(1+|y-\xi'|)^{ 5} }  |\partial ^{2} \phi|+  \frac{1}{(1+|y-\xi'|)^{ 6} }| \partial\phi |\big],
\end{align*} 
and
\begin{align*} 
&\big| \partial_{\xi' } \big(  \tilde{u}_{0}^{  \frac{2}{n-4}}  \Delta_{\tilde{g}_{0}} \tilde{u}_{0}^{  \frac{2}{n-4}}+\frac{2}{n-2}\tilde{u}_{0}^{\frac{6-n}{n-4}} \Delta_{\tilde{g}_{0}} \tilde{u}_{0}^{  \frac{n-2}{n-4}}-\frac{n}{4(n-1 )}     R_{\tilde{g}_{0}} \tilde{u}_{0}^{  \frac{4}{n-4}} \big)\big||\phi |
 \leq C    \frac{1}{(1+|y-\xi'|)^{ 5} }|\phi |.
\end{align*} 

By the definition of $\|\cdot\|_{\star\star}$, we get
 \begin{align*} 
	& \big\|( \partial_{\xi' } \tilde{u}_{0}^{  \frac{4}{n-4}})\Delta_{\tilde{g}_{0}}  \phi +   \nabla_{\tilde{g}_{0}}(\partial_{\xi'}   \tilde{u}_{0}^{  \frac{4}{n-4}})\cdot\nabla_{\tilde{g}_{0}}  \phi 
	\\& + \partial_{\xi' } \big(  \tilde{u}_{0}^{  \frac{2}{n-4}}  \Delta_{\tilde{g}_{0}} \tilde{u}_{0}^{  \frac{2}{n-4}}+\frac{2}{n-2}\tilde{u}_{0}^{\frac{6-n}{n-4}} \Delta_{\tilde{g}_{0}} \tilde{u}_{0}^{  \frac{n-2}{n-4}}-\frac{n}{4(n-1 )}   R_{\tilde{g}_{0}} \tilde{u}_{0}^{  \frac{4}{n-4}}   \big) \phi \big\| _{\star\star} \leq C \|\phi \|_{\star}.
\end{align*}
Again by the definition of $\|\cdot\|_{\star\star}$ and \eqref{estimateforb}, we directly compute that 
\begin{align*} 
	|b_{i}| \| \mathcal{L}_{\tilde{g}_{0}}  (\chi Z_{i}) \| _{\star\star}&\leq C \|\phi \|_{\star}.
\end{align*} 
 
Similarly, together with \eqref{c_i est 26}, we have
\begin{align*} 
		|c_{i}|\| \partial_{\xi'}(\chi Z_{i})\|_{\star\star}\leq C\| \zeta \|_{\star\star}.
\end{align*}
Therefore $ \| \tilde{\zeta} \|_{\star\star}\leq C(\|\zeta \|_{\star\star}+\|\partial_{\xi'}\zeta \|_{\star\star})$, and so $ \| \eta \|_{\star}\leq C(\|\zeta \|_{\star\star}+\|\partial_{\xi'}\zeta \|_{\star\star})$. 
On the other hand, by Proposition \ref{unique sol linear equ zeta} and \eqref{estimateforb}, we have $ \|b_{i} Z  _{i}\|_{\star}\leq C\|\phi \|_{\star}\leq C\|\zeta \|_{\star\star} $. 
Thus we get $ \|Z \|_{\star}\leq C  (\| \zeta  \|_{\star\star}+\|\partial_{\xi'}\zeta \|_{\star\star})$.  

An analogous argument applies to differentiation with respect to $\lambda'$, thereby completing the proof.
\end{proof}

\subsection{Nonlinear problem}\label{Sect. nonlinear problem}
To solve \eqref{linear equ}, we proceed indirectly by introducing an intermediate step.
We write $v=\tilde{u}_{0}+\phi$, where $\phi$ is defined as the solution of
\begin{align} \label{nonlinear c}
\left\{\begin{aligned}   
	&	\mathcal{L}_{\tilde{g}_{0}}  \phi =   \mathcal{N}(\phi)-\mathcal{R} (y)+	 \sum_{i=0}^n c_{i} \chi Z_{i} \quad \text{in}  \ \mathbb{R}^{n}
			\\  &  \int _{\mathbb{R}^{n}}\phi \chi Z_{i}=0, \quad i=0,1,\ldots,n .   
\end{aligned}\right.
\end{align} 

For notational simplicity, we let $a= \frac{ (n+2 )(n-4 )}{4} $ and $p=\frac{n-2}{n-4}$ throughout this subsection.
\begin{prop}\label{nonlinear existence}
There exists a unique solution to \eqref{nonlinear c} such that
$$\|\phi\|_{\star}\leq \beta,$$
where $\beta$ is a large constant independent of $\alpha$ and $\epsilon$.
\end{prop}
\begin{proof}
Equation \eqref{nonlinear c} can be rewritten as
  $$\phi=T[ \mathcal{N}(\phi)-\mathcal{R}] =:A(\phi).$$
Using Proposition \ref{unique sol linear equ zeta}, 
\begin{equation}\label{A(p)}
\|A(\phi) \|_{\star} \leq C(\|\mathcal{R}\|_{\star\star}+\|\mathcal{N}(\phi)\|_{\star\star}).
\end{equation}
Define the set
  $$\mathcal{S}=\{\phi\in \mathcal{H}\cap L_{\star}^{\infty}(\mathbb{R}^{n}): \|\phi\|_{\star}\leq \beta\},$$ 
where the space $\mathcal{H}$ is given by \eqref{set H def} and $\beta$ is a positive constant determined by \eqref{R rst} and \eqref{N  phi est}. 

Note that by \eqref{equ tilde u0},
\begin{align*}   
	\mathcal{R}(y)=  &  P_{\tilde{g}_{0}}\tilde{u}_{0}-\Delta^{2} \tilde{u}_{0}
	-   \frac{ (n+2 )(n-4 )}{4}\tilde{u}_{0} ^{  \frac{2}{n-4}}  ( L_{\tilde{g}_{0}} \tilde{u}_{0} ^{  \frac{n-2}{n-4}}  +\Delta \tilde{u}_{0} ^{  \frac{n-2}{n-4}} ) 
	\\=&  P_{\tilde{g}_{0}}\tilde{u}_{0}-\Delta^{2} \tilde{u}_{0}
	+ \frac{ (n+2 )(n-4 )}{4}\tilde{u}_{0} ^{  \frac{2}{n-4}}  ( \Delta_{\tilde{g}_{0}} \tilde{u}_{0}^{  \frac{n-2}{n-4}}- \Delta \tilde{u}_{0}  ^{  \frac{n-2}{n-4}}   -\frac{n-2}{4(n-1 )}R_{\tilde{g}_{0}} \tilde{u}_{0}^{  \frac{n-2}{n-4}}   ) .     
\end{align*}	 

We only present the case $n\geq 26$; the case $n=25$ is entirely analogous.

By \cite[Lemma 4.9]{WeiZhao2013},
\begin{equation*} 
	\begin{aligned}
		| P_{\tilde{g}_{0}}\tilde{u}_{0}-\Delta^{2} \tilde{u}_{0}| 
		&\leq  \begin{cases} 
			C \frac{\mu \epsilon^{10}}{ \left(1+\left|y-\xi^{\prime}\right|\right)^{n-10}}\ \  \text{for} \ 
			\left|y \right| \leq \frac{\rho}{\epsilon}  
			\\C \frac{\alpha \epsilon}  { \left(1+\left|y-\xi^{\prime}\right|\right)^{n-1}} \ \  \  \text{for} \   \frac{\rho}{\epsilon} \leq \left|y \right| \leq \frac{1}{\epsilon} 
				\\0, \quad \quad \quad \quad \quad \ \   \ \text{for} \   \left|y \right| \geq \frac{1}{\epsilon}.
		\end{cases}       
	\end{aligned}
\end{equation*}
In a manner analogous to \cite[Proposition 7]{Brendle2008} (or, alternatively, \cite[Proposition 5]{BrendleMarques2009}), we have
\begin{align*}
 	& \big|\tilde{u}_{0} ^{  \frac{2}{n-4}} \big( \Delta_{\tilde{g}_{0}} \tilde{u}_{0}^{  \frac{n-2}{n-4}}- \Delta \tilde{u}_{0}  ^{  \frac{n-2}{n-4}}   -\frac{n-2}{4(n-1 )}R_{\tilde{g}_{0}} \tilde{u}_{0}^{  \frac{n-2}{n-4}}  \big)  \big|
 	\\=&\tilde{u}_{0} ^{  \frac{2}{n-4}} \big| \partial_{i} \big[(\tilde{g}_{0} ^{ij}-\delta_{ij} ) \partial_{j}\tilde{u}_{0}^{  \frac{n-2}{n-4}}\big] -\frac{n-2}{4(n-1 )}R_{\tilde{g}_{0}} \tilde{u}_{0}^{  \frac{n-2}{n-4}}     \big|
 		\\=&\tilde{u}_{0} ^{  \frac{2}{n-4}} \big| \partial_{i} \big[(\tilde{h}_{ij} +O(|\tilde{h}|^{2})) \partial_{j}\tilde{u}_{0}^{  \frac{n-2}{n-4}}\big] -\frac{n-2}{4(n-1 )}R_{\tilde{g}_{0}} \tilde{u}_{0}^{  \frac{n-2}{n-4}}     \big|
 	\\\leq &\left\{\begin{aligned} 
 		&   C \frac{\mu  \epsilon^{10}}{ \left(1+\left|y-\xi^{\prime}\right|\right)^{n-10}}\ \  \text{for} \ 
 		\left|y \right| \leq \frac{\rho}{\epsilon},  
 		\\  & C \frac{\alpha \epsilon}  { \left(1+\left|y-\xi^{\prime}\right|\right)^{n }}\quad \ \  \  \text{for} \   \frac{\rho}{\epsilon} \leq \left|y \right| \leq \frac{1}{\epsilon} 
 			\\&0, \quad \quad \quad \quad \quad \quad\quad \ \  \ \text{for} \   \left|y \right| \geq \frac{1}{\epsilon} ,
 	\end{aligned}\right. 
 \end{align*} 	 
Here we used \cite[Lemma 4.3]{WeiZhao2013} to estimate $|R_{\tilde{g}_{0}}|$, which yields
 	 $ 	|R_{\tilde{g}_{0}}(y)|   \leq C  \mu^{2}  \epsilon^{20}   \left|y \right| ^{20}\leq C  \mu   \epsilon^{10}  \left(1+\left|y-\xi^{\prime}\right|\right)^{ 10}$, for $\left|y \right| \leq \frac{\rho}{\epsilon}$,  
 	 $ 	|R_{\tilde{g}_{0}}(y)|   \leq  C  \alpha \epsilon^{2} $,  for $\frac{\rho}{\epsilon} \leq \left|y \right| \leq \frac{1}{\epsilon} $, and 
 		 $ 	|R_{\tilde{g}_{0}}(y)|=0,$  for $ \left|y \right| \geq \frac{1}{\epsilon} $.
 		 
Putting these facts together, we obtain
\begin{align}\label{R rst}
 \|\mathcal{R}\|_{\star\star}\leq C,
\end{align}
where $C$ is a constant independent of $\alpha$ and $\epsilon$.
Recalling $p=\frac{n-2}{n-4}$, we observe that
\begin{equation} \label{Nphi expension}
	\begin{aligned}
\mathcal{N}(\phi)=& a\Big[  (  \tilde{u}_{0}+\phi) ^{  p-1 }   L_{\tilde{g}_{0}} ( \tilde{u}_{0}+\phi) ^{ p}-   \tilde{u}_{0} ^{  p-1}    L_{\tilde{g}_{0}} \tilde{u}_{0} ^{ p}    -p  \tilde{u}_0^{p-1}  L_{\tilde{g}_{0}} ( \tilde{u}_0^{p-1} \phi ) -(p-1)     \tilde{u}_0^{p-2} L_{\tilde{g}_{0}} \tilde{u}_0^{p} \phi     \Big]  .
\end{aligned}
\end{equation}
By Taylor's theorem, there exists $\theta\in(0,1)$ such that
  \begin{equation}\label{N phi Taylor}
  	\begin{aligned} 
  		\mathcal{N} (\phi)=&a\Big\{(p-1) (p-2)  (  \tilde{u}_{0}+\theta\phi) ^{  p-3 }   L_{\tilde{g}_{0}} (  \tilde{u}_{0}+\theta\phi) ^{p}  \phi^{2}
  		\\&+2p(p-1)  (  \tilde{u}_{0}+\theta\phi) ^{  p-2 }   \phi L_{\tilde{g}_{0}}\big[ ( \tilde{u}_{0}+\theta\phi) ^{  p-1 }   \phi\big]
  		\\&+p(p-1)   (  \tilde{u}_{0}+\theta\phi) ^{  p-1 }   L_{\tilde{g}_{0}} \big[ ( \tilde{u}_{0}+\theta\phi) ^{  p-2 }   \phi^{2}\big]\Big\}.
  	\end{aligned} 
  \end{equation}
Hence, we have
   \begin{equation}\label{N  phi est1}
   	 \begin{aligned} 
 |\mathcal{N} (\phi )|	\leq C\big[ &(1+ |y-\xi'| )^{n-10}|\phi |^{2}+ (1+ |y-\xi'| )^{ n-9}|\phi | |\partial\phi | 
 \\&+  (1+ |y-\xi'| )^{n-8} (|\partial\phi | ^{2}+|\phi | |\partial^{2}\phi | )\big],  
\end{aligned} 
\end{equation}
which implies that
\begin{align}\label{N  phi est}
	\|\mathcal{N} (\phi)\|_{\star\star} \leq C \epsilon^{2-\sigma}\|\phi\|_{ \star}^{2} .
\end{align}   

Consequently, together with \eqref{R rst} and \eqref{A(p)}, we have $A(\phi)\in\mathcal{S}$ for $\phi\in\mathcal{S}$ and $\epsilon$ sufficiently small.
Moreover, we have
$$\|\mathcal{N}(\phi_{1})-\mathcal{N}(\phi_{2})\|_{\star\star}\leq C\epsilon^{2-\sigma}\|\phi_{1}-\phi_{2}\|_{\star}.$$
It follows that
$$\|A(\phi_{1})-A(\phi_{2})\|_{\star}\leq C\|\mathcal{N}(\phi_{1})-\mathcal{N}(\phi_{2})\|_{\star\star}\leq C\epsilon^{2-\sigma}\|\phi_{1}-\phi_{2}\|_{\star}.$$
Hence, by the contraction mapping theorem on $\mathcal{S}$, existence and uniqueness follow.
\end{proof}
 
The following proposition establishes the differentiability of $\phi$ defined in Proposition \ref{nonlinear existence}. This will be used to relate solutions of \eqref{main Q:R=c  tilde g equ} to critical points of a functional; see Proposition \ref{equvalence between solution and critical point}.

\begin{prop}\label{derivative of phi-estimates}
 Let $\phi$ be the solution to \eqref{nonlinear c}. Then we have 
 	 $$ \| \nabla_{(\xi',\lambda')}\phi \|_{ \star}\leq C.$$
 \end{prop}
\begin{proof}
	 We first show the differentiability of  $\phi$. Define a $C^{1}$ map $H:\Lambda\times   \mathcal{S}\times \mathbb{R}^{n+1}\rightarrow L_{\star}^{\infty}(\mathbb{R}^{n})\times \mathbb{R}^{n+1}$ by
\begin{align*}
 \textbf{H}((\xi',\lambda'),\phi, \textbf{c}) =\left( 
 \begin{array}{c} 
   P _{\tilde{g}_{0}}( \tilde{u}_{0}+\phi)-a( \tilde{u}_{0}+\phi)^{  p-1}    L_{\tilde{g}_{0}} (\tilde{u}_{0}+\phi)^{ p} -\sum_{i=0}^{n}c_{i}\chi  Z_{i}
 \\\int_{ \mathbb{R}^{n } }\chi  Z_{0}\phi
 \\\vdots
 \\\int_{ \mathbb{R}^{n } }\chi  Z_{n}\phi
 \end{array} \right)
\end{align*}
where $\textbf{c}=\{c_i\}$.
Then \eqref{nonlinear c} can be rewritten as $\textbf{H}((\xi',\lambda'),\phi,\textbf{c})=0$.

By Proposition \ref{nonlinear existence}, there exists a unique solution $\phi_{(\xi',\lambda')}$ for each fixed $(\xi',\lambda')\in\Lambda$.

To establish the $C^{1}$ regularity via the implicit function theorem, we show that the linear operator $\frac{\partial \textbf{H}((\xi',\lambda'),\phi,\textbf{c})}{\partial (\phi,\textbf{c})}\big|_{((\xi',\lambda'),\phi_{(\xi',\lambda')},\textbf{c}_{(\xi',\lambda')})}$ is invertible provided $\alpha$ and $\epsilon$ are sufficiently small. In fact,
\begin{align*}
\frac{\partial \textbf{H}((\xi',\lambda'),\phi, \textbf{c}) }{\partial (\phi, \textbf{c})}|_{ ((\xi',\lambda'), \phi_{(\xi',\lambda') }, \textbf{c}_{(\xi',\lambda') }  )}  ( \varphi , \textbf{d})=\left( 
	\begin{array}{c} 
		\mathcal{L}_{\tilde{g}_{0},\tilde{u}_{0}+\phi } \varphi  -\sum_{i=0}^{n}d_{i}\chi  Z_{i} 
		\\\int_{ \mathbb{R}^{n } }\chi  Z_{0}\varphi 
		\\\vdots
		\\\int_{ \mathbb{R}^{n } }\chi  Z_{n}\varphi 
	\end{array} \right)
\end{align*}
 where  
 \begin{align*}  
 	\mathcal{L}_{\tilde{g}_{0},\tilde{u}_{0}+\phi _{(\xi',\lambda') }}\varphi	 = P _{\tilde{g}_{0}}\varphi	-a   \Big\{ &  p( \tilde{u}_{0}+\phi_{(\xi',\lambda') })^{p-1 }  L_{\tilde{g}_{0}} \big[  ( \tilde{u}_{0}+\phi_{(\xi',\lambda') })^{p-1} \varphi	\big]
 	\\&+(p-1)(  \tilde{u}_{0}+\phi_{(\xi',\lambda') })^{p-2 } L_{\tilde{g}_{0}}  ( \tilde{u}_{0}+\phi_{(\xi',\lambda') })^{p}\varphi	  \Big\}  ,
 \end{align*}		
and $\textbf{d}=\{d_i\}$. Since $\|\phi\|_{\star}\le C$, we can show that this operator is invertible, similarly to the proof of Proposition \ref{unique sol linear equ zeta}.

Next, we prove the $C^{1}$ estimate. Since $\phi=T[\mathcal{N}(\phi)-\mathcal{R}]$, by applying Propositions \ref{unique sol linear equ zeta} and \ref{linear equ zeta C1} we obtain
$$\|  \partial _{\xi'} \phi \|_{ \star} \leq C\big(\|	\mathcal{R}\|_{ \star\star} +\|	\mathcal{N}(\phi )\|_{ \star\star} + \|\partial_{\xi'}	\mathcal{R}\|_{ \star\star} +\|\partial_{\xi'}	\mathcal{N}(\phi )\|_{ \star\star}\big).$$

Differentiating \eqref{Nphi expension} with respect to $\xi'$ then yields
\begin{align*} 
	a^{-1}	\partial_{\xi'}	\mathcal{N}(\phi)=& (p-1)\big[	(\tilde{u}_{0}+\phi) ^{  p-2}  (	\partial_{\xi'} \tilde{u}_{0}+	\partial_{\xi'}\phi) L_{\tilde{g}_{0}} ( \tilde{u}_{0}+\phi) ^{ p}-\tilde{u}_{0} ^{  p-2}   \partial_{\xi'} \tilde{u}_{0}   L_{\tilde{g}_{0}} \tilde{u}_{0} ^{ p} \big] 
	\\& + p\big[(  \tilde{u}_{0}+\phi) ^{  p-1 }   L_{\tilde{g}_{0}} \big( ( \tilde{u}_{0}+\phi) ^{ p-1}(	\partial_{\xi'} \tilde{u}_{0}+	\partial_{\xi'}\phi) \big)-   \tilde{u}_{0} ^{  p-1}    L_{\tilde{g}_{0}}	 (  \tilde{u}_{0} ^{ p-1} 	\partial_{\xi'} \tilde{u}_{0} ) \big]
	\\&    -p\tilde{u}_0^{p-2}\big[(p-1) 	\partial_{\xi'}   \tilde{u}_0  L_{\tilde{g}_{0}} ( \tilde{u}_0^{p-1} \phi )+ \tilde{u}_0  L_{\tilde{g}_{0}}     \big( (p-1)\tilde{u}_0^{p-2}  \partial_{\xi'}\tilde{u}_0  \phi + \tilde{u}_0^{p-1}  \partial_{\xi'}\phi \big)  \big]
	\\&-(p-1)   \tilde{u}_0^{p-3} \big[\big( 	(p-2 )\partial_{\xi'}\tilde{u}_0 L_{\tilde{g}_{0}} \tilde{u}_0^{p} +p \tilde{u}_0  L_{\tilde{g}_{0}} ( \tilde{u}_0^{p-1} \partial_{\xi'}\tilde{u}_0) \big)\phi + ( \tilde{u}_0  L_{\tilde{g}_{0}}	  \tilde{u}_0^{p}) \partial_{\xi'} \phi \big],
\end{align*}
which can be rearranged as
\begin{align*}
	a^{-1}	\partial_{\xi'}	\mathcal{N}(\phi)=& (p-1)\big[	(\tilde{u}_{0}+\phi) ^{  p-2} L_{\tilde{g}_{0}} ( \tilde{u}_{0}+\phi) ^{ p}-\tilde{u}_{0} ^{  p-2}  L_{\tilde{g}_{0}} \tilde{u}_{0} ^{ p} \big]   \partial_{\xi'} \tilde{u}_{0}  
	\\&  -\big[ (p-1)(p-2) \tilde{u}_{0} ^{  p-3}  \phi	  L_{\tilde{g}_{0}} \tilde{u}_{0} ^{ p} +p(p-1)  \tilde{u}_{0} ^{  p-2}   L_{\tilde{g}_{0}}   (  \tilde{u}_{0} ^{ p-1}  \phi  )\big] \partial_{\xi'} \tilde{u}_{0}
 \\&+p \big[ (\tilde{u}_{0}+\phi) ^{  p-1}   L_{\tilde{g}_{0}}	 \big( ( \tilde{u}_{0}+\phi) ^{ p-1} \partial_{\xi'}\tilde{u}_{0} \big)-  \tilde{u}_{0} ^{  p-1}   L_{\tilde{g}_{0}}	 \big(   \tilde{u}_{0} ^{ p-1} \partial_{\xi'}\tilde{u}_{0} \big) \big]
\\&-p(p-1)  \big[ \tilde{u}_0^{p-2}\phi  L_{\tilde{g}_{0}} (\tilde{u}_0^{p-1}  \partial_{\xi'}\tilde{u}_0 )+  \tilde{u}_0^{p-1} L_{\tilde{g}_{0}} (\tilde{u}_0^{p-2}  \partial_{\xi'}\tilde{u}_0 \phi)\big] 
 \\& + p\big[(  \tilde{u}_{0}+\phi) ^{  p-1 }   L_{\tilde{g}_{0}} \big( ( \tilde{u}_{0}+\phi) ^{ p-1} 	\partial_{\xi'}\phi  \big)-   \tilde{u}_{0} ^{  p-1}    L_{\tilde{g}_{0}}	 (  \tilde{u}_{0} ^{ p-1} 	\partial_{\xi'} \phi ) \big]
 \\& +(p-1)\big[ (\tilde{u}_{0}+\phi) ^{  p-2}   L_{\tilde{g}_{0}} ( \tilde{u}_{0}+\phi) ^{ p}  -   \tilde{u}_{0}  ^{  p-2}   L_{\tilde{g}_{0}}  \tilde{u}_{0} ^{ p}\big]  \partial_{\xi'}\phi  .
\end{align*}
Hence, a straightforward computation shows that
\begin{align*} 
	|	\partial_{\xi'}\mathcal{N} (\phi )|	\leq C\Big\{ &(1+ |y-\xi'| )^{n-11}|\phi |^{2}+(1+ |y-\xi'| )^{n-10}( |\phi | |\partial_{\xi'}\phi |+|\phi | |\partial \phi | ) 
	\\& + (1+ |y-\xi'| )^{ n-9}\big( |\partial_{\xi'}\phi | |\partial \phi | +|\phi | |\partial_{\xi'}\partial \phi |+|\partial\phi | ^{2}+|\phi | |\partial^{2}\phi | )
	\\&+  (1+ |y-\xi'| )^{n-8} (|\partial\phi | |\partial_{\xi'}\partial\phi | +|\partial_{\xi'}\phi | |\partial^{2}\phi | +|\phi | |\partial_{\xi'}\partial^{2}\phi | )\Big\},
 \end{align*}
which yields
\begin{align*}
	&\|\partial_{\xi'}\mathcal{N} (\phi )\|_{ \star\star}\leq C\epsilon^{3-\sigma}\| \phi  \|_{  \star}( \| \phi  \|_{  \star}+\| \partial_{\xi'}\phi  \|_{  \star})\leq C\epsilon^{3-\sigma} (1+\| \partial_{\xi'}\phi  \|_{  \star}),
\end{align*}
using $\|\phi \|_{ \star} \leq C$.

On the other hand, from
\begin{align*}
\partial_{\xi'} \mathcal{R}  =&  P _{\tilde{g}_{0}}(\partial_{\xi'} \tilde{u}_{0})- a\big[ \partial_{\xi'}   \tilde{u}_{0} ^{ p-1}     L_{\tilde{g}_{0}} \tilde{u}_{0} ^{ p} + \tilde{u}_{0} ^{ p-1}    L_{\tilde{g}_{0}}  (  \partial_{\xi'}  \tilde{u}_{0} ^{ p} ) \big],
\end{align*}
we verify, as in the argument for \eqref{R rst}, that
 \begin{align*}  
 	\| \partial_{\xi'} 	\mathcal{R} \|_{  \star \star} \leq C.
 \end{align*}	 
Together with \eqref{R rst} and \eqref{N  phi est}, we derive
\begin{align*}  
	\| \partial_{\xi'}\phi  \|_{  \star} \leq C.
\end{align*}	

An analogous computation gives the corresponding result for differentiation with respect to $\lambda'$.
\end{proof}

Once \eqref{nonlinear c} has been solved, a solution to \eqref{main Q:R=c  tilde g equ} follows, provided that $(\xi',\lambda')$ satisfies
\begin{align*}
 c_{i}(\xi',\lambda')=0, \ \ \text{for}\  \  i=0,1,\ldots,n.
\end{align*}
To proceed, we introduce a variational formulation. Consider the energy functional
\begin{align*}  
	\mathcal{E}_{\tilde{g}_{0}}(v):= \int _{\mathbb{R}^{n}}  v P_{\tilde{g}_{0}}v  - \frac{ (n+2 )(n-4 )^2}{4(n-2 )}  v ^{  \frac{n-2}{n-4}}     L_{\tilde{g}_{0}} v ^{  \frac{n-2}{n-4}}  .
\end{align*}
Associated with this energy, we define a reduced functional on the parameter space by
\begin{align*}
	\mathcal{F}_{\tilde{g}_{0}}(\xi',\lambda'):=\mathcal{E}_{\tilde{g}_{0}}(v_{\xi',\lambda'})=\mathcal{E}_{\tilde{g}_{0}}\big(( \tilde{u}_{0})_{\xi',\lambda'}+\phi_{\xi',\lambda'}\big),
\end{align*}
where $v_{\xi',\lambda'}=( \tilde{u}_{0})_{\xi',\lambda'}+\phi_{\xi',\lambda'}$ and $\phi_{\xi',\lambda'}$ is the solution of \eqref{nonlinear c}.

More precisely,
\begin{align*}
 \mathcal{F}_{\tilde{g}_{0}}(\xi',\lambda')=\int _{\mathbb{R}^{n}}  v_{\xi',\lambda'}P_{\tilde{g}_{0}}v_{\xi',\lambda'} - \frac{ (n+2 )(n-4 )^2}{4(n-2 )}  v_{\xi',\lambda'}^{  \frac{n-2}{n-4}}     L_{\tilde{g}_{0}} v_{\xi',\lambda'}^{  \frac{n-2}{n-4}}.
\end{align*}

The following result shows that solving \eqref{main Q:R=c  tilde g equ} is equivalent to finding a critical point of $\mathcal{F}_{\tilde{g}_{0}}$.
\begin{prop}\label{equvalence between solution and critical point}
	$v_{\xi',\lambda'}$ solves \eqref{main Q:R=c  tilde g equ} if and only if $(\xi',\lambda')$ is a critical point of $\mathcal{F}_{\tilde{g}_{0}}$.
\end{prop}
\begin{proof}
From equation
\begin{align*} 
  P _{\tilde{g}_{0}} (\tilde{u}_{0}+\phi)- \frac{ (n+2 )(n-4 )}{4 }   (\tilde{u}_{0}+\phi)^{  \frac{ 2}{n-4}}L_{\tilde{g}_{0}} (\tilde{u}_{0}+\phi)^{  \frac{n-2}{n-4}} =\sum_{i=0}^{n} c_{i}\chi Z_{i},
\end{align*}	  
a straightforward computation yields
\begin{align*}  
	&\partial_{(\xi',\lambda') }	\mathcal{F}(\xi',\lambda') 
\\=&2\int _{\mathbb{R}^{n}}\big(\partial_{(\xi',\lambda') } \tilde{u}_{0}+\partial_{(\xi',\lambda') } \phi\big) \big( P _{\tilde{g}_{0}} (\tilde{u}_{0}+\phi)- \frac{ (n+2 )(n-4 )}{4 }   (\tilde{u}_{0}+\phi)^{  \frac{ 2}{n-4}}L_{\tilde{g}_{0}} (\tilde{u}_{0}+\phi)^{  \frac{n-2}{n-4}}\big) 
 \\=& 2 \sum_{i=0}^{n} c_{i} \int _{\mathbb{R}^{n}}\chi Z_{i}\big(\partial_{(\xi',\lambda') } \tilde{u}_{0}+\partial_{(\xi',\lambda') } \phi\big).  
\end{align*}   
This proves the necessity.

Now suppose that $\partial_{(\xi',\lambda')}\mathcal{F}_{\tilde{g}_{0}}(\xi',\lambda')=0$. Then
\begin{align}\label{c sys=0}
	\sum_{i=0}^{n} c_{i} \int _{\mathbb{R}^{n}}\chi Z_{i}\big(\partial_{(\xi',\lambda') } \tilde{u}_{0}+\partial_{(\xi',\lambda') } \phi\big)=0.
\end{align}  
Observing that
\begin{align*} 
\left\{ \begin{aligned}  
\partial_{\xi'} \tilde{u}_{0}+\partial_{\xi'} \phi&=Z_{i}+o(1), \quad  i=1,\ldots,n,
\\ \partial_{\lambda'} \tilde{u}_{0}+\partial_{\lambda'} \phi&=Z_{0}+o(1),
\end{aligned} \right.	 
\end{align*} 	
the system  \eqref{c sys=0} is diagonally dominant, which implies $c_{i } (\xi',\lambda') =0$ for $ i=0,1,\ldots,n$.
\end{proof}

\section{Energy expansion}\label{Sect. Energy expansion}   	
In this section, we establish the energy expansion for $\mathcal{F}_{\tilde{g}_{0}}(\xi',\lambda')$, namely Proposition \ref{energy expansion}.
To this end, we need a sharper estimate of $\phi_{\xi',\lambda'}$. Throughout this section, we continue to write $\phi$ for $\phi_{\xi',\lambda'}$. We therefore study the problem
 \begin{align}\label{w equ}
 	\left\{\begin{aligned} 
	&   \mathcal{L} _{\tilde{g}_{0}}w 
  =   -	\mathcal{R}_{1}(y) + \sum_{i=0}^n c_{i}\chi Z_{i}   \quad \text{in}  \ \mathbb{R}^{n},
  		\\  &  \int _{\mathbb{R}^{n}}w \chi Z_{i}=0, \quad i=0,1,\ldots,n,   
  \end{aligned}\right. 
\end{align}	 
where, for  $n\geq 26$,
\begin{equation}\label{R1}   
\begin{aligned}   
		\mathcal{R}_{1}(y)= - \mu \epsilon^{10}\hat{\chi} \Big\{ &2\bar{H} _{ij} \partial_{ijss}  \tilde{u}_{0} +2\partial_{s}\bar{H} _{ij} \partial_{ijs}  \tilde{u}_{0}+\partial_{ss }\bar{H} _{ij} \partial_{ij}  \tilde{u}_{0}
  -\frac{b_{n}}{2}\partial_{ss }\bar{H} _{ij} \partial_{ij}  \tilde{u}_{0}-\frac{b_{n}}{2}\partial_{jss }\bar{H} _{ij} \partial_{i }  \tilde{u}_{0} 
	\\&+\frac{ (n+2 )(n-4 )}{4} \tilde{u}_{0} ^{  \frac{2}{n-4}}   \bar{H}_{ij} \partial_{ij}(\tilde{u}_{0} ^{  \frac{n-2}{n-4}}) \Big\}.   
\end{aligned}
\end{equation}	 
For $n=25$, the factor $\epsilon^{10}$ in \eqref{R1} should be replaced by $\epsilon^{9}$.
Here $\hat{\chi}$ denotes a cut-off function satisfying $\hat{\chi}(y)=1$ for $|y|\le \frac{\rho}{\epsilon}$ and $\hat{\chi}(y)=0$ for $|y|\ge \frac{2\rho}{\epsilon}$. For the reader's convenience, we note that the terms in $\mathcal{R}_{1}$ come from the first lower-order terms in $P_{\tilde{g}_{0}}\tilde{u}_{0}-\Delta^{2}\tilde{u}_{0}$ and in $\frac{ (n+2 )(n-4 )}{4}\tilde{u}_{0} ^{  \frac{2}{n-4}}  ( \Delta_{\tilde{g}_{0}} \tilde{u}_{0}^{  \frac{n-2}{n-4}}- \Delta \tilde{u}_{0}  ^{  \frac{n-2}{n-4}}   -\frac{n-2}{4(n-1 )}R_{\tilde{g}_{0}} \tilde{u}_{0}^{  \frac{n-2}{n-4}}   )$.

Moreover, we have the estimate
 \begin{align}\label{est R1}   
	|\mathcal{R}_{1}(y)|   
 &\leq \left\{\begin{aligned} 
	&C \begin{cases}
	 \frac{\mu \epsilon^{10}}{ \left(1+\left|y-\xi^{\prime}\right|\right)^{n-10}}, & n \geq 26, 
		\\	     \frac{\mu  \epsilon^{9}}{ \left(1+\left|y-\xi^{\prime}\right|\right)^{n-9}}, & n = 25.
	\end{cases}  \quad\quad \text{for} \ 
	\left|y \right| \leq \frac{\rho}{\epsilon},   
	\\  &  C \frac{\alpha \epsilon}  { \left(1+\left|y-\xi^{\prime}\right|\right)^{n-1}} \ \quad   \quad \quad \quad \quad  \quad\text{for} \   \left|y \right| > \frac{\rho}{\epsilon}.
\end{aligned}\right. 
\end{align}
By Proposition \ref{unique sol linear equ zeta}, there exists a solution $w$ to \eqref{w equ}.

Consider the following two norms:
\begin{align*}
	\|w\|_{\star}' =\left\{
\begin{aligned}
	\sup _{y \in \mathbb{R}^n} \sum_{i=0}^2 \frac{\left(1+\left|y-\xi^{\prime}\right|\right)^{n-14+i}  }{\mu \epsilon^{10}}|\partial^i  w(y)|,\quad & n \geq 26, 
	\\	  \sup _{y \in \mathbb{R}^n} \sum_{i=0}^2 \frac{\left(1+\left|y-\xi^{\prime}\right|\right)^{n-13+i}  }{\mu \epsilon^{9}}|\partial^i  w(y)| , \quad& n = 25,
\end{aligned}  	\right.
\end{align*}
and
\begin{align*}
	\|  \mathcal{R}_{1} \|_{{\star}{\star}}' =\left\{
	\begin{aligned}
		 \sup _{y \in \mathbb{R}^n} \frac{\left(1+\left|y-\xi^{\prime}\right|\right)^{n-10}  }{\mu \epsilon^{10}}|   \mathcal{R}_{1}(y)|,\quad & n \geq 26, 
		 	\\	  \sup _{y \in \mathbb{R}^n} \frac{\left(1+\left|y-\xi^{\prime}\right|\right)^{n-9}  }{\mu \epsilon^{9}}|   \mathcal{R}_{1}(y)|, \quad& n = 25.
\end{aligned}  	\right.		 	
\end{align*}
    
By an argument similar to that in Proposition \ref{unique sol linear equ zeta}, Problem \eqref{w equ} admits a unique solution $w$ satisfying
\begin{align*}
		\|w\|_{\star}'\leq C\| 	\mathcal{R}_{1} \|_{{\star}{\star}}'  \leq C.
\end{align*}  

In a similar manner, for $n\geq 26$ we define 
\begin{align*}
	\|\phi_{1} \|_{\star}''= 
&\sum_{i=0}^2\left[\frac{1}{\frac{1}{(1+|y-\xi^{\prime}|)^i}\left(\frac{\mu \epsilon^{20}}{\left(1+\left|y-\xi^{\prime}\right|\right)^{n-24}}+\alpha (\frac{\epsilon}{\rho})^{n-4}\right) }+\frac{\left(1+\left|y-\xi^{\prime}\right|\right)^{n-4+i}}{\alpha}\right] |\partial^i \phi_1(y)|,
\end{align*}
and
\begin{align*}
	\|  \phi_{2}\|_{{\star}{\star}}''= & \sup _{y \in \mathbb{R}^n}\bigg[\frac{\chi_{\left\{\left|y-\xi^{\prime}\right| \leq \frac{\rho}{\epsilon}\right\}}\left(1+\left|y-\xi^{\prime}\right|\right)^{n-20}}{\mu^{2} \epsilon^{20}}+\frac{\chi_{\left\{\frac{\rho}{\epsilon} \leq\left|y-\xi^{\prime}\right| \leq \frac{1}{\epsilon}\right\}}\left(1+\left|y-\xi^{\prime}\right|\right)^{n-1}}{\alpha \epsilon} \\
	&\quad \quad \quad +\frac{\chi_{\{|y-\xi^{\prime}| \geq \frac{1}{\epsilon}\}} \left(1+\left|y-\xi^{\prime}\right|\right)^{n+\sigma}}{\alpha}\bigg]| \phi_{2}(y)|. 
\end{align*} 
For $n = 25$, the norms are defined similarly, with the following replacements: in $\|\phi_{1} \|_{\star}''$, $\frac{\mu^{2} \epsilon^{20}}{\left(1+\left|y-\xi^{\prime}\right|\right)^{n-24 }}$ is replaced by $\frac{\mu^{2} \epsilon^{18}}{\left(1+\left|y-\xi^{\prime}\right|\right)^{n-22 }}$, and in $\| \phi_{2}\|_{{\star}{\star}}''$, $\frac{\chi_{\left\{|y | \leq \frac{\rho}{\epsilon}\right\}}\left(1+\left|y-\xi^{\prime}\right|\right)^{n-20}}{\mu^{2} \epsilon^{20}}$ is replaced by $\frac{\chi_{\left\{|y | \leq \frac{\rho}{\epsilon}\right\}}\left(1+\left|y-\xi^{\prime}\right|\right)^{n-18}}{\mu^{2} \epsilon^{18}}$.
    
\begin{lem}
The following estimate holds:
 \begin{align*}
 	\|\phi-w\|_{\star}''\leq C.
 \end{align*}
\end{lem}
\begin{proof}
We only present the case $n\geq 26$; the case $n=25$ is entirely analogous.

We observe that $\phi-w$ satisfies
\begin{align}\label{phi-w equ}
		\left\{\begin{aligned} 
			&  \mathcal{L} _{\tilde{g}_{0}}(\phi-w ) 
			=   -	\mathcal{R}_{2}(y) +	\mathcal{N}(\phi)+\sum_{i} c_{i}\chi Z_{i}   \quad \text{in}  \ \mathbb{R}^{n},
			\\  &  \int _{\mathbb{R}^{n}}(\phi-w ) \chi Z_{i}=0, \quad i=0,1,\ldots,n ,   
		\end{aligned}\right. 
\end{align}	
where
\begin{equation}\label{R2}   
\begin{aligned}  
 \mathcal{R}_{2}(y)= & P_{\tilde{g}_{0}}\tilde{u}_{0}-\Delta^{2} \tilde{u}_{0}
	+ \frac{ (n+2 )(n-4 )}{4}\tilde{u}_{0} ^{  \frac{2}{n-4}} \big( \Delta_{\tilde{g}_{0}} \tilde{u}_{0}^{  \frac{n-2}{n-4}}- \Delta \tilde{u}_{0}  ^{  \frac{n-2}{n-4}}   -\frac{n-2}{4(n-1 )}R_{\tilde{g}_{0}} \tilde{u}_{0}^{  \frac{n-2}{n-4}}   \big) 
	\\&+\mu \epsilon^{10}\hat{\chi} \Big\{ 2\bar{H} _{ij} \partial_{ijss}  \tilde{u}_{0} +2\partial_{s}\bar{H} _{ij} \partial_{ijs}  \tilde{u}_{0}+\partial_{ss }\bar{H} _{ij} \partial_{ij}  \tilde{u}_{0}  
	\\& \quad\quad\quad\quad-\frac{b_{n}}{2}\partial_{ss }\bar{H} _{ij} \partial_{ij}  \tilde{u}_{0}-\frac{b_{n}}{2}\partial_{jss }\bar{H} _{ij} \partial_{i }  \tilde{u}_{0}
 +  \frac{ (n+2 )(n-4 )}{4} \tilde{u}_{0} ^{  \frac{2}{n-4}}   \bar{H}_{ij} \partial_{ij}(\tilde{u}_{0} ^{  \frac{n-2}{n-4}}) \Big\},   
\end{aligned}
\end{equation}
in view of  \eqref{equ tilde u0}.

By \cite[Lemmas 4.6--4.8]{WeiZhao2013}, we obtain
\begin{align*} 
	& \Big| P_{\tilde{g}_{0}} \tilde{u} _{0} -\Delta ^{2}\tilde{u} _{0}+\mu \epsilon^{10}\hat{\chi} \big( 2\bar{H} _{ij} \partial_{ijss}  \tilde{u}_{0} +2\partial_{s}\bar{H} _{ij} \partial_{ijs}  \tilde{u}_{0}+\partial_{ss }\bar{H} _{ij} \partial_{ij}  \tilde{u}_{0}
	\\& \quad \quad \quad \quad\quad\quad\quad\quad\quad\quad-\frac{b_{n}}{2}\partial_{ss }\bar{H} _{ij} \partial_{ij}  \tilde{u}_{0}-\frac{b_{n}}{2}\partial_{jss }\bar{H} _{ij} \partial_{i }  \tilde{u}_{0}\big)\Big |	
	\\&\leq \left\{\begin{aligned} 
		&C 
			\frac{\mu^{2} \epsilon^{20}}{ \left(1+\left|y-\xi^{\prime}\right|\right)^{n-20}} 
        \quad\quad \text{for} \ 
		\left|y \right| \leq \frac{\rho}{\epsilon},   
		\\  & C \frac{\alpha \epsilon}  { \left(1+\left|y-\xi^{\prime}\right|\right)^{n-1}}  \quad \quad  \ \text{for} \   \left|y \right| > \frac{\rho}{\epsilon}.
	\end{aligned}\right. 
\end{align*}	

Following the approach in \cite[Proposition 5]{BrendleMarques2009}, we also obtain
\begin{align*}   
	& \big|\tilde{u}_{0} ^{  \frac{2}{n-4}} \big( \Delta_{\tilde{g}_{0}} \tilde{u}_{0}^{  \frac{n-2}{n-4}}- \Delta \tilde{u}_{0}  ^{  \frac{n-2}{n-4}}   -\frac{n-2}{4(n-1 )}R_{\tilde{g}_{0}} \tilde{u}_{0}^{  \frac{n-2}{n-4}}   
	+\mu \epsilon^{10}\hat{\chi} \bar{H}_{ij} \partial_{ij}(\tilde{u}_{0} ^{  \frac{n-2}{n-4}}) \big)\big|
	\\=&\tilde{u}_{0} ^{  \frac{2}{n-4}} \Big| \partial_{i} \big[(\tilde{g}_{0} ^{ij}-\delta_{ij}+\tilde{h}_{ij}) \partial_{j}\tilde{u}_{0}^{  \frac{n-2}{n-4}}\big] -\frac{n-2}{4(n-1 )}R_{\tilde{g}_{0}} \tilde{u}_{0}^{  \frac{n-2}{n-4}}     \Big|
	\\\leq &\left\{\begin{aligned} 
		&C \frac{\mu^{2} \epsilon^{20}}{ \left(1+\left|y-\xi^{\prime}\right|\right)^{n-18}}, 
        \quad\quad \text{for} \ 
		\left|y \right| \leq \frac{\rho}{\epsilon},   
		\\  & C \frac{\alpha \epsilon}  { \left(1+\left|y-\xi^{\prime}\right|\right)^{n }} \ \quad  \  \quad  \quad\text{for} \   \left|y \right| > \frac{\rho}{\epsilon}.
	\end{aligned}\right.  
\end{align*} 	 
Together, these estimates yield
	\begin{equation}\label{est R2}   
		|\mathcal{R}_{2}(y)|	
		 \leq \left\{\begin{aligned} 
			&C  
				\frac{\mu^{2} \epsilon^{20}}{ \left(1+\left|y-\xi^{\prime}\right|\right)^{n-20}}, 
            \quad\quad \text{for} \ 
			\left|y \right| \leq \frac{\rho}{\epsilon},   
			\\  & C \frac{\alpha \epsilon}  { \left(1+\left|y-\xi^{\prime}\right|\right)^{n-1}} \   \quad  \ \quad\text{for} \   \left|y \right| > \frac{\rho}{\epsilon}.
		\end{aligned}\right. 
\end{equation}  
Combining \eqref{N  phi est1} with the estimate $\|\phi \|_{\star} \leq C$, we have
\begin{align*} 
	|\mathcal{N} (\phi )|
	 &\leq \left\{\begin{aligned} 
		&C \frac{\mu^{2} \epsilon^{20}}{ \left(1+\left|y-\xi^{\prime}\right|\right)^{n-18}}+\alpha( \frac{\epsilon}{\rho})^{n+2}, \quad\ \text{for} \ 
		\left|y \right| \leq \frac{\rho}{\epsilon},   
		\\  & C \frac{\alpha }  { \left(1+\left|y-\xi^{\prime}\right|\right)^{n+2 }}  \  \quad\quad\text{for} \   \left|y \right| > \frac{\rho}{\epsilon}.
	\end{aligned}\right. 
\end{align*}	
Arguing as in the proof of Lemma \ref{est linear equ zeta}, we deduce that
\begin{align*}
\|\phi-w\|_{\star}''\leq C(\|	\mathcal{R}_{2}\|_{\star\star}''+\|	\mathcal{N}(\phi )\|_{\star\star}'')	\leq C.
\end{align*}  
\end{proof} 

With the above estimates in place, we obtain the desired energy expansion. For notational simplicity, we define
 \begin{align}\label{QR energy}
 \mathcal{E}_{0}
 :=  2^{n-3}n(n+2)(n-4 ) \int _{\mathbb{R}^{n}}  \Big( \frac{1}{1+|y|^{2}}\Big)^{n}dy = 2^{n-3}n(n+2)(n-4 )   \frac{ \pi^{\frac{n}{2}}\Gamma(\frac{n}{2})}{\Gamma(n)} .
 \end{align}   
 \begin{prop} \label{energy expansion}
Let $n\geq 25$. Then the energy admits the following expansion:
 \begin{align*} 
 	\mathcal{F}_{\tilde{g}_{0}}(\xi',\lambda') =& \mathcal{E} _{0}  +\mathcal{E} _{1}(\xi',\lambda')  -\frac{ (n+2 )(n-4 )^2}{4(n-2 )}\mathcal{E} _{2} (\xi',\lambda')  
 	\\&- \int _{\mathbb{R}^{n}}  \bigg\{  2\tilde{h}  _{ij} \partial_{ijss}  \tilde{u}_{0} +2\partial_{s}\tilde{h} _{ij} \partial_{ijs}  \tilde{u}_{0}+\partial_{ss }\tilde{h}  _{ij} \partial_{ij}  \tilde{u}_{0} -\frac{b_{n}}{2}\partial_{ss }\tilde{h} _{ij} \partial_{ij}  \tilde{u}_{0}
 	\\&\quad\quad\quad\quad-\frac{b_{n}}{2}\partial_{jss }\tilde{h}  _{ij} \partial_{i }  \tilde{u}_{0}
  -  \frac{ (n+2 )(n-4 )}{4}  \tilde{u}_{0} ^{  \frac{2}{n-4}}   \tilde{h} _{ij} \partial_{ij}(\tilde{u}_{0} ^{  \frac{n-2}{n-4}}) \bigg\}{w}
 	\\& +
 	\begin{cases}
 	O( \mu ^{3}\epsilon^{\frac{20n}{n-1}})+O\big(\mu ^{\frac{2 (n-2)}{n-4}}\epsilon^{\frac{20(n-2)}{n-4}}\big) +O\big(\alpha(\frac{\epsilon} {\rho} )^{n-4} \big)
    , & n \geq 26, \\[0.2em]
 	\\	  O(\mu ^{3}\epsilon^{\frac{18n}{n-1}})+O\big(\mu ^{\frac{2 (n-2)}{n-4}}\epsilon^{\frac{18(n-2)}{n-4}}\big)+O\big(\alpha(\frac{\epsilon} {\rho} )^{n-4} \big), & n = 25,
 	\end{cases} 
 \end{align*}  
where
 \begin{align*} 
 \mathcal{E} _{1}(\xi',\lambda') =&   - \int _{\mathbb{R}^{n}}\sum_{i,j,k,l}\tilde{h} _{il}\tilde{h} _{jl}\partial_{ikk}\tilde{u}_{0} \partial_{j}\tilde{u}_{0} +  \int _{\mathbb{R}^{n} }\big(\sum_{i,j}\tilde{h} _{ij} \partial_{ij}\tilde{u}_{0}   \big)^{2}
	\\ & -\frac{a_{n}}{4} \int _{\mathbb{R}^{n}}\sum_{i,k,l,m} (\partial _{l}\tilde{h} _{mk})^{2} (\partial_{i }\tilde{u}_{0}  )^{2}
	-\frac{b_{n}}{4} \int _{\mathbb{R}^{n}}\sum_{i,j,m,s}  \partial _{j}\tilde{h} _{ms}  \partial _{i}\tilde{h} _{sm} \partial_{i }\tilde{u}_{0} \partial_{j}\tilde{u}_{0}  
	\\& -\frac{b_{n}}{2} \int _{\mathbb{R}^{n}}\sum_{i,j,m,s}(\tilde{h} _{ms}  \partial _{s}\tilde{h} _{ij}  -  \tilde{h} _{si} \partial _{s}\tilde{h} _{mj} +\tilde{h} _{sj}\partial _{i}\tilde{h} _{ms} -\tilde{h} _{ms}\partial _{i}\tilde{h} _{sj})\partial_{m}(\partial_{i }\tilde{u}_{0} \partial_{j}\tilde{u}_{0}  )
	\\ & +\frac{b_{n}}{2} \int _{\mathbb{R}^{n}}\sum_{i,j,m,s} \tilde{h} _{is}(\partial _{mm}\tilde{h} _{js})   \partial_{i }\tilde{u}_{0}  \partial_{j}\tilde{u}_{0} 
	\\ & +\frac{n-4}{8(n-1 )} \int _{\mathbb{R}^{n}}\sum_{i,k,l,m}  [(\partial _{il}\tilde{h} _{mk})^{2}+ \partial _{l}\tilde{h} _{mk} \partial _{iil}\tilde{h} _{mk} ]\tilde{u}_{0}^{2}
	\\ & -\frac{n-4}{4(n-2 )^{2}} \int _{\mathbb{R}^{n} }\sum_{i,j,m,s}   (\partial _{mm}\tilde{h} _{ij}  \partial _{ss}\tilde{h} _{ij}  )\tilde{u}_{0}^{2} , 
\end{align*}  
and
 \begin{align*} 
 \mathcal{E} _{2}(\xi',\lambda') =&\frac{1}{2}\int _{B_{\frac{\rho}{\epsilon}}}\sum_{i,k,l }  \tilde{h} _{il}\tilde{h} _{kl}\partial_{i} (\tilde{u}_{0} ^{  \frac{n-2}{n-4}}   )\partial_{k} (\tilde{u}_{0} ^{  \frac{n-2}{n-4}}  )  -\frac{n-2}{16(n-1)}\int _{B_{\frac{\rho}{\epsilon}}}\sum_{i,k,l }  (\partial_{l}\tilde{h} _{ik})^{2}  (\tilde{u}_{0} ^{  \frac{n-2}{n-4}}  )^{2} .
\end{align*}    
\end{prop} 
\begin{proof}
		To evaluate $\mathcal{E}_{\tilde{g}_{0}}(\tilde{u}_{0}+\phi)$, we separate the computation into
	\begin{align*}
		\mathcal{E}_{\tilde{g}_{0}}(\tilde{u}_{0}+\phi)-\mathcal{E}_{\tilde{g}_{0}}(\tilde{u}_{0}) \quad \text{and} \quad  \mathcal{E}_{\tilde{g}_{0}}(\tilde{u}_{0}).
	\end{align*}

\textit{Step 1.}	We first address the difference  $\mathcal{E}_{\tilde{g}_{0}}(\tilde{u}_{0}+\phi)-\mathcal{E}_{\tilde{g}_{0}}(\tilde{u}_{0})$. 	

Since $\tilde{u}_0+\phi$ solves equation \eqref{nonlinear c},  we obtain
\begin{align} \label{phi P} 
	&\int _{\mathbb{R}^{n}} \phi P_{\tilde{g}_{0}}(\tilde{u}_{0}+\phi) - \frac{ (n+2 )(n-4 )}{4}(\tilde{u}_{0}+\phi) ^{  \frac{ 2}{n-4}}   \phi   L_{\tilde{g}_{0}}(\tilde{u}_{0}+\phi)^{  \frac{n-2}{n-4}}   
	= 0.
\end{align}  
We now compute
 \begin{equation} \label{Eg0 u0+phi est}  
	\begin{aligned} 
		&\mathcal{E}_{\tilde{g}_{0}}(\tilde{u}_{0}+\phi) \\
		=&\int _{\mathbb{R}^{n}}   (\tilde{u}_{0}+\phi)P_{\tilde{g}_{0}}(\tilde{u}_{0}+\phi) - \frac{ (n+2 )(n-4 )^2}{4(n-2 )}  (\tilde{u}_{0}+\phi) ^{  \frac{n-2}{n-4}}     L_{\tilde{g}_{0}}(\tilde{u}_{0}+\phi)^{  \frac{n-2}{n-4}}    \\
		=&\int _{\mathbb{R}^{n}}  \tilde{u}_{0} P_{\tilde{g}_{0}} \tilde{u}_{0}  + \tilde{u}_{0} P_{\tilde{g}_{0}}  \phi  + \phi P_{\tilde{g}_{0}}(\tilde{u}_{0}+\phi) - \frac{ (n+2 )(n-4 )^2}{4(n-2 )}  (\tilde{u}_{0}+\phi) ^{  \frac{n-2}{n-4}}     L_{\tilde{g}_{0}}(\tilde{u}_{0}+\phi)^{  \frac{n-2}{n-4}} \\
		\overset{\eqref{phi P}}{=}&\mathcal{E}_{\tilde{g}_0}(\tilde{u}_{0})+\frac{ (n+2 )(n-4 )^2}{4(n-2 )} \int _{\mathbb{R}^{n}}\tilde{u}_{0} ^{  \frac{n-2}{n-4}}     L_{\tilde{g}_0}\tilde{u}_{0}^{ \frac{n-2}{n-4}}-  (\tilde{u}_{0}+\phi) ^{  \frac{n-2}{n-4}}     L_{\tilde{g}_0}(\tilde{u}_{0}+\phi)^{  \frac{n-2}{n-4}}
	 \\
		&+\int \phi P_{\tilde{g}_0}\tilde{u}_{0} +\frac{ (n+2 )(n-4 )}{4} \int _{\mathbb{R}^{n}} (\tilde{u}_{0}+\phi) ^{  \frac{ 2}{n-4}}   \phi   L_{\tilde{g}_0}(\tilde{u}_{0}+\phi)^{  \frac{n-2}{n-4}} 
		\\	= \colon &\mathcal{E}_{\tilde{g}_0}(\tilde{u}_{0})+\int _{\mathbb{R}^{n}} \phi  \mathcal{R}+\frac{ (n+2 )(n-4 )}{4}I,   
\end{aligned}
\end{equation}
where we recall that $\mathcal{R}=P_{\tilde{g}_0}\tilde{u}_0-\frac{(n+2)(n-4)}{4}\tilde{u}_{0}^{\frac{2}{n-4}}\,L_{\tilde{g}_0}\tilde{u}_{0}^{\frac{n-2}{n-4}}$, and
		\begin{align*}
		I:=& \int _{\mathbb{R}^{n}}\frac{ n-4  }{ n-2  } \big[ \tilde{u}_{0} ^{  \frac{n-2}{n-4}}     L_{\tilde{g}_0}\tilde{u}_{0}^{ \frac{n-2}{n-4}}
		-   (\tilde{u}_{0}+\phi) ^{  \frac{n-2}{n-4}}     L_{\tilde{g}_0}(\tilde{u}_{0}+\phi)^{  \frac{n-2}{n-4}}\big]\\
		& +  \int _{\mathbb{R}^{n}} (\tilde{u}_{0}+\phi) ^{  \frac{ 2}{n-4}}   \phi   L_{\tilde{g}_0}(\tilde{u}_{0}+\phi)^{  \frac{n-2}{n-4}}+  \tilde{u}_{0} ^{  \frac{ 2}{n-4}}  \phi  L_{\tilde{g}_0} \tilde{u}_{0} ^{  \frac{n-2}{n-4}}  
		\\ =&\int _{\mathbb{R}^{n}} \frac{ 2 }{ n-2  } \big[ (\tilde{u}_{0}+\phi) ^{  \frac{n-2}{n-4}}     L_{\tilde{g}_0}(\tilde{u}_{0}+\phi)^{  \frac{n-2}{n-4}}- \tilde{u}_{0} ^{  \frac{n-2}{n-4}}     L_{\tilde{g}_0}\tilde{u}_{0}^{ \frac{n-2}{n-4}}  \big]\\
		& -  \int _{\mathbb{R}^{n}}\big[  \tilde{u}_{0} (\tilde{u}_{0}+\phi) ^{  \frac{ 2}{n-4}}   L_{\tilde{g}_0}(\tilde{u}_{0}+\phi)^{  \frac{n-2}{n-4}}- (\tilde{u}_{0}+\phi)  \tilde{u}_{0} ^{  \frac{ 2}{n-4}}  L_{\tilde{g}_0} \tilde{u}_{0} ^{  \frac{n-2}{n-4}} \big] .
	\end{align*} 
	
We now estimate $\int _{\mathbb{R}^{n}} \mathcal{R}\,\phi$.
Observing that $\mathcal{R}=\mathcal{R}_{1}+\mathcal{R}_{2}$, where $\mathcal{R}_{1}$ and $\mathcal{R}_{2}$ are defined in \eqref{R1} and \eqref{R2}, respectively, we obtain
	\begin{align}\label{Rphi est}
		\int _{\mathbb{R}^{n}} \mathcal{R} \phi 
		=   \int _{\mathbb{R}^{n}}  \mathcal{R}_{2}  \phi + \mathcal{R}_{1} ( \phi -w)+ \mathcal{R}_{1} w.
	\end{align}  

We only present the case $n\geq 26$; the case $n=25$ is entirely analogous.

Using \eqref{est R1} and \eqref{est R2}, together with $\|\phi-w\|_{\star}''\leq C$, $\|w\|_{\star}'\leq C$, and $\|\phi\|_{\star}\le C$, and following an argument analogous to that in \cite[Proposition 8.2]{WeiZhao2013}, we obtain, for $n\geq 26$,
	\begin{align*}
		\Big | \int _{\mathbb{R}^{n}}  \mathcal{R}_{2}  \phi \Big| 
		\leq   C \mu^3 \varepsilon^{22} \rho^8|\log \epsilon|+C \alpha^2 \rho\left(\frac{\epsilon}{\rho}\right)^{n-4},
	\end{align*}  
	\begin{align*}
		\Big | \int _{\mathbb{R}^{n}}\mathcal{R}_{1} ( \phi -w) \Big| 
		\leq   C \mu^3 \varepsilon^{22} \rho^8|\log \epsilon|+C \alpha^2 \rho\left(\frac{\epsilon}{\rho}\right)^{n-4},
	\end{align*}   
	and 
	\begin{align*}
		\int _{\mathbb{R}^{n}}\mathcal{R}_{1} w  
		= \int _{ |y|\leq \frac{\rho}{\epsilon}}\mathcal{R}_{1} w  + O\Big(\alpha\mu\rho^{10}\big(\frac{\epsilon}{\rho}\big)^{n-4} \Big) .
	\end{align*}   
Consequently, we have
	\begin{align}\label{formula1}
		&\int _{\mathbb{R}^{n}} \mathcal{R} \phi 
		=	\int _{ |y|\leq \frac{\rho}{\epsilon}}\mathcal{R}_{1} w  +O(\mu^3 \varepsilon^{22} \rho^8|\log \varepsilon| ) + O\Big( \alpha^2 \rho\big(\frac{\epsilon}{\rho}\big)^{n-4} \Big) . 
	\end{align} 

We proceed to analyze $I$. 

\textit{Claim:}
\begin{equation}\label{E I est}   
	\begin{aligned}
		I= &\int _{\mathbb{R}^{n}} \frac{ 2 }{ n-2  } \big[ (\tilde{u}_{0}+\phi) ^{  \frac{n-2}{n-4}}     L_{\tilde{g}_0}(\tilde{u}_{0}+\phi)^{  \frac{n-2}{n-4}}- \tilde{u}_{0} ^{  \frac{n-2}{n-4}}     L_{\tilde{g}_0}\tilde{u}_{0}^{ \frac{n-2}{n-4}}  \big]\\
		& -  \int _{\mathbb{R}^{n}}\big[  \tilde{u}_{0} (\tilde{u}_{0}+\phi) ^{  \frac{ 2}{n-4}}   L_{\tilde{g}_0}(\tilde{u}_{0}+\phi)^{  \frac{n-2}{n-4}}- (\tilde{u}_{0}+\phi)  \tilde{u}_{0} ^{  \frac{ 2}{n-4}}  L_{\tilde{g}_0} \tilde{u}_{0} ^{  \frac{n-2}{n-4}} \big]  
	\\=&O( \mu^{\frac{2 (n-2)}{n-4}}  \epsilon^{\frac{20(n-2)}{n-4}}  )+O\Big( \alpha   \big(\frac{\epsilon}{\rho}\big)^{n -2} \Big) .
\end{aligned}
\end{equation}
Indeed,  since $\|\phi\|_{\star}\leq C$, we have the pointwise estimate
 \begin{align*}
 \Big|\frac{ 2 }{ n-2  } \big[ (\tilde{u}_{0}+\phi) ^{  \frac{2(n-2)}{n-4}}    -  \tilde{u}_{0}^{ \frac{2(n-2)}{n-4}}  \big] 
 -   \big[  \tilde{u}_{0} (\tilde{u}_{0}+\phi) ^{  \frac{ n}{n-4}}   - (\tilde{u}_{0}+\phi) \tilde{u}_{0} ^{  \frac{n }{n-4}} \big]  \Big|
	 \leq   C  |\phi|^{  \frac{2(n-2)}{n-4}}.
\end{align*}
Moreover, combining this with the fact that $R_ {\tilde{g}_{0}}  = O(\alpha \epsilon^2)$ and $\|\phi\|_{\star}\leq C$ again, we deduce that
 \begin{align}\label{Rterm}
	&	\Big|  \int _{\mathbb{R}^{n}}  \frac{n-2}{4(n-1)}R_{\tilde{g}_0} \Big\{ \ \frac{ 2 }{ n-2  } \big[ (\tilde{u}_{0}+\phi) ^{  \frac{2(n-2)}{n-4}}    -  \tilde{u}_{0}^{ \frac{2(n-2)}{n-4}}  \big]
		 -   \big[  \tilde{u}_{0} (\tilde{u}_{0}+\phi) ^{  \frac{ n}{n-4}}   - (\tilde{u}_{0}+\phi) \tilde{u}_{0} ^{  \frac{n }{n-4}} \big]  \Big\} \Big| \nonumber
 \\=&O( \alpha  \mu^{\frac{2 (n-2)}{n-4}}  \epsilon^{\frac{20(n-2)}{n-4}+2}  ) 
           + O\Big( \alpha^2 \epsilon  \big(\frac{\epsilon}{\rho}\big)^{n -4} \Big)   .
\end{align}

It now remains to establish that
\begin{equation}\label{expansion II est}   
	\begin{aligned}
 II:=&\int _{\mathbb{R}^{n}} \frac{ 2 }{ n-2  } \big[ (\tilde{u}_{0}+\phi) ^{  \frac{n-2}{n-4}}     \Delta_{\tilde{g}_0}(\tilde{u}_{0}+\phi)^{  \frac{n-2}{n-4}}- \tilde{u}_{0} ^{  \frac{n-2}{n-4}}      \Delta_{\tilde{g}_0}\tilde{u}_{0}^{ \frac{n-2}{n-4}}  \big]\\
& -  \int _{\mathbb{R}^{n}}\big[  \tilde{u}_{0} (\tilde{u}_{0}+\phi) ^{  \frac{ 2}{n-4}}    \Delta_{\tilde{g}_0}(\tilde{u}_{0}+\phi)^{  \frac{n-2}{n-4}}- (\tilde{u}_{0}+\phi)  \tilde{u}_{0} ^{  \frac{ 2}{n-4}}   \Delta_{\tilde{g}_0} \tilde{u}_{0} ^{  \frac{n-2}{n-4}} \big]  
\\=&O( \mu^{\frac{2 (n-2)}{n-4}}  \epsilon^{\frac{20(n-2)}{n-4}}  ) 
 +O\Big( \alpha   \big(\frac{\epsilon}{\rho}\big)^{n -4} \Big).
\end{aligned}
\end{equation}

Once  \eqref{expansion II est}  is obtained, together with \eqref{Rterm}, we obtain the {\textit{Claim}}.

Let $p=\frac{n-2}{n-4}$. For $t\in \mathbb{R} $,  we define
\begin{align*}
	S(t):=&\frac{ p-1 }{ p } \int_{ \mathbb{R}^{n } } \big[ (\tilde{u}_{0}+t\phi) ^{  p}      \Delta_{\tilde{g}_0}(\tilde{u}_{0}+t\phi)^{ p}- \tilde{u}_{0} ^{  p}     \Delta_{\tilde{g}_0}\tilde{u}_{0}^{p}  \big]\\
	& - \int_{ \mathbb{R}^{n } }  \big[  \tilde{u}_{0} (\tilde{u}_{0}+t\phi) ^{ p-1}    \Delta_{\tilde{g}_0}(\tilde{u}_{0}+t\phi)^{ p}- (\tilde{u}_{0}+t\phi)  \tilde{u}_{0} ^{ p-1}   \Delta_{\tilde{g}_0} \tilde{u}_{0} ^{p} \big] . 
\end{align*}  
Obviously,  $II= S(1)$ and $S(0)=0$. 

Differentiating $S(t)$ with respect to $t$ and integrating by parts, we obtain the following:
\begin{align*}
	S'(t)=& (p-1)\int_{ \mathbb{R}^{n } }  [ 2(\tilde{u}_{0}+t\phi) ^{ p-1}   - \tilde{u}_{0}(\tilde{u}_{0} +t\phi) ^{p-2} ]  
	\phi   \Delta_{\tilde{g}_0} (\tilde{u}_{0}+t\phi)^{ p}  
		\\&-p  \int_{ \mathbb{R}^{n } }(\tilde{u}_{0}+t\phi)^{ p-1} \phi  \Delta_{\tilde{g}_0}[\tilde{u}_{0}(\tilde{u}_{0}+t\phi) ^{ p-1}   ]
	+   \int_{ \mathbb{R}^{n } } \tilde{u}_{0}  ^{p-1}   \phi   \Delta_{\tilde{g}_0} \tilde{u}_{0} ^{ p} ,
\end{align*}   
hence $S'(0)=0$.

Differentiating $S(t)$ twice with respect to $t$, we obtain
\begin{align*}
	S''(t) =  & (p-1)\int_{ \mathbb{R}^{n } }  \big[ 3p\tilde{u}_{0}+(4p-2)t\phi\big]  	 (\tilde{u}_{0} +t\phi) ^{  p-3}\phi^{2}   \Delta_{\tilde{g}_0} (\tilde{u}_{0}+t\phi)^{  p} 	\\&-3p(p-1) \int_{ \mathbb{R}^{n } } (\tilde{u}_{0}+t\phi)^{  p-2} \phi^{2}  \Delta_{\tilde{g}_0}[\tilde{u}_{0}(\tilde{u}_{0}+t\phi) ^{  p-1}   ],
\end{align*}  
so $S''(0)=0$.

Therefore, for some $\theta\in(0,1)$,
\begin{align*}
II=& S(1)=\frac{1}{6} S'''(\theta) .
\end{align*}  
Next, we compute
 \begin{align*}
	S'''(\theta)= &  (p-1) (p-2)\int_{ \mathbb{R}^{n } }   [(3p+1)\tilde{u}_{0}+(4p-2) \theta\phi](\tilde{u}_{0}+\theta\phi) ^{ p-4}   	\phi^{3}   \Delta_{\tilde{g}_0} (\tilde{u}_{0}+\theta\phi)^{ p}
	\\& -3p(p-1)(p-2)  \int_{ \mathbb{R}^{n } }    (\tilde{u}_{0}+\theta\phi) ^{ p-3}   	\phi^{3}  \Delta_{\tilde{g}_0} [\tilde{u}_{0} (\tilde{u}_{0}+\theta\phi)^{  p-1} ] 
		\\&+p(p-1) \int_{ \mathbb{R}^{n } }     [ 3p \tilde{u}_{0}+(4p-2) \theta\phi](\tilde{u}_{0}+\theta\phi) ^{ p-3}   	\phi^{2}  \Delta_{\tilde{g}_0} [ (\tilde{u}_{0}+\theta\phi)^{  p-1}\phi ] 
	\\& -3p(p-1) ^{2}  \int_{ \mathbb{R}^{n } }    (\tilde{u}_{0}+\theta\phi) ^{ p-2}   	\phi^{2}  \Delta_{\tilde{g}_0} [\tilde{u}_{0} \phi(\tilde{u}_{0}+\theta\phi)^{  p-2} ] 
	\\=&:II_{1}+II_{2}+II_{3}+II_{4}.
\end{align*} 
Since $\|\phi\|_{\star}\leq C$, we can estimate each term as follows:
 \begin{align*}
|II_{1}| &\leq C\int_{ \mathbb{R}^{n } }   |	\phi|^{2p} \Big| \frac{ (3p+1)\tilde{u}_{0}+(4p-2) \theta\phi }{\tilde{u}_{0}+\theta\phi  }\Big|  |  \tilde{u}_{0}+\theta\phi |^{ p-3}     |\phi|^{3-2p} |\Delta_{\tilde{g}_0} (\tilde{u}_{0}+\theta\phi)^{ p}| 
\\&\leq C\int_{ \mathbb{R}^{n } }  |	\phi|^{2p} |  \tilde{u}_{0}+\theta\phi |^{ p-3}     |\phi|^{3-2p}  |  \tilde{u}_{0}+\theta\phi |^{ p }   =C\int_{ \mathbb{R}^{n } }  |	\phi|^{2p} |\frac{\phi}{ \tilde{u}_{0}+\theta\phi }|^{3-2p} 
 \leq C\int_{ \mathbb{R}^{n } }  |	\phi|^{2p},
\end{align*} 
 \begin{align*}
	|II_{2}| &\leq C\int_{ \mathbb{R}^{n } } \big | (\tilde{u}_{0}+\theta\phi) ^{ p-3}   	\phi^{3}  \Delta_{\tilde{g}_0} [\tilde{u}_{0} (\tilde{u}_{0}+\theta\phi)^{  p-1} ]\big|
\\&\leq C\int_{ \mathbb{R}^{n } }  |\tilde{u}_{0}+\theta\phi |^{ p-3}     |\phi|^{3 }| \tilde{u}_{0}| |  \tilde{u}_{0}+\theta\phi |^{ p -1} 
\\&\leq C\int_{ \mathbb{R}^{n } }  |\tilde{u}_{0}+\theta\phi |^{ 2p-3}     |\phi|^{3 }  
 =C \int_{ \mathbb{R}^{n } } |	\phi|^{2p} |\frac{\phi}{ \tilde{u}_{0}+\theta\phi }|^{3-2p} \leq C\int_{ \mathbb{R}^{n } }  |	\phi|^{2p},
\end{align*}  
 \begin{align*}
	|II_{3}| &\leq C\int_{ \mathbb{R}^{n } }  \big|[ 3p \tilde{u}_{0}+(4p-2) \theta\phi](\tilde{u}_{0}+\theta\phi) ^{ p-3}   	\phi^{2}  \Delta_{\tilde{g}_0} [ (\tilde{u}_{0}+\theta\phi)^{  p-1}\phi  ]\big| 
	\\&\leq C\int_{ \mathbb{R}^{n } }  |3p \tilde{u}_{0}+(4p-2) \theta\phi||\tilde{u}_{0}+\theta\phi| ^{ p-3}   	\phi^{2}  
	\\&\quad\quad\quad\quad\cdot\big[ |\partial^{2}(\tilde{u}_{0}+\theta\phi)^{  p-1}||\phi|+2|\partial (\tilde{u}_{0}+\theta\phi)^{  p-1}||\partial\phi| +|\tilde{u}_{0}+\theta\phi|^{  p-1}|\partial^{2}\phi | ]\big] 
	\\&\leq  C\int_{ \mathbb{R}^{n } }  \big| \frac{ 3p \tilde{u}_{0}+(4p-2) \theta\phi }{\tilde{u}_{0}+\theta\phi}\big||\tilde{u}_{0}+\theta\phi| ^{ 2p-3}   |	\phi|^{2} \big(|\phi|+ |\partial\phi| + |\partial^{2}\phi | \big) 
		\\&\leq  C\int_{ \mathbb{R}^{n } }  |\tilde{u}_{0}+\theta\phi| ^{ 2p-3}   |	\phi|^{2} \big(|\phi|+ |\partial\phi| + |\partial^{2}\phi | \big) 
	\\&\leq C\int_{ \mathbb{R}^{n } }    |	\phi|^{2p} \big|\frac{\phi}{ \tilde{u}_{0}+\theta\phi } \big|^{2-2p} \Big[ \big|\frac{\phi}{ \tilde{u}_{0}+\theta\phi }\big|+  \big|\frac{\partial\phi}{ \tilde{u}_{0}+\theta\phi }\big|+  \big|\frac{\partial^{2}\phi}{ \tilde{u}_{0}+\theta\phi }\big| \Big] 
    \\&\leq O( \mu^{\frac{2 (n-2)}{n-4}}  \epsilon^{\frac{20(n-2)}{n-4}}  )  +O\Big( \alpha   \big(\frac{\epsilon}{\rho}\big)^{n -4} \Big),  
\end{align*} 
and
 \begin{align*}
	|II_{4}| &\leq C\int_{ \mathbb{R}^{n } }  \big| (\tilde{u}_{0}+\theta\phi) ^{ p-2}   	\phi^{2}  \Delta_{\tilde{g}_0} [\tilde{u}_{0} (\tilde{u}_{0}+\theta\phi)^{  p-2} \phi] \big| 
	\\&\leq C\int_{ \mathbb{R}^{n } }   |\tilde{u}_{0}+\theta\phi| ^{ p-2}   |\phi|^{2}   \cdot\big\{ |\partial^{2}[\tilde{u}_{0} (\tilde{u}_{0}+\theta\phi)^{  p-2}]|\phi|+2|\partial [\tilde{u}_{0} (\tilde{u}_{0}+\theta\phi)^{  p-2}]||\partial\phi|
\\&\quad\quad\quad\quad \quad\quad\quad\quad\quad\quad\quad\quad \quad  +|\tilde{u}_{0}||\tilde{u}_{0}+\theta\phi|^{  p-2}|\partial^{2}\phi | ]\big\} 
\\&\leq  C\int_{ \mathbb{R}^{n } }  |\tilde{u}_{0}| |\tilde{u}_{0}+\theta\phi| ^{ 2p-4}   |	\phi|^{2} \big(|\phi|+ |\partial\phi| + |\partial^{2}\phi | \big)  
\\&\leq C\int_{ \mathbb{R}^{n } }    |	\phi|^{2p} \big|\frac{\phi}{ \tilde{u}_{0}+\theta\phi } \big|^{2-2p}\big|\frac{ \tilde{u}_{0}}{  \tilde{u}_{0}+\theta\phi } \big|  \Big[ \big|\frac{\phi}{ \tilde{u}_{0}+\theta\phi }\big|+  \big|\frac{\partial\phi}{ \tilde{u}_{0}+\theta\phi }\big|+  \big|\frac{\partial^{2}\phi}{ \tilde{u}_{0}+\theta\phi }\big| \Big]
 \\&\leq O( \mu^{\frac{2 (n-2)}{n-4}}  \epsilon^{\frac{20(n-2)}{n-4}}  )  +O\Big( \alpha   \big(\frac{\epsilon}{\rho}\big)^{n -4} \Big). 
\end{align*} 
Here we used that $2p=\frac{2(n-2)}{n-4}<3$. Putting these estimates together, we obtain \eqref{expansion II est}.
  
Therefore, combining \eqref{Eg0 u0+phi est}--\eqref{E I est}, we obtain
 \begin{align} \label{Ev-Eu0}
	 \mathcal{E}_{\tilde{g}_{0}}(\tilde{u}_{0}+\phi) 
 =  \mathcal{E}_{\tilde{g}_0}(\tilde{u}_{0})+ \int _{\mathbb R^n}\mathcal{R}_{1} w +O(\mu^3 \epsilon^{ \frac{20n}{n-1}}   ) +O( \mu^{\frac{2 (n-2)}{n-4}}  \epsilon^{\frac{20(n-2)}{n-4}}  ) + O\Big( \alpha   \big(\frac{\epsilon}{\rho}\big)^{n-4} \Big) . 
\end{align}	 

\textit{Step 2.} We next compute the term $\mathcal{E}_{\tilde{g}_0}(\tilde{u}_{0})$, obtaining
 \begin{align*}  
	\mathcal{E}_{\tilde{g}_{0}}(\tilde{u}_{0})=& \int _{\mathbb{R}^{n}}  \tilde{u}_{0} P_{\tilde{g}_{0}}\tilde{u}_{0} -  \frac{ (n+2 )(n-4 )^2}{4(n-2 )} \tilde{u}_{0} ^{  \frac{n-2}{n-4}}     L_{\tilde{g}_{0}} \tilde{u}_{0} ^{  \frac{n-2}{n-4}}    
\\=& \int _{\mathbb{R}^{n}}  \tilde{u}_{0} P_{\tilde{g}_{0}}\tilde{u}_{0} -  \tilde{u}_{0} \Delta ^{2}\tilde{u}_{0} 
	\\ &   -\frac{ (n+2 )(n-4 )^2}{4(n-2 )} \big[ \int _{\mathbb{R}^{n}} \tilde{u}_{0} ^{  \frac{n-2}{n-4}}     L_{\tilde{g}_{0}} \tilde{u}_{0} ^{  \frac{n-2}{n-4}}   -\tilde{u}_{0} ^{  \frac{n-2}{n-4}} (-\Delta \tilde{u}_{0} ^{  \frac{n-2}{n-4}} )\big] 
	\\ &   +  \frac{ (n+2 )(n-4 )}{2(n-2) }  \int _{\mathbb{R}^{n}} \tilde{u}_{0} ^{  \frac{n-2}{n-4}}  (-\Delta \tilde{u}_{0} ^{  \frac{n-2}{n-4}} ) 	
		\\=:&I_{0,1}-\frac{ (n+2 )(n-4 )^2}{4(n-2 )} I_{0,2} +  I_{0,3}.
\end{align*}
Using Lemmas 4.11, 4.13, 4.14, and 4.15 in \cite{WeiZhao2013}, we deduce that
\begin{equation}\label{bigstar}   
	\begin{aligned}
	I_{0,1} = \mathcal{E} _{1}(\xi',\lambda') 
	 +O(\mu ^{3}\epsilon^{\frac{20n}{n-1}})+O\big(\alpha(\frac{\epsilon}{\rho}  )^{n-4} \big).
\end{aligned}
\end{equation}
Moreover, \cite[Proposition 11]{BrendleMarques2009}, together with the fact that $\bar{H}_{ij}(y)=f(|y|^{2})H_{ij}(y)$ where $f(|y|^{2})$ is a fourth-degree polynomial, ensures that
\begin{equation}\label{spadesuit}   
	\begin{aligned}
	I_{0,2}=& \mathcal{E} _{2}(\xi',\lambda') + O( \mu^3 \varepsilon^{23} \rho^ 7|\log \epsilon|)+O\Big(\alpha  (\frac{\epsilon}{\rho} )^{n-4} \Big).
\end{aligned}
\end{equation}
(see also Proposition \ref{Brencdle-Marques 4 poly} in Appendix \ref{Appendix A}).
 
For $I_{0,3}$, we obtain
\begin{equation}\label{clubsuit}   
	\begin{aligned}
	I_{0,3}	 
	=&\frac{ (n+2 )(n-4 )}{2(n-2) } \int _{\mathbb{R}^{n}} \tilde{u}_{0} ^{  \frac{n-2}{n-4}}  (-\Delta \tilde{u}_{0} ^{  \frac{n-2}{n-4}} ) 	  
  = \frac{ n(n+2 )(n-4 )}{8 } \int _{\mathbb{R}^{n}}     \tilde{u}_{0} ^{  \frac{2n}{n-4} }  	  
		\\ =&  2^{n-3}n(n+2)(n-4 )    \int _{\mathbb{R}^{n}}  \Big( \frac{1}{1+|y|^{2}}\Big)^{n}dy .
\end{aligned}
\end{equation}

Finally, collecting \eqref{Ev-Eu0}--\eqref{clubsuit}, we conclude that the desired expansion holds. This completes the proof.
\end{proof}

\section{Proof of Theorem \ref{main thm}}  \label{Sect. pf main thm} 
We first establish the existence of a critical point of the function $F$ defined below. Since $F$ is close to $\mathcal{F}_{\tilde{g}_{0}}(\xi',\lambda')$ and we choose the parameters $\epsilon$, $\rho$, and $\mu$ appropriately, this yields a critical point of $\mathcal{F}_{\tilde{g}_{0}}(\xi',\lambda')$ and hence a solution by Proposition \ref{equvalence between solution and critical point}.
 
From the expression of $\mathcal{F}_{\tilde{g}_{0}}(\xi',\lambda')$ in Proposition \ref{energy expansion}, recalling that $\tilde h_{ij}(y)=\mu \epsilon^{10}\bar H_{ij}(y)=\mu \epsilon^{10}f(|x|^{2})H_{ij}(x)$ for $n\geq 26$, and  $\tilde h_{ij}(y)=\mu \epsilon^{9}\bar H_{ij}(y)=\mu \epsilon^{9}f(|x|^{2})H_{ij}(x)$ for $n=25$,  we  study the function  $F:\mathbb{R}^{n}\times (0,\infty)\rightarrow  \mathbb{R}$, defined as
\begin{align*}  
F(\xi',\lambda')    
	 :=& - \int _{\mathbb{R}^{n}}\sum_{i,j,k,l}\bar{H}_{il}\bar{H}_{jl}\partial_{ikk}\tilde{u}_{0} \partial_{j}\tilde{u}_{0} +  \int _{\mathbb{R}^{n} }\big(\sum_{i,j}\bar{H}_{ij} \partial_{ij}\tilde{u}_{0}   \big)^{2}
	\\ & -\frac{a_{n}}{4} \int _{\mathbb{R}^{n}}\sum_{i,k,l,m} (\partial _{l}\bar{H}_{mk})^{2} (\partial_{i }\tilde{u}_{0}  )^{2}
	  -\frac{b_{n}}{4} \int _{\mathbb{R}^{n}}\sum_{i,j,m,s}  \partial _{j}\bar{H}_{ms}  \partial _{i}\bar{H}_{sm} \partial_{i }\tilde{u}_{0} \partial_{j}\tilde{u}_{0}  
	\\& -\frac{b_{n}}{2} \int _{\mathbb{R}^{n}}\sum_{i,j,m,s}(\bar{H}_{ms}  \partial _{s}\bar{H}_{ij}  -  \bar{H}_{si} \partial _{s}\bar{H}_{mj} +\bar{H}_{sj}\partial _{i}\bar{H}_{ms} -\bar{H}_{ms}\partial _{i}\bar{H}_{sj})\partial_{m}(\partial_{i }\tilde{u}_{0} \partial_{j}\tilde{u}_{0}  )
	\\ & +\frac{b_{n}}{2} \int _{\mathbb{R}^{n}}\sum_{i,j,m,s} \bar{H}_{is}(\partial _{mm}\bar{H}_{js})  \partial_{i }\tilde{u}_{0}  \partial_{j}\tilde{u}_{0} 
	\\ & +\frac{n-4}{8(n-1 )} \int _{\mathbb{R}^{n}}\sum_{i,k,l,m}  [(\partial _{il}\bar{H}_{mk})^{2}+ \partial _{l}\bar{H}_{mk} \partial _{iil}\bar{H}_{mk} ]\tilde{u}_{0}^{2}
	\\ & -\frac{n-4}{4(n-2 )^{2}} \int _{\mathbb{R}^{n} }\sum_{i,j,m,s}   (\partial _{mm}\bar{H}_{ij}  \partial _{ss}\bar{H}_{ij}  )\tilde{u}_{0}^{2} 
\\&  -\frac{ (n+2 )(n-4 )^2}{4(n-2 )} \bigg\{ \frac{1}{2}\int _{\mathbb{R}^{n}}\sum_{i,k,l }  \bar{H}_{il}\bar{H}_{kl}\partial_{i} (\tilde{u}_{0} ^{  \frac{n-2}{n-4}}   )\partial_{k} (\tilde{u}_{0} ^{  \frac{n-2}{n-4}}  )
	\\&\quad\quad\quad  \quad\quad\quad\quad\quad\quad -\frac{n-2}{16(n-1)}\int _{\mathbb{R}^{n}}\sum_{i,k,l }  (\partial_{l}\bar{H}_{ik})^{2}  (\tilde{u}_{0} ^{  \frac{n-2}{n-4}}  )^{2} \bigg\}	  	 
	\\&- \int _{\mathbb{R}^{n}}  \bigg\{  2\bar{H} _{ij} \partial_{ijss}  \tilde{u}_{0} +2\partial_{s}\bar{H} _{ij} \partial_{ijs}  \tilde{u}_{0}+\partial_{ss }\bar{H} _{ij} \partial_{ij}  \tilde{u}_{0}
 -\frac{b_{n}}{2}\partial_{ss }\bar{H} _{ij} \partial_{ij}  \tilde{u}_{0}-\frac{b_{n}}{2}\partial_{jss }\bar{H} _{ij} \partial_{i }  \tilde{u}_{0}
	\\&\quad\quad\quad- \frac{ (n+2 )(n-4 )}{4}  \tilde{u}_{0} ^{  \frac{2}{n-4}}   \bar{H}_{ij} \partial_{ij}(\tilde{u}_{0} ^{  \frac{n-2}{n-4}}) \bigg\}\bar{w},
\end{align*}
where $\bar{w}=\frac{w}{\mu\epsilon^{10}}$ for $n\geq 26$ and $\bar{w}=\frac{w}{\mu\epsilon^{9}}$ for $n=25$, and where $w$ is defined in \eqref{w equ}.

Since $\bar{H}_{ik}$ is symmetric, i.e., $\bar{H}_{ik}(-y)=\bar{H}_{ik}(y)$, we obtain the following proposition.
 \begin{prop} 
 	The function $F$ satisfies $F(\xi',\lambda')=F(-\xi',\lambda')$ for all $(\xi',\lambda')\in\mathbb{R}^{n}\times(0,+\infty)$. Consequently, $\frac{\partial}{\partial\xi'}F(0,\lambda')=0$ and $\frac{\partial^{2}}{\partial\xi'\partial\lambda'}F(0,\lambda')=0$ for all $\lambda'>0$.
 \end{prop} 

Using \eqref{w equ} for $w$, we note that if $\xi'=0$, then by the properties of $\bar{H}_{ij}$ we have $\mathcal{R}_{1}(y)=0$. Hence, together with Theorem \ref{eigenvalue thm}, $\bar{w}|_{\xi'=0}(y)=0$ for all $y\in\mathbb{R}^{n}$.
 
For simplicity, we use $F_{Q}$ and $F_R$ to denote the $F$ defined in \cite[Section 9]{WeiZhao2013} and \cite[Section 4]{BrendleMarques2009}, respectively; these represent the corresponding reduced energy functionals for $Q$-curvature and scalar curvature.
Since
\begin{align*}
	\tilde{u}_{0} ^{  \frac{n-2}{n-4}}  =2^{\frac{n-2}{2}}\big( \frac{\lambda'}{\lambda'^{2} +|y-\xi' |^{2}}\big)^{\frac{n-2}{2}},
\end{align*}
we have
\begin{equation}
    F(0,\lambda')=2^{n-4}F_Q(0,\lambda')-\frac{(n+2)(n-4)^2}{4(n-2)}2^{n-2}F_R(0,\lambda'). 
\end{equation}
The detailed expressions of $F(0,\lambda')$ and its Hessian at $(0,\lambda')$ are given in Propositions \ref{F(0,lambda)} and \ref{Hessian F(0,lambda)}.

\begin{prop}\label{summary}
For $n\ge 25$, there exists a smooth function $f:[0,+\infty)\rightarrow \mathbb{R}$ such that the corresponding function $F(\xi',\lambda')$ has a strict local negative minimum at $(0,1)$. In particular,
\begin{equation}\label{f for all dim}
f(s):=
\begin{cases}
 \tau_n-8126s+1662s^{2}-98s^{3}+ s^{4} , & n\ge 26,\\[4pt]
 1028(s+1)^{1/2}-10 (s+1)^{3/2}-\frac{121}{6}(s+1)^{5/2} +(s+1)^{7/2} +\tau_{25}(s+1)^{-1/2}, & n=25.
\end{cases}
\end{equation}
where $\tau_n$ and $\tau_{25}$ are constants determined in Lemmas \ref{fourth lem} and \ref{3.5 lem}, respectively.
\end{prop}

For readability, we postpone the proof of Proposition \ref{summary} to Appendix \ref{Appendix A}, where we divide Proposition \ref{f for all dim} into two propositions: Proposition \ref{fourth prop} for $n\ge 26$ and Proposition \ref{3.5 prop} for $n=25$.

We now prove Theorem \ref{main thm} by a gluing method.
\begin{prop}\label{prop h}
	Assume $n\geq 25$. Let $g_{0}$ be a Riemannian metric on $\mathbb{R}^{n}$, smooth for $n\geq 26$ and of class $C^9$ for $n=25$, of the form $g_{0}(x)=\exp(h(x))$,  where $h(x)$ is a trace-free symmetric two-tensor on $\mathbb{R}^{n}$ satisfying  $h(x)=0$ for $|x|\geq 1$, 
	\begin{align*}  
		|h(x)|+|\partial h(x)|+|\partial^{2} h(x)|+|\partial^{3} h(x)|+|\partial^{4} h(x)|\leq \alpha
	\end{align*} 
	for all $x\in\mathbb{R}^{n}$, and 
	\begin{align*}  
		h_{ij}(x)=\mu \epsilon^{8}f(\epsilon^{-2}|x|^{2})H_{ij}(x), \ \text{for} \ n \geq 26,
	\end{align*} 
	and 
	\begin{align*}  
		h_{ij}(x)=\mu \epsilon^{7}f(\epsilon^{-2}|x|^{2})H_{ij}(x), \ \text{for} \ n =25,
	\end{align*} 
for $|x|\leq \rho$, where $f$ is defined in \eqref{f for all dim}.
    
If $\alpha$ and $\rho^{4-n}\mu^{-2}\epsilon^{n-24}$ are sufficiently small, then there exists a positive solution $u(x)$ to
	\begin{gather*} 
		P _{g_{0}}u=\frac{ (n+2 )(n-4 )}{4}  u^{  \frac{2}{n-4}}    L_{g_{0}}u^{  \frac{n-2}{n-4}}    \ \ \text{in}\  \mathbb{R}^{n},
		\\ \int_{\mathbb{R}^{n} }  u^{  \frac{2n}{n-4}} <2^{n}\int_{\mathbb{R}^{n} }  \left( \frac{1}{1+|x|^{2}}\right) ^{n},
		\\\sup_{|x|\leq \epsilon} u \geq C\epsilon^{ - \frac{n-4}{2}} .
	\end{gather*}	   
\end{prop}
\begin{proof}
From Propositions \ref{fourth prop} and \ref{3.5 prop}, it follows that the function $F(\xi',\lambda')$ attains a strict local minimum at $(0,1)$, with $F(0,1)<0$. Thus, there exists an open set $\Omega\subset\Lambda$ containing $(0,1)$ such that
\begin{align*}
	F(0, 1)<\inf_{(\xi',\lambda')\in\partial \Omega} F(\xi',\lambda')<0.
\end{align*}
By Proposition \ref{energy expansion}, we obtain the following. For $n\geq 26$,
\begin{align*}  
		\mathcal{F} (\xi',\lambda')=\mathcal{E}_{0} +\mu^{2}\epsilon^{20}F (\xi',\lambda')+O(\mu ^{3}\epsilon^{\frac{20n}{n-1}} + \mu ^{\frac{2 (n-2)}{n-4}}\epsilon^{\frac{20(n-2)}{n-4}})+O\big(\alpha(\frac{\epsilon}{\rho}  )^{n-4} \big),
\end{align*}
and for $n=25$,
\begin{align*}  
	\mathcal{F} (\xi',\lambda')=\mathcal{E}_{0} +\mu^{2}\epsilon^{18}F (\xi',\lambda')+ O(\mu ^{3}\epsilon^{\frac{18n}{n-1}} + \mu ^{\frac{2 (n-2)}{n-4}}\epsilon^{\frac{18(n-2)}{n-4}})+O\big(\alpha(\frac{\epsilon} {\rho} )^{n-4} \big).
\end{align*}
Therefore, if $\rho^{4-n}\mu^{-2}\epsilon^{n-24}$ is sufficiently small, then
\begin{align*}
	\mathcal{F}_{ \tilde{g}_{0}}(0, 1)<\inf_{(\xi',\lambda')\in\partial \Omega} 	\mathcal{F}_{ \tilde{g}_{0}}(\xi',\lambda')<\mathcal{E}_{0}.
\end{align*}
Hence, there exists $(\bar{\xi}',\bar{\lambda}')\in\Omega$ satisfying
 \begin{align*}
 	\mathcal{F}_{ \tilde{g}_{0}}(\bar{\xi}',\bar{\lambda}')=\inf_{(\xi',\lambda')\in \Omega} 	\mathcal{F}_{ \tilde{g}_{0}}(\xi',\lambda')<\mathcal{E}_{0}.
 \end{align*}
Thus, by Proposition \ref{equvalence between solution and critical point}, we conclude that $\tilde{u}_{0}(y)+\phi(y)$ is a solution of \eqref{main Q:R=c  tilde g equ} with $\|\phi\|_{\star}\leq C$. From the definition of $\|\cdot\|_{\star}$, we have
\begin{align*}
|\phi(y)|\leq C\frac{\alpha}{(1+|y-\xi'|)^{n-4}}\leq C\alpha\tilde{u}_{0}(y),
\end{align*}
which implies that $\tilde{u}_{0}(y)+\phi(y)>0$ for sufficiently small $\alpha$. Reverting the scaling $y=\epsilon^{-1}x$, we obtain $u_{0}(x)+\epsilon^{-\frac{n-4}{2}}\phi(\epsilon^{-1}x)$ as the desired positive solution.
\end{proof}
\begin{thm}
Assume $n\geq 25$. Then there exists a Riemannian metric $g_{0}$ on $\mathbb{R}^{n}$, smooth for $n\geq 26$ and of class $C^9$ but not $C^{10}$ for $n=25$, with the following properties:
	\begin{enumerate}[label=\textup{(\roman*)}]
	\item $(g_{0})_{ij}(x)=\delta_{ij}$ for $|x|\geq \frac{1}{2}$,
	\item $g_{0}$ is not conformally flat,
	\item There exists a sequence of positive solutions $u_{k}$ $(k\in\mathbb{N})$ such that
	\begin{gather*} 
		P _{g_{0}}u_{k}=\frac{ (n+2 )(n-4 )}{4}  u_{k}^{  \frac{2}{n-4}}    L_{g_{0}}u_{k}^{  \frac{n-2}{n-4}}    \ \ \text{in}\  \mathbb{R}^{n},
		\\\int_{\mathbb{R}^{n} }  u_{k}^{  \frac{2n}{n-4}} <2^{n}\int_{\mathbb{R}^{n} }  \left( \frac{1}{1+|x|^{2}}\right) ^{n},
		\\\sup_{|x|\leq 1} u_{k} \rightarrow \infty\  \text{as}\ k  \rightarrow \infty .
	\end{gather*}	
\end{enumerate}
\end{thm}
\begin{proof}
Let $\eta: \mathbb{R}\to\mathbb{R}$ be a smooth cut-off function satisfying $\eta(r)=1$ for $r\leq 1$ and $\eta(r)=0$ for $r\geq 2$. We then define a trace-free symmetric $2$-tensor on $\mathbb{R}^{n}$ by
\begin{align*}
h_{ij}(x)=\sum_{N=N_{0}}^{\infty}\eta(4N^{2}|x-x_{N}|)2^{-65N}f_1(2^{16N}|x-x_{N}|)H_{ij}(x-x_{N}), \ \text{for} \ n \geq 26,
\end{align*}
and 
 \begin{align*}
	h_{ij}(x)=\sum_{N=N_{0}}^{\infty}\eta(4N^{2}|x-x_{N}|^2)2^{-57N}f_{2}(2^{16N}|x-x_{N}|^2)H_{ij}(x-x_{N}), \ \text{for} \ n =25,
\end{align*}
where $x_{N}=(\frac{1}{N},0,\ldots,0)\in\mathbb{R}^{n}$. For $n\geq 26$, the tensor $h$ is smooth. In dimension $n=25$, it belongs to $C^9\setminus C^{10}$. Consequently, the metric $g_0=\exp(h)$ is smooth for $n\geq 26$ and of class $C^9$ but not $C^{10}$ for $n=25$.

Let $\alpha$ be the constant introduced in Proposition \ref{prop h}. If $N_{0}$ is sufficiently large, then $|h|+|\partial h|+|\partial^{2}h|+|\partial^{3}h|+|\partial^{4}h|\leq \alpha$ for all $x\in\mathbb{R}^{n}$ and $h(x)=0$ for $|x|\geq \frac{1}{2}$. Furthermore, for $N\geq N_{0}$ and $|x-x_{N}|\leq \frac{1}{4N^{2}}$, we have
 \begin{align*}
	h_{ij}(x)= 2^{-65N}f_1(2^{16N}|x-x_{N}|^2)H_{ij}(x-x_{N}), \ \text{for} \ n \geq 26,
\end{align*}	 
and 
 \begin{align*}
	h_{ij}(x)= 2^{-57N}f_2(2^{16N}|x-x_{N}|^2)H_{ij}(x-x_{N}), \ \text{for} \ n =25.
\end{align*}	 
Applying Proposition \ref{prop h} with $\mu=2^{-N}$, $\epsilon=2^{-8N}$, and $\rho=\frac{1}{4N^{2}}$ for each $x_{N}$, the assertion follows immediately.
\end{proof}

\appendix
\section{}\label{Appendix A} 
\subsection{Explicit formula of derivatives of \texorpdfstring{$F$}{F}}\label{Sect. Explicit formula}
By applying \cite[Proposition 17]{BrendleMarques2009} and \cite[Proposition 9.2]{WeiZhao2013}, we derive the following expression for $F(0,\lambda')$.
\begin{prop} \label{F(0,lambda)} 
We have
	\begin{align*}  
	&2^{4-n }	F(0,\lambda')    
	\\	=& \frac{n-4}{2}|\mathbb{S}^{n-1}| \sum_{ i,j,k,l}(W_{ikjl}+W_{iljk})^{2}
	\\&\cdot \biggl\{\lambda'^{n-4}\Big[-\frac{a_{n}(n-4)}{n(n+2 )}\int _{0}^{\infty} [ r^{2} f'(r^{2}) ^{2}+2f(r^{2}) f'(r^{2}) ] \frac{r^{n+5}}{(\lambda'^{2}+r^{2})^{n-2}} dr
	\\&-\frac{a_{n}(n-4)}{2n}\int _{0}^{\infty}  f(r^{2})^{2}   \frac{r^{n+3}}{(\lambda'^{2}+r^{2})^{n-2}} dr 
	\\&-\frac{b_{n}(n-4)}{n(n+2)}\int _{0}^{\infty}  [r^{2}f'(r^{2})+f(r^{2})]^{2} \frac{r^{n+3}}{(\lambda'^{2}+r^{2})^{n-2}}dr  
	\\&+\frac{1}{2n(n-1 )(n+2)}\int _{0}^{\infty}         \big[ 3(n+8)f'(r^{2})^{2}+2(n+8)f(r^{2}) f''(r^{2})    	\\&\quad\quad\quad\quad\quad\quad\quad\quad\quad\quad+2(n+18)  r^{2}f'(r^{2})f''(r^{2})+4r^{4}  f''(r^{2})^{2}
\\&\quad\quad\quad\quad\quad\quad\quad\quad\quad\quad+4r^{2}f(r^{2})  f'''(r^{2})+4r^{4}f'(r^{2})  f'''(r^{2}) \big]\frac{r^{n+3}}{(\lambda'^{2}+r^{2})^{n-4}} dr 
	\\&+\frac{1}{2n(n-1 ) }\int _{0}^{\infty}         \big[ 4r^{2}f'(r^{2})^{2}+ (n+8)f(r^{2}) f'(r^{2}) +2r^{2}     f(r^{2}) f''(r^{2})\big]\frac{r^{n+1}}{(\lambda'^{2}+r^{2})^{n-4}} dr 
	\\& + \frac{1}{4(n-1)} \int _{0}^{\infty}      f(r^{2})^{2}\frac{r^{n-1}}{(\lambda'^{2}+r^{2})^{n-4}} dr
	\\& -\frac{1}{ n(n-2 )^{2}(n+2)}\int _{0}^{\infty}    \big[(n+4) f'(r^{2})+2r^{2}  f''(r^{2})   \big]^{2}\frac{r^{n+3}}{(\lambda'^{2}+r^{2})^{n-4}} dr 
	\Big]
	\\&+ \frac{ (n-4) }{8 n(n-1 )}\lambda'^{n-2}
	\int _{0}^{\infty}\frac{r^{n+1}}{(\lambda'^{2}+r^{2})^{n-2}}
	\big[ (n+2)f(r^{2}) ^{2}+4r^{2}f(r^{2})f'(r^{2})+2r^{4}f'(r^{2})^{2}   \big]dr \biggr\}.	 
\end{align*}
\end{prop} 

We now show that $(0,1)$ is a strict local minimum of $F$. To this end, we compute the Hessian matrix of $F$ at $(0,\lambda')$.
\begin{prop}\label{Hessian F(0,lambda)} 
The second derivatives of $F$ at $(0,\lambda')$ are given by
 \begin{align*}   
	&2^{4-n }\frac{\partial^{2}}{\partial\xi '_{p}  \partial\xi '_{q} }	F(0,\lambda')
	\\=&  \frac{(n-4)^{2}}{n(n+2)}|\mathbb{S}^{n-1}| 
 \cdot\bigg\{  \sum_{ i,j,k  }(W_{ikjp}+W_{ipjk})(W_{ikjq}+W_{iqjk})\tilde{J}_{1} (\lambda')
	  +  \sum_{ i,j,k,l  }(W_{ikjl}+W_{iljk}) ^{2} \delta_{pq}\tilde{J}_{2} (\lambda')\bigg\},
\end{align*}   
where:
\begin{align*} 	 
\tilde{J}_{1} (\lambda' )=&  \lambda'^{n-4}\biggl\{- (n-2) \int _{0}^{\infty}  f(r^{2}) ^{2} \big[  \frac{nr^{n+5}}{(\lambda'^{2}+r^{2})^{n }}-\frac{(n+2)r^{n+3}}{(\lambda'^{2}+r^{2})^{n-1}}  \big] dr   
	\\&- \frac{8a_{n}(n-2)}{n+4}\int _{0}^{\infty} \big[r^{2}  f'(r^{2})^{2}+2f(r^{2}) f'(r^{2})  \big]  \big[  \frac{(n-1)r^{n+7}}{(\lambda'^{2}+r^{2})^{n }}-\frac{2r^{n+5}}{(\lambda'^{2}+r^{2})^{n-1}}  \big]    dr
	\\&-2a_{n}(n-2)  \int _{0}^{\infty} f (r^{2})^{2}  \big[  \frac{(n-1)r^{n+5}}{(\lambda'^{2}+r^{2})^{n }}-\frac{2r^{n+3}}{(\lambda'^{2}+r^{2})^{n-1}}  \big]    dr
	\\&-\frac{8b_{n}(n-1)(n-2) }{n+4} \int _{0}^{\infty} \big[r^{2}f' (r^{2}) +f(r^{2})\big]^{2}    \frac{ r^{n+5}}{(\lambda'^{2}+r^{2})^{n }}    dr
	\\&+\frac{16b_{n} (n-2) }{n+4} \int _{0}^{\infty} \big[r^{2}f' (r^{2}) ^{2}+f(r^{2})f'(r^{2})\big]    \frac{ r^{n+5}}{(\lambda'^{2}+r^{2})^{n -1}}    dr
	\\&+4b_{n} (n-2)  \int _{0}^{\infty} \big[r^{2}f(r^{2})f' (r^{2})  +f(r^{2}) ^{2} \big]    \frac{ r^{n+3}}{(\lambda'^{2}+r^{2})^{n -1}}    dr
	\\&-\frac{4b_{n}}{n+4}  \int _{0}^{\infty}  f' (r^{2}) ^{2}   \frac{ r^{n+5}}{(\lambda'^{2}+r^{2})^{n -2}}    dr
	\\&-\frac{ b_{n}(n+2)}{2}  \int _{0}^{\infty}  f (r^{2}) ^{2}   \frac{ r^{n+1}}{(\lambda'^{2}+r^{2})^{n -2}}    dr
	\\&-2b_{n}  \int _{0}^{\infty}  f (r^{2})  f' (r^{2})     \frac{ r^{n+3}}{(\lambda'^{2}+r^{2})^{n -2}}    dr
	\\&+b_{n}  \int _{0}^{\infty}\big[  (n+4) f (r^{2})  f' (r^{2})+2r^{2} f  (r^{2})  f'' (r^{2}) \big]    \frac{ r^{n+3}}{(\lambda'^{2}+r^{2})^{n -2}}    dr
	\\&+\frac{4(n-3)}{(n-1)(n+4)}   \int _{0}^{\infty}\big[  3(n+8) f' (r^{2})^{2}+4r^{4}   f''(r^{2}) ^{2}+2(n+18)r^{2} f' (r^{2})  f'' (r^{2}) 
	\\&\quad \quad \quad \quad \quad \quad \quad \quad \quad +2(n+ 8) f  (r^{2})  f'' (r^{2})+4r^{2} f  (r^{2})  f''' (r^{2}) 
	\\&\quad \quad \quad \quad \quad \quad \quad \quad \quad +4r^{4}f'  (r^{2})  f''' (r^{2})  \big]    \frac{ r^{n+5}}{(\lambda'^{2}+r^{2})^{n -2}}    dr 
	\\&+\frac{2(n-3)}{ n-1 } \int _{0}^{\infty}\big[ 4r^{2} f' (r^{2})^{2}+ (n+8)f( r^{2}  ) f' (r^{2}) +2 r^{2} f  (r^{2})  f'' (r^{2})  \big]     \cdot\frac{ r^{n+3}}{(\lambda'^{2}+r^{2})^{n -2}}    dr 
	\\&-\frac{8(n-3)}{  (n-2  )^{2}(n+4)}  \int _{0}^{\infty}\big[ (n+4) f' (r^{2}) +2 r^{2} f'' (r^{2})   \big]^{2}    \cdot\frac{ r^{n+5}}{(\lambda'^{2}+r^{2})^{n -2}}    dr \bigg\}
	\\&+  \frac{2  (n+2)(n-2)}{ (n+4)} \lambda'^{n-2}  
	\cdot \int _{0}^{\infty}  \big[ 2 f (r^{2}) f'  (r^{2})  +   r^{2}   f ' (r^{2}) ^{2}\big] \cdot   \frac{  r^{n+5}}{(\lambda'^{2}+r^{2})^{n  }}   dr ,
\end{align*}  
and
\begin{align*}  
\tilde{J}_{2} (\lambda' )=&\lambda'^{n-4}\bigg\{-  \frac{a_{n}}{n+4} \int _{0}^{\infty} [r^{2}f'(r^{2}) ^{2} +2f (r^{2}) f'(r^{2}) ] 
	\\&\quad\quad\quad\quad\quad\quad\cdot\big[  \frac{2 (n-1 )(n-2 )r^{n+7}}{(\lambda'^{2}+r^{2})^{n }}-\frac{(n-2)(n+8)r^{n+5}}{(\lambda'^{2}+r^{2})^{n-1}} +\frac{(n+4)r^{n+3}}{(\lambda'^{2}+r^{2})^{n-2}}  \big] dr   
	\\&-  \frac{a_{n}}{2} \int _{0}^{\infty}  f (r^{2}) ^{2}  \big[  \frac{2 (n-1 )(n-2 )r^{n+5}}{(\lambda'^{2}+r^{2})^{n }}-\frac{(n-2)(n+6)r^{n+3}}{(\lambda'^{2}+r^{2})^{n-1}} +\frac{(n+2)r^{n+1}}{(\lambda'^{2}+r^{2})^{n-2}}  \big] dr  
	\\&-  \frac{b_{n}(n-2)}{n+4} \int _{0}^{\infty}  [r^{2}f'(r^{2})  + f (r^{2})   ] ^{2}  \big[  \frac{2 (n-1 ) r^{n+5}}{(\lambda'^{2}+r^{2})^{n }}-\frac{ (n+4)r^{n+3}}{(\lambda'^{2}+r^{2})^{n-1}}  \big] dr  
	\\&+  \frac{4b_{n}(n-2)}{n+4} \int _{0}^{\infty}  [r^{2}f'(r^{2})^{2}  + f (r^{2})  f' (r^{2})   ]     \frac{ r^{n+5}}{(\lambda'^{2}+r^{2})^{n-1 }} dr  
	\\ &-  \frac{ b_{n} }{n+4} \int _{0}^{\infty}   f'(r^{2})^{2}    \frac{ r^{n+5}}{(\lambda'^{2}+r^{2})^{n-2 }} dr  
	\\&+\frac{1}{2(n-1)(n+4)}   \int _{0}^{\infty}\big[  3(n+8) f' (r^{2})^{2}+4r^{4}   f''(r^{2}) ^{2}+2(n+18)r^{2} f' (r^{2})  f'' (r^{2}) 
	\\&\quad \quad \quad \quad \quad \quad \quad \quad \quad +2(n+ 8) f  (r^{2})  f'' (r^{2})+4r^{2} f  (r^{2})  f''' (r^{2}) 
	\\&\quad \quad \quad \quad \quad \quad \quad \quad \quad +4r^{4}f'  (r^{2})  f''' (r^{2})  \big]  \cdot \big[ \frac{ 2(n-3)r^{n+5}}{(\lambda'^{2}+r^{2})^{n -2}}  -\frac{ (n+4)r^{n+3}}{(\lambda'^{2}+r^{2})^{n -3}}   \big]   dr 
	\\&+\frac{1}{2(n-1) }    \int _{0}^{\infty}\big[ 4r^{2}f'  (r^{2}) ^{2}+  (n+ 8) f  (r^{2})  f'  (r^{2})+2r^{2} f  (r^{2})  f''  (r^{2})] 
	\\&\quad \quad \quad \quad \quad \quad \quad \quad \quad  \cdot \big[ \frac{ 2(n-3)r^{n+3}}{(\lambda'^{2}+r^{2})^{n -2}}  -\frac{ (n+2)r^{n+1}}{(\lambda'^{2}+r^{2})^{n -3}}   \big]   dr 
	\\&+\frac{n+2}{4(n-1) }    \int _{0}^{\infty}  f  (r^{2}) ^{2}   \cdot \big[ \frac{ 2(n-3)r^{n+1}}{(\lambda'^{2}+r^{2})^{n -2}}  -\frac{ nr^{n-1}}{(\lambda'^{2}+r^{2})^{n -3}}   \big]   dr 
	\\&-\frac{1}{ (n-2) ^{2}(n+4)}    \int _{0}^{\infty} \big[  (n+4)  f'  (r^{2})  +  2r^{2}   f''  (r^{2}) \big] ^{2}\cdot \big[ \frac{ 2(n-3)r^{n+5}}{(\lambda'^{2}+r^{2})^{n -2}}  -\frac{  (n+4 )r^{n+3}}{(\lambda'^{2}+r^{2})^{n -3}}   \big]   dr\bigg\}
	\\&+\lambda'^{n-2}\bigg\{\frac{  (n+2)(n-2)}{ 2(n+4)} 
	\cdot \int _{0}^{\infty}   \big[ 2 f (r^{2}) f'  (r^{2})  +   r^{2}   f ' (r^{2}) ^{2}\big] \cdot   \frac{  r^{n+5}}{(\lambda'^{2}+r^{2})^{n  }}   dr 
	\\&\quad \quad \quad \quad - \frac{ (n+2) (n-2)}{4 (n-1)} 
	\cdot \int _{0}^{\infty}     f ' (r^{2}) ^{2}  \cdot   \frac{  r^{n+5}}{(\lambda'^{2}+r^{2})^{n-1  }}   dr  \bigg\} .
\end{align*}  	 
\end{prop}
\begin{proof}
To begin with, from the proof of \cite[Proposition 21]{Brendle2008}, by the algebraic properties of $\bar H_{ij}$,  we know that
\begin{align*}  
   \tilde{u}_{0} ^{  \frac{2}{n-4}}  \bar{H}_{ij} \partial_{ij} \tilde{u}_{0} ^{  \frac{n-2}{n-4}} 
&=  \frac{2}{ \lambda'^2+|y-\xi'|^{2}} n(n-2)\lambda'^{\frac{n-2}{2}}(\lambda'^2+|y-\xi'|^{2})^{-\frac{n+2}{2}}\sum_{ i,j}\bar{H}_{ij}  \xi'_{i}\xi'_{j}
\\&=     2n(n-2)\lambda'^{\frac{n-2}{2}}(\lambda'^2+|y-\xi'|^{2})^{-\frac{n+4}{2}}\sum_{ i,j}\bar{H}_{ij}  \xi'_{i}\xi'_{j},
\end{align*}  
Hence, using the fact that $\bar{w}|_{\xi'=0}=0$, we obtain
\begin{align*}  
	& \frac{\partial^{2}}{\partial\xi'  _{p}  \partial\xi'  _{q} }  \big(   \tilde{u}_{0} ^{  \frac{2}{n-4}}  \bar{H}_{ij} \partial_{ij} \tilde{u}_{0} ^{  \frac{n-2}{n-4}}\bar{w}  \big)|_{\xi'=0}
	=4n(n-2)\lambda'^{\frac{n-2}{2}}(\lambda'^2+|y|^{2})^{-\frac{n+4}{2}} \bar{H}_{pq} \bar{w}  | _{\xi'=0}=0.
\end{align*}  
On the other hand, by the proof of \cite[Lemma 9.15]{WeiZhao2011}, and noting again that $\bar{w}|_{\xi'=0}=0$, we have
 \begin{align*} 
	\frac{\partial^{2}}{\partial\xi'  _{p}  \partial\xi'  _{q} } \big[& \big(  2\bar{H} _{ij} \partial_{ijss}  \tilde{u}_{0} +2\partial_{s}\bar{H} _{ij} \partial_{ijs}  \tilde{u}_{0}+\partial_{ss }\bar{H} _{ij} \partial_{ij}  \tilde{u}_{0}
	\\&-\frac{b_{n}}{2}\partial_{ss }\bar{H} _{ij} \partial_{ij}  \tilde{u}_{0}-\frac{b_{n}}{2}\partial_{jss }\bar{H} _{ij} \partial_{i }  \tilde{u}_{0}  \big)\bar{w}   \big]| _{\xi'=0}=0.
\end{align*}   
Finally, combining \cite[Proposition 20]{BrendleMarques2009} and \cite[Proposition 9.3]{WeiZhao2013} yields the desired expression for the second derivative $\frac{\partial^{2}}{\partial\xi'_{p}\partial\xi'_{q}}F(0,\lambda')$.
\end{proof}

In the next two subsections, we verify that for a suitable $f$, the corresponding $F$ satisfies
\begin{itemize}
\item $F(0,1)<0$,
 \item $\frac{\partial}{\partial \lambda'} F(0,1)=0$,
    \item $\frac{\partial^2}{\partial\lambda'^2}F(0,1)>0$,
    \item $\frac{\partial^2}{\partial \xi_p'\partial \xi_q'}F(0,1)$ is positive definite.
\end{itemize}

 \subsection{Quartic Polynomials and the Case \texorpdfstring{$n\geq 26$}{n>=26}} \label{Sect. n26} 
In this section, we take a step toward our ultimate goal for $n\geq 25$ by showing that, when one uses a quartic polynomial, $F(\xi',\lambda')$ possesses a strict local minimum at $(0,1)$ in all dimensions $n\geq 26$ (Proposition \ref{fourth prop}). This is in contrast with the $Q$--curvature case, where the corresponding dimension bound is $n\geq 25$.

We take
\begin{align*}
	f(s)=\tau-8126s+1662s^{2}-98s^{3}+ s^{4},
\end{align*} 
where $\tau$ is a constant to be determined later.

We point out that for a linear $f(s)$, $F(\xi',\lambda')$ admits a strict local minimum at $(0,1)$ in dimensions $n\geq 52$; the detailed computation is omitted due to its length, and the resulting dimension bound coincides with those in the scalar curvature and $Q$--curvature cases.

Using the computational software \textit{Mathematica}, we carry out the following  calculations.
 \begin{prop}
 	Suppose $n>24$. Then
 	\begin{align*}  
 		2^{4-n}F(0,\lambda')    
 		=& \frac{n-4}{16(n^{2}-4)}|\mathbb{S}^{n-1}| \sum_{i,j,k,l}(W_{ikjl}+W_{iljk})^{2} \frac{\Gamma(\frac{n}{2}-9)\Gamma(\frac{n}{2}+7)}{\Gamma(n+1)}I(\lambda'),
 	\end{align*}
	where:
  \begin{align*} 
		I(\lambda')=&-\frac{\lambda'^{20} (n-4) (n+14) \big(n^6+44 n^5-664 n^4-16704 n^3+35664 n^2-75200 n+76800\big)}{2 (n-24) (n-22) (n-20)}
		\\&+\frac{98 \lambda'^{18} (n-4) \big(n^6+34 n^5-528 n^4-10352 n^3+21936 n^2-45536 n+47360\big)}{(n-22) (n-20)}
		\\&-\frac{\lambda'^{16} (n-4) }{2 (n-20) (n+12)}\big(12928 n^6+329480 n^5-5286448 n^4-77975904 n^3+168522816 n^2
		\\&\quad \quad \quad-353488640 n+371619840\big)
		\\&-\frac{\lambda'^{14} (n-4) }{(n+10) (n+12)}\big(-171002 n^6-3045532 n^5+51789584 n^4+540829920 n^3
		\\&\quad \quad \quad-1189179360 n^2+2480090176 n-2673669888\big)
		\\&-\frac{\lambda'^{12} (n-18) (n-4)}{2 (n+8) (n+10) (n+12)} \big(2 n^6 \tau+4354940 n^6+8 n^5 \tau+49073888 n^5
		-336 n^4 \tau
		\\&\quad \quad \quad-921554688 n^4-1984 n^3 \tau-6219348160 n^3+2592 n^2 \tau+14402028864 n^2
		\\&\quad \quad \quad+7808 n \tau-30185605888 n-5120 \tau+33611444224\big)
		\\&-\frac{\lambda'^{10} (n-18) (n-16) (n-4)}{(n+8) (n+10) (n+12)} \big(-98 n^5 \tau-13505412 n^5+392 n^4 \tau+8624 n^3 \tau
		\\&\quad \quad \quad+1836736032 n^3-9408 n^2 \tau-4321731840 n^2-23520 n \tau+7995203904 n
		\\&\quad \quad \quad+12544 \tau-7779117312\big)
		\\&-\frac{\lambda'^8 (n-18) (n-16) (n-14) (n-4) }{2 (n+6) (n+8) (n+10) (n+12)}\big(3324 n^5 \tau+66031876 n^5-13296 n^4 \tau
		\\&\quad \quad \quad-132063752 n^4-159552 n^3 \tau-4490167568 n^3+265920 n^2 \tau
		\\&\quad \quad \quad+13206375200 n^2+159552 n \tau-27469260416 n+25356240384\big)
		\\&+\frac{8126 \lambda'^6 (n-18) (n-16) (n-14) (n-12) (n-4) (n-2) \big(n^3-6 n^2+4 n-8\big) \tau}{(n+6) (n+8) (n+10) (n+12)}
		\\&-\frac{\lambda'^4 (n-18) (n-16) (n-14) (n-12) (n-10) (n-4) (n-2)^2 \big(n^2-2 n+8\big) \tau^2}{2 (n+4) (n+6) (n+8) (n+10) (n+12)}.
	\end{align*}
 \end{prop}  
 
 \begin{prop}
 	Suppose $n>24$. Then
 	\begin{align*}  
 		&2^{4-n} \frac{\partial^{2}}{\partial\xi '_{p}  \partial\xi '_{q} }	F(0,\lambda')     
 		\\=& \frac{2(n-4)^{2}}{(n+2)(n+4)}|\mathbb{S}^{n-1}|\frac{\Gamma(\frac{n}{2}-7)\Gamma(\frac{n}{2}+5)}{\Gamma(n+1)}
 		\\& \cdot\bigg\{ \sum_{ i,j,k  }(W_{ikjp}+W_{ipjk})(W_{ikjq}+W_{iqjk}) J_{1}(\lambda')
 		+  \sum_{ i,j,k,l}(W_{ikjl}+W_{iljk})^{2}  \delta_{pq}J_{2}(\lambda')\bigg\}  
 	\end{align*}
	where:
\begin{align*}
 &J_{1}(\lambda')
\\= &-\frac{3 \lambda'^{18} (n+10) (n+12) (n+14) (n+16) }{(n-22) (n-20) (n-18) (n-16) (n-2)}\big(n^5-2 n^4-376 n^3+744 n^2-1536 n+1568\big)
\\&+\frac{49 \lambda'^{16} (n+10) (n+12) (n+14) }{2 (n-20) (n-18) (n-16) (n-2)}\big(19 n^5-38 n^4-5548 n^3+11128 n^2-23744 n+24320\big)
\\&-\frac{\lambda'^{14} (n+10) (n+12) }{8 (n-18) (n-16) (n-2)}\big(190596 n^5-381192 n^4-42693504 n^3+85883712 n^2
\\&\quad\quad\quad\quad-186732096 n+192334464\big)
\\&-\frac{\lambda'^{12} (n+10) }{4 (n-16) (n-2)}\big(-1864770 n^5+3729540 n^4+305822280 n^3-609835728 n^2
\\&\quad\quad\quad\quad+1342578432 n-1404059904\big)
\\&-\frac{\lambda'^{10} }{4 (n-2)}\big(4 n^5 \tau+16623412 n^5-8 n^4 \tau-33246824 n^4-448 n^3 \tau-1861822144 n^3+448 n^2 \tau
\\&+3734787520 n^2+1216 n \tau-8515936832 n-640 \tau+9036742784\big)
\\&-\frac{\lambda'^8 (n-14) }{4 (n-2) (n+8)}\big(-294 n^5 \tau-67527060 n^5+588 n^4 \tau+135054120 n^4+19992 n^3 \tau+4591840080 n^3
\\&\quad\quad\quad\quad-30576 n^2 \tau-9615853344 n^2-18816 n \tau+24201698304 n-25930391040\big)
\\&-\frac{\lambda'^6 (n-14) (n-12) }{8 (n-2) (n+6) (n+8)}\big(6648 n^5 \tau+198095628 n^5-13296 n^4 \tau-396191256 n^4-212736 n^3 \tau
\\&\quad\quad\quad\quad -6339060096 n^3+531840 n^2 \tau+15847650240 n^2-531840 n \tau-53882010816 n
\\&\quad\quad\quad\quad  +638208 \tau+57051540864\big)
\\& +\frac{4063 \lambda'^4 (n-14) (n-12) (n-10) (n-2) }{2 (n+6) (n+8)}\big(n^2-2 n+8\big) \tau,
\end{align*}
and
 \begin{align*}
	&J_{2}(\lambda')
	\\= &
-\frac{\lambda'^{18} (n+10) (n+12) (n+14) }{4 (n-22) (n-20) (n-18) (n-16) (n-2)}\big(n^6+42 n^5-512 n^4-13744 n^3+30960 n^2
\\&\quad\quad\quad\quad-62432 n+62464\big)
\\&+\frac{49 \lambda'^{16} (n+10) (n+12) }{8 (n-20) (n-18) (n-16) (n-2)}\big(7 n^6+228 n^5-2816 n^4-58288 n^3+132624 n^2-260480 n
\\&\quad\quad\quad\quad+263680\big)
\\&-\frac{\lambda'^{14} (n+10) }{16 (n-18) (n-16) (n-2)}\big(38784 n^6+952980 n^5-11878992 n^4-182528832 n^3+433691136 n^2
\\&\quad\quad\quad\quad-847209024 n+860398848\big)
\\&-\frac{\lambda'^{12}}{16 (n-16) (n-2)} \big(-855010 n^6-14918160 n^5+189634160 n^4+2024166400 n^3-5095058208 n^2
\\&\quad\quad\quad\quad+9684699392 n-9954548224\big)
\\&-\frac{\lambda'^{10} }{8 (n-2) (n+8)}\big(2 n^6 \tau +4354940 n^6+12 n^5 \tau +49870236 n^5-256 n^4 \tau -652394704 n^4-1568 n^3 \tau \\&\quad\quad\quad\quad-4363896800 n^3+2400 n^2 \tau +12390504576 n^2+4544 n \tau -22942227520 n
\\&\quad\quad\quad\quad-2560 \tau +23896246784\big)
\\&-\frac{\lambda'^8 (n-14) }{16 (n-2) (n+8)}\big(-294 n^5 \tau -40516236 n^5+588 n^4 \tau -27010824 n^4+19992 n^3 \tau +3835537008 n^3
\\&\quad\quad\quad\quad-30576 n^2 \tau -10696286304 n^2-18816 n \tau +17286927360 n-15558234624\big)
\\&-\frac{\lambda'^6 (n-14) (n-12) }{16 (n+6) (n+8)}\big(3324 n^4 \tau +66031876 n^4+66031876 n^3-106368 n^2 \tau -2509211288 n^2
\\&\quad\quad\quad\quad+53184 n \tau +4490167568 n-159552 \tau -6867315104\big)
\\&+\frac{4063 \lambda'^4 (n-14) (n-12) (n-10) (n-2)  }{8 (n+6) (n+8)}\big(4 n^3+5 n^2-50 n+8\big) \tau.
\end{align*} 
 \end{prop}

 \begin{lem}\label{fourth lem}
	Let $n\geq 26$. Then there exists $\tau_{n}\in\mathbb{R}$ such that $I(1)<0$, $I'(1)=0$, $I''(1)>0$, $J_{1}(1)>0$, and $J_{2}(1)>0$.
 \end{lem} 
 \begin{proof}
Given that
\begin{align*}
	&I'(1)= \frac{2(n-4)}{(n-24)(n-22)(n-20)(n+4)(n+6)(n+8)(n+10)(n+12)}\cdot\big(A_{n}\tau^{2}+B_{n}\tau+C_{n}\big),
\end{align*}
where
\begin{align*}
	A_{n} &=     n^{12}-142 n^{11}+8844 n^{10}-317368 n^9+7249296 n^8-109954080 n^7+1122984512 n^6
	\\&-7705357952 n^5+35253785088 n^4-108010784768 n^3+224774037504 n^2-287748587520 n
	\\&+163499212800;
	\\ 	B_{n} &= -18214 n^{12}+2456908 n^{11}-142655544 n^{10}+4636621936 n^9-91495851552 n^8
		\\&+1101089216064 n^7-7443947480192 n^6+18658803730688 n^5+74788571867136 n^4
			\\&-558755740905472 n^3+991310540046336 n^2-862370243149824 n+792508712878080;
	\\ 	C_{n} &=76455333 n^{12}-9496798398 n^{11}+487687831980 n^{10}-13000920861624 n^9
		\\&+176667988757136 n^8-686070988965408 n^7-11894011635817920 n^6+151214182660593024 n^5
		\\&-379057448996411904 n^4-1794812456598964224 n^3
		+5984012982653706240 n^2
			\\&-12808242643596410880 n+11975419412663500800,
\end{align*} 
we see that $I'(1)=0$ is equivalent to
\begin{align*}
	A_{n}\tau^{2}+ B_{n}\tau +C_{n} =0.
\end{align*}
Furthermore, since the discriminant satisfies 
\begin{align*}
	\Delta_{n }=&B_{n} ^{2}-4A_{n} C_{n} 
	\\=& 16 (n-24) (n-22) (n-20) (n-18) (n-2)^2 (n+4)
	\\\big( &
1620529 n^{17}-369276932 n^{16}+38693244560 n^{15}-2464277820464 n^{14}+106229576869392 n^{13}
\\&-3266915908421760 n^{12}+73625792383951616 n^{11}-1230743913746199552 n^{10}
\\&+15302475939568590592 n^9-141249074519139441664 n^8+968729111720801476608 n^7
\\&-5022880287336529661952 n^6+20587626227193041596416 n^5
\\&-68653467325694663163904 n^4
\\&+170053384910948394663936 n^3-262372610147447927734272 n^2
\\&+287758117059735538630656 n-148042446923943452344320\big),
\end{align*}
the largest root of $\Delta_{n}=0$ is approximately $n\approx 25.2137$. It follows that $\Delta_{n}>0$ for all $n\geq 26$. Consequently, for each $n\geq 26$, there exists a root $\tau_{n}>11170$ satisfying $I'(1)=0$.

Indeed, the largest root  of equation
\begin{align*}
	&11170^2A_{n}+ 11170B_{n} +C_{n} 
	\\=&-3 \big(742049 n^{12}-76560054 n^{11}+772814300 n^{10}+269170030568 n^9-19715338205232 n^8
	\\&+701918019214176 n^7-15023545763238080 n^6
+200585338481638272 n^5
	\\&-1618302295008553472 n^4
+7170833628717733888 n^3-15033920374301368320 n^2
	\\&+19447664333668024320 n-13742452889144524800\big)=0
\end{align*}  
is approximately $n\approx 25.2144$, and hence $11170^2A_{n}+11170B_{n}+C_{n}<0$ for all $n\geq 26$.
Notice that the largest root of $A_{n}=0$ is $n=24$, and hence $A_{n}>0$ for $n\geq 25$. Thus, for each $n\geq 26$, there exists a root $\tau_{n}>11170$.
 
For $n\geq 26$,
\begin{align*}
 I'' (1 )|_{\tau_{n}} =&I'' (1 )-3I ' (1)|  _{\tau_{n}}
	\\=&\frac{32(n-4) }{(n-24) (n-22) (n-20) (n+6) (n+8) (n+10) (n+12)}
	\\\cdot&\Big\{   \big(1566 n^{11}-230001 n^{10}+14732610 n^9-538267956 n^8+12301596120 n^7
	\\&-181122590832 n^6+1701179115168 n^5-9695289109824 n^4+30053192924544 n^3
	\\&-41714529818112 n^2+30870200881152 n-24473204490240\big)\tau_{n}
	\\&-13515584 n^{11}+1859758177 n^{10}-108004006090 n^9+3377613429348 n^8
		\\&-58889664760440 n^7+485915001564336 n^6+487778304094688 n^5
 	\\&-42769131943904704 n^4+296499749724193408 n^3-690320202562788864 n^2
 		\\&+1116108537469526016 n-869682352497131520\}
	\\=:&\frac{32 (n-4) ( \mathcal{A}_{1}^{1}\tau_{n} +\mathcal{A}_{1}^{2} )}{(n-24) (n-22) (n-20) (n+6) (n+8) (n+10) (n+12)} 
	\\>&\frac{32(n-4) (11170\mathcal{A}_{1}^{1}  +\mathcal{A}_{1}^{2} )}{(n-24) (n-22) (n-20) (n+6) (n+8) (n+10) (n+12)} 
	\\=:&\mathcal{A}_{1}>0,
\end{align*} 
since the largest root of $\mathcal{A}_{1}=0$ is approximately $n\approx 25.2686$, while $\mathcal{A}_{1}^{1}>0$ because the largest root of $\mathcal{A}_{1}^{1}=0$ is $n=24$.
 	
For $n\geq 26$, we have
\begin{align*}
	& (n-22) (n-20) (n-18) (n-16) (n-2) (n+6) (n+8)J_{1}(1)|_{\tau_{n}} 
	\\=& 	  \big(1273 n^{11}-159550 n^{10}+8700324 n^9-270325864 n^8+5262961488 n^7-66521653920 n^6
	\\&+548035371968 n^5-2897126992256 n^4+9629696254464 n^3-20143268648960 n^2
	\\&+25768971460608 n-15536901980160\big) \tau _{n}
	\\& -11593211 n^{11}+1371873158 n^{10}-68641659972 n^9+1861924238912 n^8-28736231800272 n^7
	\\&+228410960082528 n^6-353883534347584 n^5-8163841713754112 n^4+57364717229618688 n^3
	\\&-140245792558693376 n^2+284729830927712256 n-235962924403507200
	\\=:&\mathcal{A}_{2}^{1}\tau _{n}+\mathcal{A}_{2}^{2}
	\\>&  11170\mathcal{A}_{2}^{1} +\mathcal{A}_{2}^{2} = :\mathcal{A}_{2}>0
\end{align*} 
as the largest root of $\mathcal{A}_{2}=0$ is approximately $n\approx 22.9875$, while $\mathcal{A}_{2}^{1}>0$ because the largest root of $\mathcal{A}_{2}^{1}=0$ is $n=22$.
 	
And
\begin{align*}
	&4 (n-22) (n-20) (n-18) (n-16) (n-2) (n+6) (n+8)J_{2}(1) |_{\tau_{n}} 
	\\= & \big(8126 n^{12}-933217 n^{11}+45703594 n^{10}-1238185620 n^9+20029332248 n^8-191946477456 n^7
 	\\&+954841916896 n^6-721773172928 n^5-16618668229504 n^4+80222193412608 n^3
 		\\&-148984173639680 n^2 +108806134161408 n-15536901980160\big) \tau_{n} 
	\\&-8352153 n^{11}+977561442 n^{10}-48139753884 n^9+1273465425288 n^8-18756163846128 n^7
 	\\&+130839431347872 n^6+110105639248320 n^5-8393232335119488 n^4+52285301346103296 n^3
 	\\&-124338339294219264 n^2+172130995742121984 n-124578057769205760
	\\=&\mathcal{A}_{3}^{1}\tau _{n}+\mathcal{A}_{3}^{2}
	\\>& 11170 \mathcal{A}_{3}^{1} +\mathcal{A}_{3}^{2} = :\mathcal{A}_{3}>0,
\end{align*} 
as the largest root of $\mathcal{A}_{3}=0$ is approximately $n\approx 22.0001$, while $\mathcal{A}_{3}^{1}>0$ because the largest root of $\mathcal{A}_{3}^{1}=0$ is $n=22$.

Since
\begin{align*}
	I(1)= -\frac{n-4}{2 (n-24) (n-22) (n-20) (n+4) (n+6) (n+8) (n+10) (n+12)}\cdot\big(\mathcal{B}_{1}\tau^{2}+\mathcal{B}_{2}\tau+\mathcal{B}_{3}\big),
\end{align*}
where
\begin{align*}
	\mathcal{B}_{1}= &	n^{12}-142 n^{11}+8844 n^{10}-317368 n^9+7249296 n^8-109954080 n^7+1122984512 n^6-7705357952 n^5
		\\&+35253785088 n^4-108010784768 n^3+224774037504 n^2-287748587520 n+163499212800;
	\\\mathcal{B}_{2}= &-13122 n^{12}+1752108 n^{11}-100568472 n^{10}+3226224432 n^9-62714800416 n^8+741477620544 n^7
		\\&-4899895257216 n^6+11703905392896 n^5+50631426812928 n^4-361163754442752 n^3
 	\\&+645118219173888 n^2-586093582614528 n+529583293071360;
	\\\mathcal{B}_{3}= &43046721 n^{12}-5227375422 n^{11}+261337884108 n^{10}-6737176804728 n^9+87122623070352 n^8
		\\&-279542808279840 n^7-6350539679256000 n^6+73120298093589888 n^5-166477479359344128 n^4
	\\&-861066600050608128 n^3+2876091630629855232 n^2-6102096697681182720 n
		\\&+5618627094380544000.
\end{align*}
Then, for $n\geq 26$, the discriminant
\begin{align*}
 &\mathcal{B}_{2}^{2}-4\mathcal{B}_{1}\mathcal{B}_{3}
	\\=&-384 (n-2)^2 \big(1620529 n^{21}-441121646 n^{20}+55843138980 n^{19}-4360164337744 n^{18}
	\\&+234750910395552 n^{17}-9228664185107808 n^{16}+273560805100823168 n^{15}-6223477339605259520 n^{14}
 	\\&+109503959307436125952 n^{13}-1489428781354813212160 n^{12}+15522662178160866147328 n^{11}
 	\\&-121571272407562455662592 n^{10}+689051447430617885958144 n^9-2591685612091941693120512 n^8
 	\\&+4467761841381215891456000 n^7+15361185445136161488896000 n^6
 	\\&-165314646106637036774490112 n^5+777437100544382359805362176 n^4
 	\\&-2213536329196429249584562176 n^3+3670518539840253520936697856 n^2
 	\\&-4194585523626473981780951040 n+2209704403314827483047526400\big)
	\\=&:\mathcal{A}_{4}<0
\end{align*}
since the (largest) root of $\mathcal{A}_{4}$ is $n=24$. Moreover, $\mathcal{B}_{1}>0$ for $n\ge 26$, since the largest root of $\mathcal{B}_{1}=0$ is also $n=24$. Consequently, $I(1)<0$ for all $n\ge 26$.
\end{proof}
 
\begin{prop}\label{fourth prop}
	Let $n\geq 26$ and let $\tau_{n}$ be as in Lemma \ref{fourth lem}. Then the function $F(\xi',\lambda')$ has a strict local minimum at $(0,1)$.
\end{prop}  
\begin{proof}
Since $I'(1)=0$, we have $\frac{\partial}{\partial\lambda'}F(0,1)=0$. By symmetry, $\frac{\partial}{\partial\xi'}F(0,1)=0$. Therefore, $(0,1)$ is a critical point of $F(\xi',\lambda')$.
 
Since $J_{1}(1)>0$ and $J_{2}(1)>0$, the matrix $\big\{\frac{\partial^{2}}{\partial\xi'_{p}\partial\xi'_{q}}F(0,1)\big\}$ is positive definite. On the other hand, since $I''(1)>0$, we have $\frac{\partial^{2}}{\partial\lambda'^2}F(0,1)>0$. Thus, $(0,1)$ is a strict local minimum of $F(\xi',\lambda')$.
\end{proof}

\subsection{A Function of Fractional Order \texorpdfstring{$3.5$}{3.5} and the Case \texorpdfstring{$n=25$}{n=25}} \label{Sect. n25} 
This subsection completes the final case $n=25$. For this purpose, we introduce the function
 \begin{align*} 
 f(s)=1028(s+1)^{1/2}-10 (s+1)^{3/2}-\frac{121}{6}(s+1)^{5/2}+(s+1)^{7/2} +\tau_{25}(s+1)^{-1/2}, 
 \end{align*}
 where
\begin{align}\label{tau25}
 		\tau  _{25}=-\frac{8 \left(\sqrt{1872413768932826265919670946}+686611801547814\right)}{1997391566979}\approx-2923.3455.
 	\end{align}     

We performed the following calculations using \textit{Mathematica}.
 \begin{prop}
 	\begin{align*}  
 		2^{-21}F(0,\lambda' )    
 		=& \frac{21}{2}|\mathbb{S}^{24}| \sum_{ i,j,k,l}(W_{ikjl}+W_{iljk})^{2}  I(\lambda' ) 
 	\end{align*}
	where:
\begin{align*}
&	I(\lambda')
=	-\frac{\pi \lambda'^4}{804061418471659929600 (\lambda'+1)^{25}} \cdot
\\ \bigg\{ & 12\tau^2 \bigl(206935783873 \lambda'^{23} + 5173394596825 \lambda'^{22} + 61899939898043 \lambda'^{21} \\
	& \quad + 471432421311475 \lambda'^{20} + 2563692929921422 \lambda'^{19} + 10583515167428190 \lambda'^{18} \\
	& \quad + 34419164335856274 \lambda'^{17} + 90307822390129890 \lambda'^{16} + 194173514325267750 \lambda'^{15} \\
	& \quad + 345577279044694710 \lambda'^{14} + 512097684246361650 \lambda'^{13} + 633563405425936770 \lambda'^{12} \\
	& \quad + 654533088444422400 \lambda'^{11} + 563587635674265600 \lambda'^{10} + 403012851502917600 \lambda'^9 \\
	& \quad + 238092552479587680 \lambda'^8 + 115391265522735825 \lambda'^7 + 45420824131737705 \lambda'^6 \\
	& \quad + 14306847118522875 \lambda'^5 + 3525198183836595 \lambda'^4 + 655558478552250 \lambda'^3 \\
	& \quad + 86640750795690 \lambda'^2 + 7261828323750 \lambda' + 290473132950\bigr) \\
	& + 4\tau \bigl(14429703386070 \lambda'^{31} + 360742584651750 \lambda'^{30} + 4231019430734985 \lambda'^{29} \\
	& \quad + 30741028160810625 \lambda'^{28} + 153027734590593660 \lambda'^{27} + 537986052269381100 \lambda'^{26} \\
	& \quad + 1276714383840759388 \lambda'^{25} + 1454736270124343500 \lambda'^{24} - 2996406485732400398 \lambda'^{23} \\
	& \quad - 21123667401110772750 \lambda'^{22} - 62029920896712335809 \lambda'^{21} - 112581744747648590185 \lambda'^{20} \\
	& \quad - 94537899261514939576 \lambda'^{19} + 173328305782064474760 \lambda'^{18} + 947421312418140288270 \lambda'^{17} \\
	& \quad + 2415063391720845134430 \lambda'^{16} + 4478935756575848014650 \lambda'^{15} + 6640759628766212202570 \lambda'^{14} \\
	& \quad + 8156442318784040247975 \lambda'^{13} + 8433272454107931236175 \lambda'^{12} + 7392741900850895991900 \lambda'^{11} \\
	& \quad + 5505711005652520306380 \lambda'^{10} + 3478126266077220324000 \lambda'^9 + 1855123815531500503920 \lambda'^8 \\
	& \quad + 828775443075295895550 \lambda'^7 + 306433823693011396830 \lambda'^6 + 92156382947703776625 \lambda'^5 \\
	& \quad + 21978278652649685625 \lambda'^4 + 4000294174735404000 \lambda'^3 + 522099308029440960 \lambda'^2 \\
	& \quad + 43520137144233750 \lambda' + 1740805485769350\bigr) \\
	& + (\lambda'+1)^2 \bigl(47782021512460620 \lambda'^{37} + 1098986494786594260 \lambda'^{36} + 11946077641925118860 \lambda'^{35} \\
	& \quad + 81338162681837144020 \lambda'^{34} + 387097713106056400975 \lambda'^{33} + 1357487577321380516505 \lambda'^{32} \\
	& \quad + 3587420676624046730190 \lambda'^{31} + 7115211795048081921490 \lambda'^{30} + 10107510757237381038775 \lambda'^{29} \\
	& \quad + 8564206868344862711585 \lambda'^{28} - 264179455788349923780 \lambda'^{27} - 10818199494631961351900 \lambda'^{26} \\
	& \quad - 5739454741760203200375 \lambda'^{25} + 32487442258414169143375 \lambda'^{24} + 92475304980584393226338 \lambda'^{23} \\
	& \quad + 109014885719352132162174 \lambda'^{22} - 16518444189432737120345 \lambda'^{21} - 323693244963573412886959 \lambda'^{20} \\
	& \quad - 663657333642117276586080 \lambda'^{19} - 644026522490534172769216 \lambda'^{18} + 231909054069673697512623 \lambda'^{17} \\
	& \quad + 2234953491734023467936025 \lambda'^{16} + 5091909290671118961294130 \lambda'^{15}
     \\
	& \quad + 7966423219949788447579950 \lambda'^{14}+ 9842740280632634320183245 \lambda'^{13} \\
	& \quad + 10080801724561822596930235 \lambda'^{12} + 8744296338773971687506460 \lambda'^{11} \\
	& \quad + 6485469169584938185104260 \lambda'^{10} + 4123058561675945121629595 \lambda'^9  \\
	& \quad + 2240469034777757217212205 \lambda'^8 + 1032771120987648238443870 \lambda'^7 + 398785243920441249431490 \lambda'^6 \\
	& \quad+ 126602814375858551126025 \lambda'^5  + 32166467181330845493375 \lambda'^4 + 6283665262269877127100 \lambda'^3 \\
	& \quad+ 885251477694649911300 \lambda'^2  + 79983629117653811550 \lambda' + 3477549092071904850\bigr) \bigg\}.
\end{align*} 
 \end{prop}   
   \begin{prop}
 	\begin{align*}  
 		&2^{-21} \frac{\partial^{2}}{\partial\xi'_{p}\partial\xi'_{q}}F(0,\lambda')
 		\\=& \frac{49}{75}|\mathbb{S}^{24}| 
 		 \cdot\bigg\{ \sum_{ i,j,k  }(W_{ikjp}+W_{ipjk})(W_{ikjq}+W_{iqjk}) J_{1}(\lambda' )
 		+  \sum_{ i,j,k,l}(W_{ikjl}+W_{iljk})^{2}  \delta_{pq}J_{2}(\lambda' )\bigg\}, 
 	\end{align*}
	where:
   \begin{align*}
		& J_{1} (\lambda')=-\frac{\pi \lambda'^4}{103634582825236168704 (1+\lambda')^{27}} \cdot \\
	&\bigg\{  -12\tau^2 \Big( 1203388693650 + 32491494728550 \lambda' + 417911082685785 \lambda'^2 \\
	&\quad + 3398996511721395 \lambda'^3 + 19562934790715805 \lambda'^4 + 84465217343350215 \lambda'^5 \\
	&\quad + 282954355636821120 \lambda'^6 + 750471409058839200 \lambda'^7 + 1595393198393421600 \lambda'^8 \\
	&\quad + 2739941670800304000 \lambda'^9 + 3824934893341125150 \lambda'^{10} + 4365794117140792650 \lambda'^{11} \\
	&\quad + 4095381088507508430 \lambda'^{12} + 3169663057350540810 \lambda'^{13} + 2028122630154888540 \lambda'^{14} \\
	&\quad + 1072191530515372020 \lambda'^{15} + 466441978900251510 \lambda'^{16} + 165579754975106706 \lambda'^{17} \\
	&\quad + 47280452455882737 \lambda'^{18} + 10615184617538683 \lambda'^{19} + 1807244800421229 \lambda'^{20} \\
	&\quad + 219600540003767 \lambda'^{21} + 16986139201980 \lambda'^{22} + 629116266740 \lambda'^{23} \Big) \\
	&+ 4\tau \Big( -341762388996600 - 9227584502908200 \lambda' - 120628732603399605 \lambda'^2 \\
	&\quad - 1017748607586066135 \lambda'^3 - 6231048285172596795 \lambda'^4 - 29493954939551352705 \lambda'^5 \\
	&\quad - 112201268013930371130 \lambda'^6 - 351779772255224704830 \lambda'^7 - 924403804975173772350 \lambda'^8 \\
	&\quad - 2060549412110237587770 \lambda'^9 - 3935320391233862807235 \lambda'^{10} \\
	&\quad - 6504356287497196991025 \lambda'^{11} - 9397112786689621100445 \lambda'^{12} \\
	&\quad - 11956214777840037946215 \lambda'^{13} - 13433329284957200958300 \lambda'^{14} \\
	&\quad - 13298390490379739071860 \lambda'^{15} - 11527327327449073892700 \lambda'^{16} \\
	&\quad - 8663387548776382562004 \lambda'^{17} - 5561934075180959620203 \lambda'^{18} \\
	&\quad - 2975925737495445613385 \lambda'^{19} - 1263239403297325468677 \lambda'^{20} \\
	&\quad - 371477202277838945535 \lambda'^{21} - 28095661113751552698 \lambda'^{22} \\
	&\quad + 47901766779748831170 \lambda'^{23} + 36943173622406365938 \lambda'^{24} \\
	&\quad + 16599828097336996694 \lambda'^{25} + 5360641053122110275 \lambda'^{26} \\
	&\quad + 1290977972998108785 \lambda'^{27} + 228926662675563645 \lambda'^{28} \\
	&\quad + 28445896261120455 \lambda'^{29} + 2221617594950640 \lambda'^{30} \\
	&\quad + 82282133146320 \lambda'^{31} \Big) \\
	&+ (1+\lambda')^2 \Big( 13041534430888713750 + 326038360772217843750 \lambda' \\
	&\quad + 3906100966830627327555 \lambda'^2 + 29836545130144371688875 \lambda'^3 \\
	&\quad + 163060144457147799003750 \lambda'^4 + 678083756640387959907750 \lambda'^5 \\
	&\quad + 2226983716260408450982125 \lambda'^6 + 5914318545387113631081525 \lambda'^7 \\
	&\quad + 12888116239188595870452240 \lambda'^8 + 23215024228522798807362000 \lambda'^9 \\
	&\quad + 34560703409386043167181875 \lambda'^{10} + 42066559107685396799906875 \lambda'^{11} \\
	&\quad + 40535299607122904954691090 \lambda'^{12} + 28081242988789780820210450 \lambda'^{13} \\
	&\quad + 8405792195492415791028965 \lambda'^{14} - 10692238540028126846701875 \lambda'^{15} \\
	&\quad - 21895648772218246431883760 \lambda'^{16} - 22587426314917455615528656 \lambda'^{17} \\
	&\quad - 15615190760104063466027055 \lambda'^{18} - 6532379832638700039885687 \lambda'^{19} \\
	&\quad + 107263715325420943150650 \lambda'^{20} + 2751865607055133732802298 \lambda'^{21} \\
	&\quad + 2433708085351280278653375 \lambda'^{22} + 1068082448523480183719319 \lambda'^{23} \\
	&\quad + 17058743862360269088600 \lambda'^{24} - 337823454744207386737000 \lambda'^{25} \\
	&\quad - 241707877644379888578775 \lambda'^{26} - 46917387261185085252735 \lambda'^{27} \\
	&\quad + 63629107019934812134750 \lambda'^{28} + 78536511796721594319070 \lambda'^{29} \\
	&\quad + 51391005423286035291375 \lambda'^{30} + 23676284280462626027495 \lambda'^{31} \\
	&\quad + 8164515502347735724850 \lambda'^{32} + 2123570523688641557250 \lambda'^{33} \\
	&\quad + 407795286522667306000 \lambda'^{34} + 54859320646265626000 \lambda'^{35} \\
	&\quad + 4633114084778718000 \lambda'^{36} + 185324563391148720 \lambda'^{37} \Big) \bigg\},
\end{align*}
and
\begin{align*}
		&J_{2}(\lambda')=-\frac{\pi \lambda'^4}{17272430470872694784 (\lambda'+1)^{27}} \cdot \\
	&\bigg\{ -4\tau^2 \Big( 907449370971 \lambda'^{23} + 24501133016217 \lambda'^{22} + 317128820014786 \lambda'^{21} \\
	&\quad + 2616869861797230 \lambda'^{20} + 15441084103183635 \lambda'^{19} + 69253251122259945 \lambda'^{18} \\
	&\quad + 244921497216685380 \lambda'^{17} + 699165728774667180 \lambda'^{16} + 1635177202827708150 \lambda'^{15} \\
	&\quad + 3161117253853055250 \lambda'^{14} + 5072563388098755900 \lambda'^{13} + 6757967510022857700 \lambda'^{12} \\
	&\quad + 7450300450223121150 \lambda'^{11} + 6754441199987677050 \lambda'^{10} + 4994902610275536000 \lambda'^9 \\
	&\quad + 2988014480514272160 \lambda'^8 + 1435851850681159695 \lambda'^7 + 550078623003941685 \lambda'^6 \\
	&\quad + 166093871425096050 \lambda'^5 + 38770989509951550 \lambda'^4 + 6770196875084175 \lambda'^3 \\
	&\quad + 834792678395325 \lambda'^2 + 64982989457100 \lambda' + 2406777387300 \Big) \\
	&+ 4\tau \Big( 25176829023450 \lambda'^{31} + 679774383633150 \lambda'^{30} + 8737954012979685 \lambda'^{29} \\
	&\quad + 70966174588807095 \lambda'^{28} + 406874784486378585 \lambda'^{27} + 1737492362727846075 \lambda'^{26} \\
	&\quad + 5645680993281843150 \lambda'^{25} + 13778739736524178650 \lambda'^{24} + 22936247183467295236 \lambda'^{23} \\
	&\quad + 12285556551693069372 \lambda'^{22} - 77823383192067445569 \lambda'^{21} \\
	&\quad - 367580212325623937835 \lambda'^{20} - 1041775096047772408153 \lambda'^{19} \\
	&\quad - 2310499173730864313691 \lambda'^{18} - 4302493387986713255388 \lambda'^{17} \\
	&\quad - 6897101846216536031220 \lambda'^{16} - 9581191733923263692550 \lambda'^{15} \\
	&\quad - 11504912949178634423490 \lambda'^{14} - 11855973756893317537005 \lambda'^{13} \\
	&\quad - 10388646744494852409615 \lambda'^{12} - 7666540617069416748825 \lambda'^{11} \\
	&\quad - 4725746641416122175675 \lambda'^{10} - 2419604617741760772450 \lambda'^9 \\
	&\quad - 1026327069981064187190 \lambda'^8 - 360260032082186031480 \lambda'^7 \\
	&\quad - 104450349812996894040 \lambda'^6 - 24872851751637583575 \lambda'^5 \\
	&\quad - 4798764883345287645 \lambda'^4 - 728763773339206215 \lambda'^3 \\
	&\quad - 82280810684309445 \lambda'^2 - 6151723001938800 \lambda' - 227841592664400 \Big) \\
	&+ (\lambda'+1)^2 \Big( 218964753992804820 \lambda'^{37} + 5474118849820120500 \lambda'^{36} \\
	&\quad + 64825444129007936000 \lambda'^{35} + 482019382462613336000 \lambda'^{34} \\
	&\quad + 2511468397632633820125 \lambda'^{33} + 9665410628163293264725 \lambda'^{32} \\
	&\quad + 28079761639847865477705 \lambda'^{31} + 61168734465597319516625 \lambda'^{30} \\
	&\quad + 94254468871874636981825 \lambda'^{29} + 78692455703482359701625 \lambda'^{28} \\
	&\quad - 50179637444926972278335 \lambda'^{27} - 285986916888895778309975 \lambda'^{26} \\
	&\quad - 428721025909210706632675 \lambda'^{25} - 76674077728927168416875 \lambda'^{24} \\
	&\quad + 1091530949349340952792925 \lambda'^{23} + 2782195864453471458563525 \lambda'^{22} \\
	&\quad + 3718721457358679432195373 \lambda'^{21} + 2041707059438566492224325 \lambda'^{20} \\
	&\quad - 3335754154031378705443995 \lambda'^{19} - 11161964391330220945275475 \lambda'^{18} \\
	&\quad - 17504967370718691296283837 \lambda'^{17} - 17590036951058053776464885 \lambda'^{16} \\
	&\quad - 9045200156393057112269125 \lambda'^{15} + 5848183059559680903416275 \lambda'^{14} \\
	&\quad + 21027014623584811983171875 \lambda'^{13} + 30316355431675283552949835 \lambda'^{12} \\
	&\quad + 31087544138241017880140675 \lambda'^{11} + 25138341469519077359455675 \lambda'^{10} \\
	&\quad + 16600194272471682970061375 \lambda'^9 + 9064786653220893246169575 \lambda'^8 \\
	&\quad + 4098575687974128123178975 \lambda'^7 + 1524164923780960111723735 \lambda'^6 \\
	&\quad + 459540308786969865622875 \lambda'^5 + 109706372544630548879875 \lambda'^4 \\
	&\quad + 19975192434374721412375 \lambda'^3 + 2607433805124890496495 \lambda'^2 \\
	&\quad + 217358907181478562500 \lambda' + 8694356287259142500 \Big) \bigg\}.
\end{align*}
 \end{prop}   
\begin{lem}\label{3.5 lem}
	Let $n=25$. Then there exists $\tau_{25}\in\mathbb{R}$ such that $I(1)<0$, $I'(1)=0$, $I''(1)>0$, $J_{1}(1)>0$, and $J_{2}(1)>0$.
 \end{lem} 
 \begin{proof}
A direct computation shows that
 \begin{align*}
	I'(1)
	=&-\frac{7 \pi  \left(5992174700937 \tau^2+32957366474295072  \tau+45136950172205187200\right)}{8576655130364372582400}.
\end{align*}
Solving the equation $I'(1)=0$, we select the root
\begin{align*}
	\tau_{25}=-\frac{8\left(\sqrt{1872413768932826265919670946}+686611801547814\right)}{1997391566979}\approx-2923.3455.
\end{align*}
Then we have
 \begin{align*}
	I(1) |_{\tau_{25}}
	=& -\frac{\pi  \left(22101086455227  \tau_{25}^2+96208621810778160  \tau_{25}+109176791784653017600\right)}{12864982695546558873600}
	\\  \approx&-0.0041<0,
\end{align*}   
and
\begin{align*}
	I'' (1 )|_{\tau_{25}}
	=&-\frac{7 \pi  \left(27569429864823  \tau_{25}^2+222733074462954048  \tau_{25}+403655945400788504192\right)}{25729965391093117747200}
	\\\approx &0.0101>0.
\end{align*}  
 
On the other hand,
\begin{align*}
	J_{1}(1)|_{\tau_{25}}
	= &\frac{\pi\left(172831653637\tau_{25}^2+306674306360544\tau_{25}-393500867003307648\right)}{2382404202878992384}\\
	\approx &0.2466>0,
\end{align*}
and
\begin{align*}	 J_{2}(1) |_{\tau_{25}} 
  =&\frac{\pi  \left(4878475952883 \tau_{25}^2+6070206880074240 \tau_{25}-8380883301904529792\right)}{228710803476383268864}
	\\\approx &0.2138>0.
 \end{align*} 
 \end{proof}
Analogously to Proposition \ref{fourth prop}, we obtain the following.
\begin{prop}\label{3.5 prop}
	Let $n=25$, and let $\tau_{25}$ be as in \eqref{tau25} and Lemma \ref{3.5 lem}. Then the function $F(\xi',\lambda')$ has a strict local minimum at $(0,1)$.
 \end{prop}

\subsection{An Auxiliary Estimate for the Energy Expansion}\label{Sect. Auxiliary Est}
In this appendix subsection, we provide an auxiliary estimate for the energy expansion used in the proof of Proposition \ref{energy expansion}.
\begin{prop}\label{Brencdle-Marques 4 poly}
	Suppose $n\geq 25$. Then
 \begin{align*} 
 \big|&\int _{\mathbb{R}^{n}} \tilde{u}_{0} ^{  \frac{n-2}{n-4}}     L_{\tilde{g}_{0}} \tilde{u}_{0} ^{  \frac{n-2}{n-4}}   -\tilde{u}_{0} ^{  \frac{n-2}{n-4}} (-\Delta \tilde{u}_{0} ^{  \frac{n-2}{n-4}} )
 -\frac{1}{2}\int _{B_{\frac{\rho}{\epsilon}} }\sum_{i,k,l }  \tilde{h}_{il}\tilde{h}_{kl}\partial_{i}\tilde{u}_{0} ^{  \frac{n-2}{n-4}}\partial_{k}\tilde{u}_{0} ^{  \frac{n-2}{n-4}} 
\\ &+\frac{n-2}{16(n-1)}\int _{B_{\frac{\rho}{\epsilon}} }\sum_{i,k,l }  (\partial_{l}\tilde{h}  _{ik})^{2} \tilde{u}_{0} ^{  \frac{2n-4}{n-4}}\big| 
 \\= &O( \mu^3 \varepsilon^{23} \rho^ 7|\log \epsilon|)+O\Big(\alpha  (\frac{\epsilon}{\rho} )^{n-4} \Big)  . 
 \end{align*}
\end{prop}
\begin{proof}
We estimate
 \begin{align*} 
	& \int _{\mathbb{R}^{n}} |d\tilde{u}_{0} ^{  \frac{n-2}{n-4}}   |_{\tilde{g}_{0}} ^{2}     - |d\tilde{u}_{0} ^{  \frac{n-2}{n-4}}  |  ^{2}    
+\sum_{i,k} \tilde{h}_{ik}   \partial_{i} \tilde{u}_{0} ^{  \frac{n-2}{n-4}} \partial_{k} \tilde{u}_{0} ^{  \frac{n-2}{n-4}}
	 -	\frac{1}{2}\int _{B_{\frac{\rho}{\epsilon}} }\sum_{i,k,l }  \tilde{h}_{il}\tilde{h}_{kl}\partial_{i}\tilde{u}_{0} ^{  \frac{n-2}{n-4}}\partial_{k}\tilde{u}_{0} ^{  \frac{n-2}{n-4}}  
	 \\\leq  & \int _{B_{\frac{\rho}{\epsilon}}}  \big| ( g^{ik}-\delta_{ik}+\tilde{h}_{ik}-\frac{1}{2} \sum_{ l }  \tilde{h}_{il}\tilde{h}_{kl}) \partial_{i} \tilde{u}_{0} ^{  \frac{n-2}{n-4}} \partial_{k} \tilde{u}_{0} ^{  \frac{n-2}{n-4}} \big| +	 \int _{\mathbb{R}^{n}\setminus B_{\frac{\rho}{\epsilon}} }  \big| ( g^{ik}-\delta_{ik}+\tilde{h}_{ik} )\partial_{i} \tilde{u}_{0} ^{  \frac{n-2}{n-4}} \partial_{k} \tilde{u}_{0} ^{  \frac{n-2}{n-4}} \big| 
	 	 \\\leq  & \int _{B_{\frac{\rho}{\epsilon}}} O( \tilde{h}^{3}) \frac{1}{(1+|y-\xi '|)^{2n-2 }  }+	 \int _{\mathbb{R}^{n}\setminus B_{\frac{\rho}{\epsilon}} }  O( \tilde{h}^{2}) \frac{1}{(1+|y-\xi '|)^{2n-2 }  } 
	\\\leq &C \mu^3 \varepsilon^{23} \rho^ 7|\log \epsilon|+C \alpha^2  \left(\frac{\epsilon}{\rho}\right)^{n-2},
\end{align*}
and 
 \begin{align*} 
 & \int _{\mathbb{R}^{n}}   R_{\tilde{g}_{0}}  \tilde{u}_{0} ^{  \frac{2n-4}{n-4}}  +\frac{1}{4}\sum_{i,k,l} \int _{B_{\frac{\rho}{\epsilon}} }(\partial_{l}\tilde{h}_{ik})^{2}   \tilde{u}_{0} ^{  \frac{2n-4}{n-4}} 
 \\	\leq& \int _{\mathbb{R}^{n}} \big| [ R_{\tilde{g}_{0}}  +\frac{1}{4}\sum_{i,k,l} (\partial_{l}\tilde{h}_{ik})^{2} ]  \tilde{u}_{0} ^{  \frac{2n-4}{n-4}}\big| 
+\int _{\mathbb{R}^{n}\setminus B_{\frac{\rho}{\epsilon}} }   |  R_{\tilde{g}_{0}}  |  \tilde{u}_{0} ^{  \frac{2n-4}{n-4}} 
\\	\leq& \int _{B_{\frac{\rho}{\epsilon}}} (|\tilde{h}|^{2}|\partial^{2}\tilde{h}|+|\tilde{h}|  |\partial\tilde{h}|^{2})  \frac{1}{(1+|y-\xi '|)^{2n-4 }  }
+\alpha^{2} \epsilon^{2}\int _{\mathbb{R}^{n}\setminus B_{\frac{\rho}{\epsilon}} }     \frac{1}{(1+|y-\xi '|)^{2n-4 }  }  
\\\leq &C \mu^3 \varepsilon^{23} \rho^ 7|\log \epsilon|+C \alpha^2 \epsilon^{2}\left(\frac{\epsilon}{\rho}\right)^{n-4} .   
\end{align*}	 
On the other hand, note that
 \begin{align*} 
	&\partial_{i}\tilde{u}_{0} ^{\frac{n-2}{n-4}}\partial_{k}\tilde{u}_{0} ^{\frac{n-2}{n-4}}- \frac{n-2}{4(n-1)}\partial_{i} \partial_{k}\tilde{u}_{0} ^{\frac{2n-4}{n-4}}=\frac{1}{n}\big(|d\tilde{u}_{0} ^{\frac{n-2}{n-4}}|^{2}- \frac{n-2}{4(n-1)}\Delta \tilde{u}_{0} ^{\frac{2n-4}{n-4}}\big)\delta_{ik},
\end{align*}	
Since $\tilde{h}$ is trace-free, we have
\begin{align*} 
	&   \int _{\mathbb{R}^{n}}  \tilde{h}_{ik} \partial_{i}\tilde{u}_{0} ^{  \frac{n-2}{n-4}}\partial_{k}\tilde{u}_{0} ^{  \frac{n-2}{n-4}} =  \int _{\mathbb{R}^{n}}  \frac{n-2}{4(n-1)} \tilde{h}_{ik}\partial_{i} \partial_{k}\tilde{u}_{0} ^{\frac{2n-4}{n-4}}
 =   \int _{\mathbb{R}^{n}}  \frac{n-2}{4(n-1)} \partial_{i} \partial_{k}\tilde{h}_{ik}\tilde{u}_{0} ^{\frac{2n-4}{n-4}}.
\end{align*}
Using $\sum_{i}\partial_{i}\tilde{h}_{ik}=0$ in $B_{\frac{\rho}{\epsilon}}$, we obtain
\begin{align*} 
	& \big|  \int _{\mathbb{R}^{n}}  \tilde{h}_{ik} \partial_{i}\tilde{u}_{0} ^{  \frac{n-2}{n-4}}\partial_{k}\tilde{u}_{0} ^{  \frac{n-2}{n-4}} \big| 
\leq C\alpha\epsilon^{2}\int _{\mathbb{R}^{n}\setminus B_{\frac{\rho}{\epsilon}} }    \tilde{u}_{0} ^{\frac{2n-4}{n-4}}\leq C\alpha \epsilon^{2}(\frac{\epsilon}{\rho}  )^{n-4}.
\end{align*}
Combining the above estimates, we obtain the desired result.
 \end{proof}
  
 \section{An Eigenvalue Problem}\label{Sect. eigenvalue problem}

In Appendix B, we discuss an eigenvalue problem \eqref{eigen eq 1}, which is used in Section \ref{Sect. linear problem}.

Denote
 \begin{align}  \label{U def} 
 	U (x): =\alpha_{n} (  1+|x  |^{2}  )^{-\frac{n-4}{2}} ,\ \   x\in\mathbb{R}^{n}\  ,\mathrm{where}  \  \alpha_{n} =2 ^{ \frac{n-4}{ 2} }.
 \end{align}	
We consider the following eigenvalue problem for $v\in H=\{v\in W^{2,2}_{\text{loc}}:\int |\nabla^2 v|^2<\infty\}$:
\begin{align}\label{eigen eq 1}
	\Delta^{2} v  +\mu\frac{ (n^{2}-4 )(n-4)}{4n} [U^{\frac{2}{ n-4 }}  \Delta ( U^{\frac{2}{ n-4 }}v)+\frac{2}{n-2}(U^{\frac{6-n}{ n-4 }}\Delta U^{\frac{n-2}{ n-4 }} )v]=0   \ \ \text{in} \  \mathbb{ R}^{n}.
\end{align} 
Since
\begin{align}\label{U n-2 n-4 U 8}
	\Delta U^{\frac{n-2}{ n-4 }} =-\frac{n(n-2)}{4}U^{\frac{n+2}{ n-4 }},
\end{align}
 equation \eqref{eigen eq 1} is equivalent to
\begin{align} \label{eigen eq} 
	\Delta^{2} v + \frac{ (n^{2}-4 )(n-4)}{4n} \mu[U^{\frac{2}{ n-4 }}  \Delta ( U^{\frac{2}{ n-4 }}v)-\frac{n}{2}U^{  \frac{8}{n-4}} v]=0  \ \ \text{in} \  \mathbb{ R}^{n}.
\end{align}

For any $\xi\in\mathbb{R}^n$ and $\lambda\in\mathbb{R}_+$, define
 \begin{align}\label{U xi  lambda def}
  U_{\xi,\lambda}:= \lambda^{-\frac{n-4}{2}}U\big(\lambda^{-1} (x-\xi) \big).
 \end{align} 

We refer to \cite{LuWei2000,BartschWethWillem2003} for related eigenvalue problems for other equations; the proofs there rely on analyzing the second derivatives of the functional associated with the sharp Sobolev inequalities. To determine the eigenvalues of \eqref{eigen eq}, we rewrite the equation on the round sphere $\mathbb{S}^n$. Using the spectrum of the Laplace operator on $\mathbb{S}^n$, we derive the corresponding eigenvalues of \eqref{eigen eq}. We refer to \cite{ORey} for related computations for the Laplacian.

We now state the following eigenvalue theorem, which is used in the linearized problem \eqref{linear equ zeta z}.
\begin{thm}\label{eigenvalue thm}
The eigenvalues $\mu$ of \eqref{eigen eq} are discrete and given by
\[ 
\frac{\left(\lambda_k+\frac{n(n-2)}{4}\right)\left(\lambda_k+\frac{(n+2)(n-4)}{4}\right)}{\frac{\left(n^2-4\right)(n-4)}{4 n}\left(\lambda_k+\frac{n^2}{4}\right)},
\] 
where $\lambda_k = k(k+n-1)$ are the eigenvalues of $-\Delta_{\mathbb S^n}$. The corresponding eigenspaces consist of the projections onto $\mathbb{R}^{n+1}$ of the spaces of spherical harmonics of degree $k$ on $\mathbb S^n$.

In particular, when $\mu=\frac{n}{n-4}$ (corresponding to $\lambda_1=n$), the eigenspace is spanned by
\[
\left\{\partial_{\lambda} U_{\xi,\lambda},\, \partial_{\xi_j} U_{\xi,\lambda} \;:\; j=1,\ldots,n \right\}.
\]
\end{thm}

\begin{proof}
Let $\Phi:\ss^n\setminus\{N\}\to\R^n$ be the stereographic projection from the north pole $N$.
In coordinates $y=\Phi(x)$, the round metric satisfies
\begin{equation}\label{eq:metric}
(\Phi^{-1})^{*}g_{\ss^{n}}=U^{\frac{4}{n-4}}(y) \,|dy|^{2},
\qquad 
U^{\frac{2}{n-4}}(y):=\frac{2}{1+{|y|}^{2}}.
\end{equation}

Given any function $v:\R^n\to\R$, define $\varphi$ on $\ss^n\setminus \{N\}$ by
\begin{equation}\label{eq:phi-def}
\varphi(x):=U^{-1}(y)\,v(y),\qquad y=\Phi(x).
\end{equation}
Equivalently,
\begin{equation}\label{eq:v-phi}
v(y)=U \varphi\circ\Phi^{-1}(y) .
\end{equation}

Let $P_{\ss^n}$ denote the Paneitz operator on the round sphere; that is,
\begin{equation}\label{eq:P-factor}
P_{\Sp^n}
=\left(-\Delta_{\Sp^n}+\frac{n(n-2)}{4}\right)
 \left(-\Delta_{\Sp^n}+\frac{(n+2)(n-4)}{4}\right).
\end{equation}
Since $P_{\R^n}=\Delta^2$ and $(\Phi^{-1})^*g_{\ss^n}=U^{\frac{4}{n-4}}\delta$, the conformal covariance yields
\begin{equation}\label{eq:bilap-cov}
(P_{\ss^n}\varphi)\circ\Phi^{-1}=U^{-\frac{n+4}{n-4}}\Delta^2 \left(U\varphi\circ\Phi^{-1}\right). 
\end{equation}
Using \eqref{eq:v-phi}, this becomes
\begin{equation}\label{eq:bilap-transform}
\Delta^{2}v(y)=U^{\frac{n+4}{n-4}}\,(P_{\ss^n}\varphi)\big(\Phi^{-1}(y)\big).
\end{equation}

Let $L_{\mathbb{S}^n} := -\Delta_{\mathbb{S}^n} + \frac{n(n-2)}{4}$ denote the conformal Laplacian on $\mathbb{S}^n$. Its conformal covariance is given by
\begin{equation}\label{eq:conf-lap-cov}
L_{\ss^n}\varphi\circ\Phi^{-1}
=(U^{\frac{n-2}{n-4}})^{-\frac{n+2}{n-2}}\Big(-\Delta\big(U^{\frac{n-2}{n-4}}\varphi\circ\Phi^{-1}\big)\Big).
\end{equation}
By \eqref{eq:v-phi}, we have $U^{\frac{2}{n-4}} v=U^{\frac{n-2}{n-4}}\varphi\circ\Phi^{-1}$. Hence
\begin{equation}\label{eq:OmegaDeltaOmegav}
U^{\frac{2}{n-4}}\,\Delta(U^{\frac{2}{n-4}} v)
=-U^{\frac{n+4}{n-4}}\,(L_{\ss^n}\varphi)\circ\Phi^{-1}.
\end{equation}
Moreover,
\begin{equation}\label{eq:Omega4v}
U^{\frac{8}{n-4}}v
=U^{\frac{n+4}{n-4}}\varphi\circ\Phi^{-1}.
\end{equation}

Therefore, substituting \eqref{eq:bilap-transform}--\eqref{eq:Omega4v} into \eqref{eigen eq} and noting that $\int_{\mathbb R^n} |\nabla^2 v|^2\,dx<\infty$, elliptic regularity yields the equivalent equation on $\ss^n$:
\begin{equation}\label{eq:sphere}
P_{\ss^n}\varphi
+\frac{(n^{2}-4)(n-4)}{4n}\,\mu
\left(\Delta_{\ss^n}\varphi-\frac{n^{2}}{4}\varphi\right)=0
\qquad\text{on }\ss^n.
\end{equation}

Since the eigenfunctions of $-\Delta_{\ss^n}$ span $W^{2,2}(\mathbb S^n)$ and are also eigenfunctions of \eqref{eq:sphere}, the result follows.

When the eigenvalue of \eqref{eigen eq} is $\mu=\frac{n}{n-4}$, the corresponding eigenvalue of \eqref{eq:sphere} is $\lambda_1 = n$.
The associated eigenfunctions on $\mathbb{S}^n$ are $\{x_1,\dots,x_{n+1}\}$.
Via the stereographic projection \eqref{eq:v-phi}, these correspond in $\mathbb{R}^n$ to
\[
\left\{\partial_{\lambda} U_{\xi,\lambda},\, \partial_{\xi_j} U_{\xi,\lambda} \;:\; j=1,\ldots,n \right\}.
\]
\end{proof}

\begin{cor} For any function $\varphi\in C^{\infty}(\mathbb S^n)$, we have
    \[\frac{\int_{\mathbb S^n} \varphi P_{g_{\mathbb S^n}}\varphi dv_{g_{\mathbb S^n}}}{\int_{\mathbb S^n}|\nabla_{\ss^n}\varphi|^2+\frac{n^{2}}{4}\varphi^2 dv_{g_{\mathbb S^n}}}\ge \frac{n^2-4}{4}.\]
\end{cor}

We also provide a proof of the eigenvalue problem used in the proof of Lemma \ref{est linear equ zeta}.

We begin by recalling that
	\begin{align}  \label{eigen U} 
		\Delta^{2} U + \frac{ (n^{2}-4 )(n-4)}{4n}  (U^{\frac{2}{ n-4 }}  \Delta   U^{\frac{n-2 }{ n-4 }} -\frac{n}{2}U^{  \frac{8}{n-4}} U)=0,
	\end{align} 
and
	\begin{align} \label{eigen Z j} 
		\Delta^{2} Z_{j}+ \frac{  n^{2}-4  } {4 }  [U^{\frac{2}{ n-4 }}  \Delta (  U^{\frac{2}{ n-4 }}Z_{j})-\frac{n}{2}U^{  \frac{8}{n-4}} Z_{j}]=0, \ \ j=0,1,\ldots,n.
	\end{align} 
	where $Z_{0}=\partial_{\lambda}U_{\xi,\lambda}$ and $Z_{j}=\partial_{\xi_{j}}U_{\xi,\lambda}$ for $j=1,\ldots,n$ and $\xi=(\xi_{1},\ldots,\xi_{n})\in \mathbb{R}^n$.
\begin{thm}\label{eigenvalue thm2}
		The eigenvalues $\mu$ of \eqref{eigen eq} are discrete and
		\begin{align*}    
			\mu_{1}=1,\mu_{2}=\mu_{3}=\cdots=\mu_{n+2}=\frac{n}{n-4}<\mu_{n+3} .
		\end{align*}	
		The corresponding linearly independent eigenvectors are given by
		\begin{align*}    
			v_{1}=U, v_{j+2}= Z_j,\ \ j=0,1,\ldots,n.
		\end{align*}	
	\end{thm}
	\begin{proof}
		We follow the argument given in \cite{BartschWethWillem2003}.

		Consider the functional
		\begin{align*}
			Q(u)=\frac{1}{2}	\|u\|^{2}-\frac{ n-4 }{2( n-2) }  \|u\|_{*}^{\frac{2(n-2)}{n-4}},
		\end{align*}
		that is
		\begin{align*}
			Q(u)= \frac{1}{2} \int_{\mathbb{R}^{n }} (\Delta u)^{2}-\frac{ n-4  }{2( n-2) }  \int_{\mathbb{R}^{n }}\frac{ (n+2 )(n-4 )}{4} |\nabla u ^{\frac{n-2}{n-4}}|^{2} .
		\end{align*}
		Then $U$ minimizes the functional $Q(u)$ on the manifold
		\begin{align*}
			\mathcal{N}:=\{  u\in\mathcal{D}_0^{2,2}(\mathbb{R}^{n } ) \cap C^2(\mathbb R^n)\setminus \{0\}:    	\|u\|^{2}= \|u\|_{*}^{\frac{2(n-2)}{n-4}},u >0\,\,\text{and}\,\,R_{g_u}>0  \},
		\end{align*} 
where $\mathcal{D}_0^{2,2}(\mathbb{R}^n)$ is defined as the completion of $C_0^\infty(\mathbb{R}^n)$ under the norm $\|\cdot\|$, and the norms $\|\cdot\|$ and $\|\cdot\|_{*}$ are given by
\begin{align}\label{norm def}
\|u\| &:= \left( \int_{\mathbb{R}^n} (\Delta u)^2   \right)^{\frac{1}{2}}, \\
\|u\|_{*} &:= \left( \int_{\mathbb{R}^n} \frac{(n+2)(n-4)}{4} \, \left| \nabla u^{\frac{n-2}{n-4}} \right|^2  \right)^{\frac{n-4}{2(n-2)}}.
\end{align}
		Indeed, since locally
		\begin{align}  \label{best soblev  const}
			\frac{	\|u\|^{2}}{\|u\|_{*}^{2}}\geq 
			\frac{	\|U\|^{2}}{\|U\|_{*}^{2}} ,
		\end{align}
		we have
		\begin{align*}
			Q(u)= &\big[\frac{1}{2}  -\frac{ n-4  }{2( n-2) }   \big]\|u\|_{*}^{\frac{2( n-2)}{n-4}} 
			\\= &\big[\frac{1}{2}  -\frac{ n-4 }{2( n-2) }    \big](\|u\|_{*}^{\frac{2 }{n-4}} )^{ n-2 }
			\\=&\big[\frac{1}{2}  -\frac{ n- 4  }{2( n-2) }  \big]\Big(\frac{	\|u\|}{\|u\|_{*}}\Big)^{ n-2 } 
			\\\geq  &\big[\frac{1}{2}  -\frac{ n-4 }{2( n-2) }   \big]\Big(\frac{	\|U\|}{\|U\|_{*}}\Big)^{ n-2 }  
			\\=& \big[\frac{1}{2}  -\frac{ n-4  }{2( n-2) }  \big]\|U\|_{*}^{\frac{2( n-2)}{n-4}}=Q(U).
		\end{align*}
		Hence, the second derivative $d^{2} Q(U)$ is given by
		\begin{align*}
			(v,w)\mapsto\int_{\mathbb{R}^{n }}\Delta v\Delta w +\frac{ n^{2}-4  }{4}\int_{\mathbb{R}^{n }}U^{\frac{2}{ n-4 }}v \Delta ( U^{\frac{2}{ n-4 }}w)+\frac{2}{n-2}(U^{\frac{6-n}{ n-4 }}\Delta U^{\frac{n-2}{ n-4 }} )vw,
		\end{align*}
		which is a non-negative quadratic form when restricted to the tangent space $T_{U}\mathcal{N}$. Since $T_{U}\mathcal{N}$ has codimension one, we deduce that $\mu_{2}\geq \frac{n}{n-4}$. By \eqref{eigen Z j}, we have $\mu_{2}=\frac{n}{n-4}$, and $Z_j$ ($j=0,1,\ldots,n$) belong to the corresponding eigenspace $V_{2}$. Moreover, by \eqref{eigen U}, we have $\mu_{1}=1$ with the corresponding eigenfunction $U$, which implies
		\begin{align*}
			\int_{\mathbb{R}^{n }}(\Delta v)^{2}  \geq \frac{ (n^{2}-4 )(n-4)}{4n} \int_{\mathbb{R}^{n }} -  U^{\frac{2}{ n-4 }}v  \Delta ( U^{\frac{2}{ n-4 }}v) -\frac{2}{n-2}(U^{\frac{6-n}{ n-4 }}\Delta U^{\frac{n-2}{ n-4 }} )v^{2}.
		\end{align*}
		That is, by \eqref{U n-2 n-4 U 8},  
		\begin{align*} 
			\int_{\mathbb{R}^{n }}(\Delta v)^{2}  \geq  \frac{ (n^{2}-4 )(n-4)}{4n}  \int_{\mathbb{R}^{n }}  |\nabla( U^{\frac{2}{ n-4 }}v)|^{2} +\frac{n}{2} U^{\frac{8}{ n-4 }}  v^{2}.
		\end{align*}
		More precisely, we have the key inequality
		\begin{align}\label{key eigen ineq}
			\int_{\mathbb{R}^{n }}(\Delta v)^{2}  \geq  \frac{ (n^{2}-4 )(n-4)}{4n}  \int_{\mathbb{R}^{n }}  |\nabla( \frac{2}{1+|x|^{2}}v)|^{2} +\frac{n}{2}\frac{16}{(1+|x|^{2})^{4}} v^{2}
		\end{align}
		for any $v \in\mathcal{D}_0^{2,2}(\mathbb{R}^{n } )$.

		It suffices to prove that the dimension of $V_{2}$ does not exceed $n+1$. Denote the eigenvalues of degree $k$ of the Laplace--Beltrami operator on $\mathbb{S}^{n-1}$ by
		\begin{align*}
			\lambda_{k}=k(k+n-2), \ \ k=0,1,2,\ldots
		\end{align*}	
		and let $\Psi_{k,i}(\theta)$ be the corresponding spherical harmonics on $\mathbb{S}^{n-1}$. For any $v(x)=v(r,\theta)\in V_{2}$ with $r\geq 0$ and $\theta\in\mathbb{S}^{n-1}$, it is clear that
		\begin{align} \label{mu2 eigen eq} 
			\Delta^{2} v + \frac{  n^{2}-4  }{4 } [U^{\frac{2}{ n-4 }} \Delta (U^{\frac{2}{ n-4 }}v)-\frac{n}{2}U^{  \frac{8}{n-4}} v]=0  \ \ \text{in} \  \mathbb{ R}^{n}.
		\end{align}
		Now set
		\begin{align*}    
			\phi_{k,i}(r)= \int_{\mathbb{S}^{n-1}} v(x)\Psi_{k,i} (\theta)d\theta,
		\end{align*}		 
		and
		\begin{align}\label{definition of w} 
			w_{k,i}(r)=-\int_{\mathbb{S}^{n-1}} \Delta v(x)\Psi _{k,i} (\theta)d\theta.
		\end{align}	
		First, we compute
		\begin{align*}    
			\Delta \phi_{k,i} =&  \int_{\mathbb{S}^{n-1}}   \big( \Delta v-\frac{1}{r^{2}}\Delta _{\mathbb{S}^{n-1}} v \big) \Psi  _{k,i} (\theta)d\theta   
			=   \int_{\mathbb{S}^{n-1}} \Delta v (x)\Psi  _{k,i} (\theta)d\theta -\frac{1}{r^{2}} \int_{\mathbb{S}^{n-1}}  v\Delta _{\mathbb{S}^{n-1}}  \Psi  _{k,i} (\theta)d\theta
			\\=& - w_{k,i}  +\frac{\lambda_{i}}{r^{2}} \int_{\mathbb{S}^{n-1}}  v  \Psi  _{k,i} (\theta)d\theta
			=  -w_{k,i} +\frac{\lambda_{i}}{r^{2}} \phi_{k,i},
		\end{align*}	
		that is
		\begin{align} \label{w_ki eq}
			w_{k,i}=-(\Delta -\frac{\lambda_{k}}{r^{2}} )\phi_{k,i}.
		\end{align}	
		Next, by \eqref{definition of w} and then using \eqref{mu2 eigen eq}, we find
		\begin{align*}    
			\Delta w_{k,i} =&  -\int_{\mathbb{S}^{n-1}}   \big[ \Delta^{2} v-\frac{1}{r^{2}}\Delta _{\mathbb{S}^{n-1}} ( \Delta v)\big] \Psi  _{k,i} (\theta)d\theta   
			\\=&  -\int_{\mathbb{S}^{n-1}} \Delta^{2} v (x)\Psi  _{k,i} (\theta)d\theta +\frac{1}{r^{2}} \int_{\mathbb{S}^{n-1}}   \Delta v\Delta _{\mathbb{S}^{n-1}}  \Psi  _{k,i} (\theta)d\theta
			\\=&      \frac{  n^{2}-4  }{4 }\int_{\mathbb{S}^{n-1}}   \big[U^{\frac{2}{ n-4 }} \Delta (  U^{\frac{2}{ n-4 }}v)-\frac{n}{2}U^{  \frac{8}{n-4}} v\big]      \Psi  _{k,i} (\theta)d\theta 
			-\frac{\lambda_{k}}{r^{2}} \int_{\mathbb{S}^{n-1}}   \Delta v   \Psi  _{k,i} (\theta)d\theta
			\\=&   \frac{  n^{2}-4  }{4 } \Big\{ U^{\frac{ 2}{n-4}}  \int_{\mathbb{S}^{n-1}}    \Delta ( U^{\frac{2}{ n-4 }}v)\Psi  _{k,i} (\theta)d\theta   -\frac{n}{2}U^{  \frac{8}{n-4}} \int_{\mathbb{S}^{n-1}}v \Psi  _{k,i} (\theta)d\theta     \Big\}  +\frac{\lambda_{k}}{r^{2}} w_{k,i}
			\\=& \frac{n^{2}-4}{4}  \Big\{ U^{\frac{ 2}{n-4}}  \big[\Delta(U^{\frac{ 2}{n-4}} \phi_{k,i} )-\frac{\lambda_{k}}{r^{2}} (U^{\frac{ 2}{n-4}} \phi_{k,i} )  \big]  -\frac{ n }{2}  U^{\frac{ 8}{n-4}  }    \phi _{k,i}\Big\} 
			+\frac{\lambda_{k}}{r^{2}} w_{k,i},
		\end{align*}	 
		where we used
		\begin{align*}  
			&    \int_{\mathbb{S}^{n-1}} \Delta   ( U^{\frac{2}{ n-4 }}v)\Psi  _{k,i} (\theta)d\theta  
			\\ =&\int_{\mathbb{S}^{n-1}}    \big(  \partial_{rr}+\frac{n-1}{r}\partial_{r}\big)( U^{\frac{2}{ n-4 }}v)\Psi  _{k,i} (\theta)d\theta  +\frac{1}{r^{2}}\int_{\mathbb{S}^{n-1}}  \Delta _{\mathbb{S}^{n-1}}  ( U^{\frac{2}{ n-4 }}v)\Psi  _{k,i} (\theta)d\theta 
			\\ =&\Delta  \int_{\mathbb{S}^{n-1}}     U^{\frac{2}{ n-4 }}v \Psi  _{k,i} (\theta)d\theta  +\frac{1}{r^{2}} \int_{\mathbb{S}^{n-1}} U^{\frac{2}{ n-4 }}v \Delta _{\mathbb{S}^{n-1}}  \Psi  _{k,i} (\theta)d\theta       
			\\=& \Delta   (U^{\frac{2}{ n-4 }}\phi_{k,i}) -\frac{\lambda_{k}}{r^{2}}U^{\frac{2}{ n-4 }}\phi_{k,i}.
		\end{align*}	
		Therefore, we obtain
		\begin{align}   \label{Delta w_ki eq}
			(\Delta -\frac{\lambda_{k}}{r^{2}} ) w_{k,i}  =&\frac{n^{2}-4}{4} \Big[ U^{\frac{ 2}{n-4}}  (\Delta -\frac{\lambda_{k}}{r^{2}} )  (U^{\frac{ 2}{n-4}} \phi_{k,i} ) -\frac{ n }{2}  U^{\frac{ 8}{n-4}  }    \phi _{k,i} \Big].
		\end{align}	 
		Combining \eqref{w_ki eq} and \eqref{Delta w_ki eq}, we obtain
		\begin{align*}
			(\Delta -\frac{\lambda_{k}}{r^{2}} )^{2}\phi_{k,i}=&-\frac{n^{2}-4}{4} \Big[ U^{\frac{ 2}{n-4}}  (\Delta -\frac{\lambda_{k}}{r^{2}} )  (U^{\frac{ 2}{n-4}} \phi_{k,i} ) -\frac{ n }{2}  U^{\frac{ 8}{n-4}  }    \phi _{k,i} \Big].
		\end{align*}	 
Using the operator identity for the radial functions 
\[
\left(\Delta - \frac{\lambda_k}{r^2}\right)
= r^k \left( \frac{\partial^2}{\partial r^2}
+ \frac{n+2k-1}{r} \frac{\partial}{\partial r} \right) r^{-k},
\]
we deduce that
\begin{align*}
\left(\Delta - \frac{\lambda_k}{r^2}\right)^2
= r^k \left( \frac{\partial^2}{\partial r^2}
+ \frac{n+2k-1}{r} \frac{\partial}{\partial r} \right)^2 r^{-k}.
\end{align*}
Consequently, we obtain
		\begin{align*}
			&\big(\frac{\partial ^{2}}{\partial r^{2}}+\frac{n+2k-1}{r}\frac{\partial}{\partial r}\big)^{2} (r^{-k}\phi_{k,i})
			\\=&-\frac{n^{2}-4}{4} \Big[ U^{\frac{ 2}{n-4}}\big(\frac{\partial ^{2}}{\partial r^{2}}+\frac{n+2k-1}{r}\frac{\partial}{\partial r}\big) (U^{\frac{ 2}{n-4}} r^{-k}\phi_{k,i} ) {-}\frac{ n }{2} U^{\frac{ 8}{n-4}  }   (r^{-k}\phi_{k,i}) \Big].
		\end{align*}	

		Set $\tilde{\phi }_{k,i}(r)=r^{-k}\phi_{k,i}(r)\in C^{\infty}(0,\infty)$. Then
		\begin{align*}
			&\big(\frac{\partial ^{2}}{\partial r^{2}}+\frac{n+2k-1}{r}\frac{\partial}{\partial r}\big)^{2} \tilde{\phi }_{k,i} 
			= -\frac{n^{2}-4}{4} \Big[ U^{\frac{ 2}{n-4}}(\frac{\partial ^{2}}{\partial r^{2}}+\frac{n+2k-1}{r}\frac{\partial}{\partial r}) (U^{\frac{ 2}{n-4}} \tilde{\phi }_{k,i} ) {-}\frac{ n }{2} U^{\frac{ 8}{n-4}  }  \tilde{\phi }_{k,i} \Big].
		\end{align*}	
		Subsequently, define the functions $\Phi _{k,i}:\mathbb{R}^{n+2k}\to \mathbb{R}$ by $\Phi _{k,i}(y)=\tilde{\phi }_{k,i}(|y|)$. By \cite[Lemma 2.3]{BartschWethWillem2003} and an argument similar to the proof of \cite[Theorem 2.1]{BartschWethWillem2003}, $\Phi _{k,i}\in \mathcal{D}_0^{2,2}(\mathbb{R}^{n+2k})$ and $\Phi _{k,i}$ is a weak solution of the equation
		\begin{align*}
			\Delta_{\mathbb{R}^{n+2k}}^{2}\Phi _{k,i}=&-\frac{n^{2}-4}{4} \Big[ U^{\frac{ 2}{n-4}} \Delta_{\mathbb{R}^{n+2k}}  ( U^{\frac{ 2}{n-4}}\Phi _{k,i} )   -\frac{ n }{2} U^{\frac{ 8}{n-4}  }    \Phi _{k,i} \Big] .
		\end{align*}
		Moreover, by elliptic regularity theory, $\Phi _{k,i}$ is smooth.
		Multiplying by $\Phi _{k,i} $ and integrating over $\mathbb{R}^{n+2k}$, we have
		\begin{align*}
			\int_{\mathbb{R}^{n+2k}}( \Delta_{\mathbb{R}^{n+2k}}\Phi _{k,i})^{2}=&\frac{n^{2}-4}{4}  \int_{\mathbb{R}^{n+2k}}  \left(|\nabla ( U^{\frac{ 2}{n-4}}\Phi _{k,i} )|^{2}   +\frac{ n }{2} U^{\frac{ 8}{n-4}  }    \Phi _{k,i}  ^{2}\right).
		\end{align*}
		Applying \eqref{key eigen ineq} in $\mathbb{ R}^{n+2k}$, we obtain
		\begin{align*}
			\int_{\mathbb{R}^{n+2k}}( \Delta_{\mathbb{R}^{n+2k}}\Phi _{k,i})^{2}\geq  \frac{ ((n+2k)^{2}-4 )(n+2k-4)}{4(n+2k)}  \int_{ \mathbb{R}^{n+2k } }  |\nabla( \frac{2}{1+ r^{2}}\Phi _{k,i})|^{2} +\frac{n+2k}{2}\frac{16}{(1+r^{2})^{4}} \Phi _{k,i}^{2}
		\end{align*}
		Consequently, we obtain
		\begin{align*}
			&\frac{ ((n+2k)^{2}-4 )(n+2k-4)}{4(n+2k)}  \int_{ \mathbb{R}^{n+2k } }  |\nabla( \frac{2}{1+ r^{2}}\Phi _{k,i})|^{2} +\frac{n+2k}{2}\frac{16}{(1+r^{2})^{4}} \Phi _{k,i}^{2}
			\\\leq& \frac{n^{2}-4}{4} \int_{\mathbb{R}^{n+2k}}  |\nabla ( \frac{2}{1+ r^{2}}\Phi _{k,i} )|^{2}   +\frac{ n }{2} \frac{16}{(1+r^{2})^{4}}  \Phi _{k,i}^{2}.
		\end{align*}
		Thus $\Phi _{k,i}\equiv 0$ for $k\geq 2$, which implies that the dimension of $V_{2}$ is at most $n+1$.
	\end{proof}

\end{document}